\documentclass[A4]{amsart}
\usepackage[square,sort,comma,numbers]{natbib}
\usepackage{amsmath, amssymb, amstext, amsfonts, textcomp, amsxtra, amsbsy, amsgen, amsopn, amscd, mathrsfs, amsthm, latexsym, array}
\usepackage[textwidth=16cm,textheight=22cm,centering]{geometry}

\usepackage{enumerate}

\usepackage{mathdots, bbm, dsfont}
\usepackage{color}
\usepackage{graphicx}

\usepackage{xcolor}
\definecolor{cite}{rgb}{0.00,0.60,1.00}
\definecolor{url}{rgb}{1.00,0.10,0.80}
\definecolor{link}{rgb}{0.00,0.00,1.00}
\usepackage[colorlinks,linkcolor=link,urlcolor=url,citecolor=cite,pagebackref,breaklinks]{hyperref}

\def\leq{\leqslant}

\def\geq{\geqslant}

\allowdisplaybreaks

\usepackage[all]{xy}

\newtheorem{theorem}             {Theorem}  [section]
\newtheorem{definition} [theorem] {Definition}
\newtheorem{lemma}      [theorem]{Lemma}
\newtheorem{corollary}  [theorem]{Corollary}
\newtheorem{proposition}[theorem]{Proposition}

\numberwithin{equation}{section} 

\theoremstyle{remark}
\newtheorem{remark}{\bf Remark}

\newcommand{\Cont}{{\rm C}}

\newcommand{\Sob}{{\rm S}}
\newcommand{\Sch}{\mathcal{S}}
\newcommand{\SSch}{\mathcal{S}_{\mathrm{sis}}}
\newcommand{\SMel}{\mathfrak{M}_{\mathrm{sis}}}
\newcommand{\sgn}{{\rm sgn}}
\newcommand{\intL}{{\rm L}}

\newcommand{\Nr}{{\rm Nr}}
\newcommand{\Tr}{{\rm Tr}}

\newcommand{\gp}[1]{\mathbf{#1}}
\newcommand{\GL}{{\rm GL}}
\newcommand{\GSp}{{\rm GSp}}

\newcommand{\SL}{{\rm SL}}

\newcommand{\nG}{\mathfrak{G}}
\newcommand{\diag}{{\rm diag}}

\newcommand{\ud}{\mathrm{d}}
\newcommand{\MF}{\mathcal{L}}

\newcommand{\Z}{\mathbb{Z}}

\newcommand{\Mat}{{\rm M}}
\newcommand{\id}{\mathbbm{1}}
\newcommand{\fF}{\mathbb{F}}

\makeatletter
\def\legendre@dash#1#2{\hb@xt@#1{%
  \kern-#2\p@
  \cleaders\hbox{\kern.5\p@
    \vrule\@height.2\p@\@depth.2\p@\@width\p@
    \kern.5\p@}\hfil
  \kern-#2\p@
  }}
\def\@legendre#1#2#3#4#5{\mathopen{}\left(
  \sbox\z@{$\genfrac{}{}{0pt}{#1}{#3#4}{#3#5}$}%
  \dimen@=\wd\z@
  \kern-\p@\vcenter{\box0}\kern-\dimen@\vcenter{\legendre@dash\dimen@{#2}}\kern-\p@
  \right)\mathclose{}}
\newcommand\legendre[2]{\mathchoice
  {\@legendre{0}{1}{}{#1}{#2}}
  {\@legendre{1}{.5}{\vphantom{1}}{#1}{#2}}
  {\@legendre{2}{0}{\vphantom{1}}{#1}{#2}}
  {\@legendre{3}{0}{\vphantom{1}}{#1}{#2}}
}
\def\dlegendre{\@legendre{0}{1}{}}
\def\tlegendre{\@legendre{1}{0.5}{\vphantom{1}}}
\makeatother

\newcommand{\R}{\mathbb{R}}
\newcommand{\C}{\mathbb{C}}
\newcommand{\E}{\mathbf{E}}
\newcommand{\F}{\mathbf{F}}

\newcommand{\A}{\mathbb{A}}

\newcommand{\vO}{\mathcal{O}}

\newcommand{\vP}{\mathcal{P}}
\newcommand{\VP}{\mathfrak{P}}

\newcommand{\oA}{\mathfrak{A}}

\newcommand{\oM}{\mathfrak{M}}

\newcommand{\ac}{{\rm ac}}

\newcommand{\norm}[1][\cdot]{\lvert #1 \rvert}
\newcommand{\extnorm}[1]{\left\lvert #1 \right\rvert}
\newcommand{\Norm}[1][\cdot]{\lVert #1 \rVert}
\newcommand{\extNorm}[1]{\left\lVert #1 \right\rVert}

\newcommand{\Pairing}[2]{\langle #1, #2 \rangle}
\newcommand{\extPairing}[2]{\left\langle #1, #2 \right\rangle}

\newcommand{\OFour}{\mathfrak{F}}
\newcommand{\invOFour}{\overline{\mathfrak{F}}}

\newcommand{\Mellin}[2][]{\mathfrak{M}_{#1}(#2)}
\newcommand{\RMellin}[2][]{\mathfrak{M}^{\mathrm{pv}}_{#1}(#2)}
\newcommand{\Rem}{\mathrm{r}}
\newcommand{\Ext}{\mathrm{e}}
\newcommand{\Inv}{\mathrm{i}}
\newcommand{\Mult}{\mathfrak{m}}
\newcommand{\Vor}{\mathcal{V}}
\newcommand{\VorH}{\mathcal{VH}}
\newcommand{\VHF}{\mathfrak{vh}}
\newcommand{\MS}{\mathcal{MS}}
\newcommand{\Trans}{\mathfrak{t}}

\newcommand{\rpL}{{\rm L}}
\newcommand{\rpR}{{\rm R}}

\newcommand{\Ind}{{\rm Ind}}
\newcommand{\cInd}{{\rm c}\scalebox{0.5}[1.0]{\( - \)}{\rm Ind}}
\newcommand{\QC}{\mathfrak{X}}

\newcommand{\Whi}{\mathcal{W}}

\newcommand{\Cond}{\mathbf{C}}
\newcommand{\cond}{\mathfrak{c}}

\newcommand{\Zeta}{\mathrm{Z}}

\newcommand{\RamCst}{\vartheta}

\newcommand{\SuS}{\mathit{Ss}}

\newcommand{\AutR}{\mathcal{A}}

\newcommand{\Sph}{\mathbf{S}}

\newcommand{\Vol}{{\rm Vol}}
\makeatletter
\newcommand{\rmnum}[1]{\romannumeral #1}
\newcommand{\Rmnum}[1]{\expandafter\@slowromancap\romannumeral #1@}
\makeatother

\newcommand{\Ex}{\mathcal{E}{\rm x}}

\title{Voronoi--Hankel Transforms}

\author{Zhefeng Shen}
\address{Mathematisches Institut, Universit\"at Bonn, Endenicher Allee 60, 53115 Bonn, Germany}
\email{shen\_zhefeng@mail.ustc.edu.cn}

\author{Han Wu \smallskip\\
\MakeLowercase{with an appendix by} Dihua Jiang, Zhilin Luo, Zhefeng Shen \& Han Wu}
\address{Hangzhou International Innovation Institute, Beihang University, Hangzhou 311115, P.R.China}
\email{forrestwu@buaa.edu.cn}

\date{\today}

\begin{document}



\begin{abstract}
	Let $\pi$ be a generic irreducible representation of $\GL_n(\F)$ for a local field $\F$. We introduce the Voronoi--Hankel transform $\VorH_{\pi}$ for $\pi$, and prove it is equivalent to the $\pi$-Fourier transform introduced by Jiang--Luo. We give an extension $\widetilde{\VorH}_{\pi}$ of $\VorH_{\pi}$ to some larger class of functions on $\F^{\times}$, and prove the multiplicativity with respect to parabolic induction via the above-mentioned equivalence. For non-archimedean $\F$ we give a systematical study of the relevant kernel function, including the asymptotic behavior at $0$ and $\infty$, as well as some simple integral representation based on the local Langlands correspondences for the essentially tame supercuspidals. As an application in the non-archimedean case, we give an effective version of the stability theorem for the twisted local gamma factors.
\end{abstract}

	\maketitle
	
	\tableofcontents

\section{Introduction}

	\subsection{Results over General Local Fields}

	Let $\F$ be a local field with modular character $\norm_{\F}$ and write $\Omega_{\F} = \left\{ t \in \F \ \middle| \ \norm[t]_{\F} = 1 \right\}$. Fixe a non-trivial additive character $\psi$. Let $\pi$ be a generic irreducible representation of $\GL_n(\F)$. The goal of the paper is to give a systematical study of some (integral) transform determined by the twisted gamma factors of $\pi$ by (unitary) characters of $\F^{\times}$. Precisely, as $\chi$ varies over the set $\widehat{\F^{\times}}$ of unitary characters of $\F^{\times}$, two sets of \emph{twisted} gamma factors and associated transforms on some subspaces of smooth functions on $\F^{\times}$ have been introduced in the literature.
	
	One set is $\left\{ \gamma(s, \pi \times \chi, \psi) \ \middle| \ \chi \in \widehat{\F^{\times}} \right\}$, fitting in the Rankin--Selberg theory mainly developped by Jacquet \& Piatetski-Shapiro \& Shalika \cite{JPS79, JPS83, JS90, J09}. The relevant space of functions is $\VorH(\pi)$ given in Definition \ref{def: VoronoiGLn} below. Let $\widetilde{\pi}$ be the contragredient representation of $\pi$. The associated transform is the \emph{Voronoi--Hankel transform} $\VorH_{\pi}: \VorH(\pi) \to \norm_{\F} \cdot \VorH(\widetilde{\pi})$ determined by the functional equation
\begin{equation} \label{eq: VorHTviaMellin}
	\int_{\F^{\times}} \VorH_{\pi}(h)(t) \chi^{-1}(t) \norm[t]_{\F}^{-s} \ud^{\times} t = \gamma(s, \pi \times \chi, \psi) \cdot \int_{\F^{\times}} h(t) \chi(t) \norm[t]_{\F}^s \ud^{\times} t. 
\end{equation}
	Both $\VorH(\pi)$ and $\VorH_{\pi}$ were extensively studied by many others such as Miller--Schmid \cite{MS04} and Qi \cite{Qi20} for archimedean $\F$. In fact, the transform $\VorH_{\pi}$ is the local part of a global one, which links the test functions on the two sides of the \emph{Voronoi summation formula} by Miller--Schmid \cite{MS06}. This summation formula has numerous applications in analytic number theory.
	
	The other set is $\left\{ \gamma(s, \pi \otimes \chi, \psi) \ \middle| \ \chi \in \widehat{\F^{\times}} \right\}$, fitting in the Godement--Jacquet theory mainly developed in \cite{GoJa72}. The relevant space of functions is $\Sch_{\pi}(\F^{\times})$. The associated transform is the \emph{$\pi$-Fourier transform} $\OFour_{\pi}: \Sch_{\pi}(\F^{\times}) \to \Sch_{\widetilde{\pi}}(\F^{\times})$ determined by the functional equation
\begin{equation} \label{eq: PiFourTviaMellin}
	\int_{\F^{\times}} \OFour_{\pi}(h)(t) \chi^{-1}(t) \norm[t]_{\F}^{1/2-s} \ud^{\times} t = \gamma(s, \pi \otimes \chi, \psi) \cdot \int_{\F^{\times}} h(t) \chi(t) \norm[t]_{\F}^{s-1/2} \ud^{\times} t. 
\end{equation}
	Both $\Sch_{\pi}(\F^{\times})$ and $\OFour_{\pi}$ are named by Jiang--Luo in \cite[\S 2]{JiL22}. Their motivation is a type of summation formula attached to any pair $(\pi, \rho)$, where $\pi$ is an irreducible cuspidal representation of $\gp{G}(\A)$ for a split reductive group $\gp{G}$, and $\rho$ is a finite dimensional complex representation of the dual group $\gp{G}^{\vee}(\C)$. When $\gp{G} = \GL_1$ and $\rho$ is the standard representation, this type summation formula coincides with the usual Poisson summation formula. In general such summation formula is hoped by Braverman--Kazhdan to prove the Langlands conjecture on the meromorphic continuation and global functional equation for the automorphic $L$-function $L(s, \pi, \rho)$.
	
	The local gamma factors in \eqref{eq: VorHTviaMellin} and \eqref{eq: PiFourTviaMellin} are equal, namely 
\begin{equation} \label{eq: EquivGamma}
	\gamma(s, \pi \times \chi, \psi) = \gamma(s, \pi \otimes \chi, \psi),
\end{equation} 
	by the compatibility of the Rankin--Selberg theory and the Godement--Jacquet theory proved in \cite{JPS79, JS90, J09}. Therefore it is natural to ask if their associated transforms are the same. Our first result shows that it indeed is the case, hence that the Voronoi summation formula for $\pi$ of $\GL_n(\A)$ is the summation formula attached to $(\pi, \rho)$ for the standard representation $\rho = \id: \GL_n(\C) \to \GL_n(\C)$. 
\begin{theorem} \label{thm: EquivDefs}
	(1) The two spaces $\VorH(\pi)$ and $\Sch_{\pi}(\F^{\times})$ are related by $\Sch_{\pi}(\F^{\times}) = \norm_{\F}^{1/2} \cdot \VorH(\pi)$.
	
\noindent (2) The two transforms $\VorH_{\pi}$ and $\OFour_{\pi}$ are related by $\OFour_{\pi}(f)(t) = \norm[t]_{\F}^{-1/2} \cdot \VorH_{\pi}(g)(t)$ with $g(t) := \norm[t]_{\F}^{-1/2} \cdot f(t)$.
\end{theorem}

\noindent After giving, resp. recalling the precise definition of $\VorH(\pi)$, resp. $\Sch_{\pi}(\F^{\times})$ and introducing some auxiliary spaces in \S \ref{sec: FucSps}, the proof of Theorem \ref{thm: EquivDefs} takes the subsequent \S \ref{sec: EquivDefsNA} \& \S \ref{sec: EquivDefsA}. Theorem \ref{thm: EquivDefs} (1) was also proved by Jiang--Li as part of \cite[Theorem 1.3]{JLi25} with a \emph{different} method. However, the proof is based on an earlier work \cite{JiL22} of Jiang--Luo, which requires some non-trivial clarification. Appendix \ref{sec: EqViaPW} fills in the relevant gaps with simplification. It also contains a concrete description of $\VorH(\pi)$ in Corollary \ref{cor: ConSp}.

\begin{remark}
	Theorem \ref{thm: EquivDefs} (1) is \emph{not} a consequence of \eqref{eq: EquivGamma}, although the proof we will give is based on the proof of \eqref{eq: EquivGamma} scattered in \cite{JPS79, JS90, J09}. In fact, the intermediate arguments for Theorem \ref{thm: EquivDefs} (1) are always slight \emph{refinements} of those for \eqref{eq: EquivGamma}. Theorem \ref{thm: EquivDefs} (1) and \eqref{eq: EquivGamma} readily imply Theorem \ref{thm: EquivDefs} (2).
\end{remark}

\begin{remark}
	In the archimedean case, the proof of Theorem \ref{thm: EquivDefs} (1) uses (more) ``difficult results on smooth representations'' due to Wallach \cite{Wal92}, in the spirit of Jacquet \cite[\S 1]{J09}. See Lemma \ref{lem: DualByFunctGL} for details.
\end{remark}

	The next result concerns the special case where $\pi$ is assumed to be \emph{tempered (unitary)}. In this case, we shall give an extension $\widetilde{\VorH}_{\pi}$ of $\VorH_{\pi}$ is in \S \ref{sec: ExtVorH}, and establish the \emph{multiplicativity} property, which reflects (but does not follow from) the multiplicativity of the local gamma factors for the parabolic induction.
	
\begin{theorem} \label{thm: Mult}
	Suppose $\pi = \pi_1 \boxplus \pi_2$ is parabolically induced from some tempered representations $\pi_j$ of $\GL_{n_j}(\F)$ with $n_1+n_2=n$. Then on the space $\VorH(\pi)$ the following relation holds
	$$ \VorH_{\pi} = \widetilde{\VorH}_{\pi_1} \circ \Inv \circ \widetilde{\VorH}_{\pi_2}. $$
\end{theorem}

\noindent The proof of Theorem \ref{thm: Mult} will be given in \S \ref{sec: Mult}.

\begin{remark}
	Although the statement of Theorem \ref{thm: Mult} is entirely within the framework of Voronoi--Hankel transforms, its proof that we give here relies on Theorem \ref{thm: EquivDefs} in an essential way. This is due to the lack of a good generalization of the construction of Whittaker functions by \emph{Godement sections} to general parabolic induction from, say $n_1=1$ to $n_1 \geq 2$.
\end{remark}

\begin{remark} 
	The extension $\widetilde{\VorH}_{\pi}$ applies to test functions with \emph{non-simple singularities}. For example, for any $\chi \in \widehat{\F^{\times}}$ the transform $\widetilde{\VorH}_{\chi}$ is essentially the distributional Fourier transform. For any non-trivial additive character $\psi: \F \to \Sph^1$ the function $H(x) := \chi(x) \psi(ax^2)$ lies in $\Cont(\F^{\times})$ for any $\chi \in \widehat{\F^{\times}}$. By Weil's classical work \cite{We64} we have
	$$ \widetilde{\VorH}_{\chi}(H)(y) = H^*(y) = \tfrac{\gamma_{\F}(a,\psi)}{\sqrt{2\norm[a]_{\F}}} \chi^{-1}(x) \norm[x]_{\F} \psi(x^2/(4a)), $$
where $\gamma_{\F}(a,\psi) \in \Sph^1$ is the \emph{Weil index}.  Obviously the function $H \notin \Inv(\SSch(\F)) + \chi \cdot \Sch(\F)$, and admits no properly defined Mellin transform, hence is beyond the applicability of Miller--Schmid's theory. We will apply $\widetilde{\VorH}_{\pi}$ to some analytic number theoretic problem related to a generalized Motohashi formula \cite{Wu24+}.
\end{remark}

	\subsection{Results over non-Archimedean Fields}
	
	Let $\F$ be a local non-archimedean field. Let $\pi$ be an irreducible, admissible and generic (not necessarily unitary) representation of $\GL_d(\F)$ for some $d \in \Z_{\geq 2}$. We shall prove that $\VorH_{\pi}$ is of convolution type in the following sense, and give some basic properties of the associated kernel function $\VHF_{\pi}$.

\begin{definition} \label{def: ConvTypeTrans}
	We say a transform $A$ on the space of functions on $\F^{\times}$ is of \emph{convolution type} if its domain contains $\Cont_c^{\infty}(\F^{\times})$ and if there is a locally integrable function $a$ on $\F^{\times}$ so that
	$$ A(h)(y) = \int_{\F^{\times}} a(xy) h(x) \ud^{\times} x = a*(\Inv(h))(y), \quad \forall \ h \in \Cont_c^{\infty}(\F^{\times}). $$
	We call $a(y)$ the \emph{convolution kernel} of $A$.
\end{definition}
	
\noindent In other words, we shall prove that $\VorH_{\pi}$ is represented by a kernel function $\VHF_{\pi}$ via convolution as
\begin{equation} \label{eq: VorHTviaConv}
	\VorH_{\pi}(h)(y) = \VHF_{\pi} * \Inv(h)(y) = \int_{\F^{\times}} \VHF_{\pi}(xy) h(x) \ud^{\times} x, \quad \forall \ h \in \Cont_c^{\infty}(\F^{\times}).
\end{equation}

\noindent We shall actually do something more. Namely we will find two convolution algebras $(\SSch(\F), *_{\mathrm{pv}}) \subset (\SSch(\F, \widetilde{\F^{\times}}), *_{\mathrm{pv}})$ in \S \ref{sec: ExtVorH} which extend $(\Cont_c^{\infty}(\F^{\times}),*)$, the latter being an \emph{ideal} of the former. It turns out that $\SSch(\F, \widetilde{\F^{\times}})$ contains all kernel functions $\VHF_{\pi}$. Hence the convolutional interpretation by \eqref{eq: VorHTviaConv} applies to $h \in \Inv(\SSch(\F, \widetilde{\F^{\times}}))$ and provides an extension of $\MS_{\pi}$. The whole \S \ref{sec: FANA} is devoted to the theory of $(\SSch(\F, \widetilde{\F^{\times}}), *_{\mathrm{pv}})$, with main results summarized in Proposition \ref{prop: SSchAlg} or \ref{prop: SSchAlgBis}. The main results of this part is stated as follows.
	
\begin{definition} \label{def: l0Ss}
	Let $\pi$ be a supercuspidal representation of $\GL_d(\F)$. If $d = 1$ we write $l_0(\pi) := \max(2\cond(\pi),2)$. If $d \geq 2$ we define $l_0(\pi) := \lceil (2\cond(\pi)+f(\pi))/d \rceil - 1$, where $f(\pi)$ is the number of unramified $\chi \in \widehat{\F^{\times}}$ satisfying $\pi \otimes \chi \simeq \pi$.
\end{definition}
	
\begin{theorem} \label{thm: VHKernNA}
	Let $\F$ be a local non-archimedean field. Let $\pi$ be a generic, admissible and irreducible representation of $\GL_d(\F)$. Let $\SuS(\pi) = \{ \pi_1, \dots, \pi_k \}$ be the \emph{supercuspidal support} of $\pi$.

\noindent (1) The transform $\VorH_{\pi}$ is represented by a kernel function $\VHF_{\pi} \in \SSch(\F, \widetilde{\F^{\times}})$ and we have
	$$ \VHF_{\pi} = \VHF_{\pi_1} *_{\mathrm{pv}} \cdots *_{\mathrm{pv}} \VHF_{\pi_k}. $$
	Consequently we have for $\Re(s) \gg 1$ the following equality
	$$ \int_{\F^{\times}}^{\mathrm{pv}} \VHF_{\pi}(t) \chi^{-1}(t) \norm[t]_{\F}^{-s} \ud^{\times} t = \gamma(s, \pi \times \chi, \psi). $$
	
\noindent (2) Let $l_0(\pi) := \max \left\{ l_0(\pi_i) \ \middle| \ 1 \leq i \leq k \right\}$. Let $\omega$ be the central character of $\pi$. If $v_{\F}(t) \leq -dl_0$ and $\VHF_{\pi}(t) \neq 0$, then we have $v_{\F}(t) = -dl$ for some $l \geq l_0$ and
	$$ \VHF_{\pi}(t) = \norm[t]_{\F} \cdot G_l(t; \omega^{-1}, d) $$
depends only on $\omega$, where $G_l(t; \omega, d)$ is given in Definition \ref{def: GermParGI} below.
\end{theorem}

\begin{remark}
	The definition of the \emph{supercuspidal support} of $\pi$ will be recalled in the beginning of the proof of Theorem \ref{thm: VHKernNA} (1), given in \S \ref{sec: VHKernNA}.
\end{remark}

\begin{remark}
	The core of the proof of Theorem \ref{thm: VHKernNA} (2) is the case of a single supercuspidal representation, i.e., Proposition \ref{prop: VHKernSc}. Our main discovery is that the theory of \emph{cuspidal inducing datum} (recalled in \S \ref{sec: CID}) is strong enough to determine the asymptotic behavior of $\VHF_{\pi}$. A crucial technical innovation is Proposition \ref{prop: POrderNbhdQ}, which puts some intuition gained from the Lie algebra computation at archimedean places into effect for \emph{principal orders}, the building blocks for supercuspidal representations. Its advantage over the naive analogue of the exponential map from the Lie algebra is its independence of the characteristic of $\F$. In particular, Theorem \ref{thm: VHKernNA} is valid for local function fields as well.
\end{remark}

\begin{remark}
	Readers familiar with the work of Jiang--Luo \cite{JiL22} may find some ``shortcuts'' to certain results give here, based on the theirs. For example, the integral representation of $\VHF_{\pi}$ for supercuspidal $\pi$ given in Proposition \ref{prop: VHKernSc} (1) can be easily deduced from their general integral representation of the $\pi$-Fourier kernel function \cite[(3-15)]{JiL22}, together with the consideration of matrix coefficients made from ``minimal vectors'' in the sense of Hu--Nelson--Saha \cite{HNS18}. It would be worth emphasizing that our treatment is self-contained, and in particular independent of these cited works.
\end{remark}

\begin{remark}
	It should be mentioned that the space $\SSch(\F, \widetilde{\F^{\times}})$ does not cover all interesting applications of our extended Voronoi--Hankel transform $\widetilde{\VorH}_{\pi}$. We intend to study special functions of the shape $\widetilde{\VorH}_{\pi_1}(\VHF_{\pi_2})$ in future, which is different from $\widetilde{\VorH}_{\pi_1}(\Inv(\VHF_{\pi_2})) = \VHF_{\pi_1} *_{\mathrm{pv}} \VHF_{\pi_2}$. This has applications to the analysis of the local weight transforms in some Motohashi-type formulas similar to \cite[(15.3) \& (15.6)]{BrM03}.
\end{remark}

	Further simplification of the integral representation of $\VHF_{\pi}(t)$ for supercuspidal $\pi$ given in Proposition \ref{prop: VHKernSc} (1) for small $\norm[t]_{\F}$ turns out to be subtler. Instead of the theory of cuspidal inducing datum, it seems necessary to involve some fine aspects of the \emph{local Langlands correspondences} (abbreviated as \emph{LLC}). We are able to give such simplification only for the \emph{essentially tame} supercuspidal representations, for which the corresponding \emph{LLC} are sufficiently explicit thanks to the works of Bushnell--Henniart \cite{BuH05_T1, BuH05_T2, BuH10_T3}. The relevant result is stated in Proposition \ref{prop: VHKernET}. It looks challenging to find analogous simple integral representation of $\VHF_{\pi}$ for general supercuspidal $\pi$, namely integral representation in terms of the \emph{Langlands parameter} of $\pi$.

\begin{remark}
	It would be interesting to find a direct computational proof of Proposition \ref{prop: VHKernET} without appealing to the works of Bushnell--Henniart. In this case, an analytic formula of the Langlands constants (``$\lambda(\E/\F, \psi)$'' in Proposition \ref{prop: VHKernET}) would follow, similar to the Weil's index for quadratic $\E/\F$.
\end{remark}

\begin{remark}
	In \cite{BS11} the authors studied $\VHF_{\pi}$ for supercuspidal $\pi$ of $\GL_2(\F)$. Their proof, based on the Weil's representation, makes essential use of the exceptional isomorphism $\GL_2 \simeq \GSp_2$. Although that approach can be extened to the archimedean fields, it is not generalizable to $\GL_d(\F)$ for $d \geq 3$.
\end{remark}

	\subsection{Application to Stability of Gamma Factors}
	
	We continue with the non-archimedean setting in the previous subsection. The stability of gamma factors is the phenomenon that $\gamma(s, \pi \times \chi, \psi)$ depends only the central character $\omega$ of $\pi$ (or becomes ``stable'') if the conductor exponent $\cond(\chi)$ is sufficiently large. This was initially studied by Jacquet--Shalika \cite{JS85}. For representations $\pi$ of groups different from $\GL_d$, analogues can be found in \cite{CPS98, Ki00, RS05, Br08}. The stability theorems have important applications to establish functorial lifts from other groups to $\GL_d$.  
	
	It should be mentioned that although the stability phenomenon is inspired by the stability of the Artin gamma factors, the strength of the above mentioned results is not superseded even if we knew the full \emph{LLC}. In fact the stability of the Artin gamma factors involves the Brauer induction theorem, which is not effective. Hence we have no control on the implied constant $l_0$, called the \emph{stability barrier}, in $\cond(\chi) \geq l_0$ for which $\gamma(s, \pi \times \chi, \psi)$ becomes stable. Whereas in the above-mentioned stability on the analytic side, $l_0$ can be quantified as some large multiple of the conductor exponent $\cond(\pi)$ of $\pi$, although it has never been explicated.
	
	In the literature, we find only Bushnell--Henniart \cite[25.7]{BuH06} in the case of $\GL_2$, where the stability barrier is explicitly given in terms of some invariants of $\pi$. Bushnell--Henniart's method involves the subtler theory of \emph{cuspidal types}. Our Theorem \ref{thm: VHKernNA} implies an effective version of the stability theorem for $\GL_d$, which can be viewed as both a generalization and simplification of the theirs.
	
\begin{corollary} \label{cor: StRangeGen}
	Let notation be as in Theorem \ref{thm: VHKernNA}. For any character $\chi$ with $\cond(\chi) \geq l_0(\pi)$ we have 
	$$ \gamma(s, \pi \times \chi, \psi) = \gamma(s, \chi, \psi)^{d-1} \gamma(s, \chi \omega, \psi), \quad \cond(\pi \times \chi) = d \cdot \cond(\chi). $$
\end{corollary}

\noindent Two (simple) proofs of Corollary \ref{cor: StRangeGen} will be also given in \S \ref{sec: VHKernNA}.

\section{Preliminary}

	\subsection{Notation and Convention}
	
	For any measured space $(\Omega, \ud x)$ and a subset $A \subset \Omega$ with finite volume, we write for any measurable function $f: \Omega \to \C$
	$$ \oint_A f(x) \ud x := \Vol(A, \ud x)^{-1} \int_A f(x) \ud x. $$
	
	Throughout the paper $\F$ denotes a local field with module $\norm_{\F}$. The symbol $\norm$ is reserved for the usual absolute value on $\R, \C$. A non-trivial additive character $\psi$ of $\F$ is chosen and fixed. The Haar measure $\ud x$ on $\F$ is the one self-dual with respect to $\psi$. For $d \in \Z_{\geq 1}$ we put the $d^2$-th tensor product $\ud X$ of $\ud x$ on $\Mat_d(\F)$, the algebra of $d \times d$ matrices over $\F$. Let $\ud g = \norm[\det X]_{\F}^{-d} \ud X$ be the chosen Haar measure on $\GL_d(\F)$. We denote by $\gp{B}_d$ the standard Borel subgroup of $\GL_d$ consisting of upper triangular matrices. The unipotent radical of $\gp{B}_d$ is denoted by $\gp{N}_d$, consisting of upper triangular matrices with diagonal entries equal to $1$.
	
	Let $\pi$ be a generic, admissible and irreducible representation of $\GL_d(\F)$. The local gamma factor has the following standard decomposition
	$$ \gamma(s, \pi, \psi) = \frac{L(1-s, \widetilde{\pi})}{L(s, \pi)} \varepsilon(s, \pi, \psi), \quad \varepsilon(s, \pi, \psi) = \varepsilon(1/2, \pi, \psi) \cdot \left( \Cond(\pi)\Cond(\psi) \right)^{1/2-s} $$
where $L(s, \pi)$ is the standard local $L$-factor of $\pi$, $\widetilde{\pi}$ is the contragredient representation of $\pi$, $\varepsilon(s, \pi, \psi)$ is the local epsilon-factor, $\Cond(\pi)$ (resp. $\Cond(\psi)$) is the conductor of $\pi$ (resp. of $\psi$) and $\varepsilon(1/2, \pi, \psi) \in \Sph^1 = \left\{ z \in \C^{\times} \ \middle| \ \norm[z]=1 \right\}$ is called the local root number.

	If $\F$ is a local non-archimedean field, write $\vO_{\F}$ (resp. $\vP_{\F}$, resp. $\varpi_{\F}$) for the valuation ring (resp. ideal of $\F$, resp. a uniformizer of $\F$). Write $q_{\F} = \extnorm{\vO_{\F}/\vP_{\F}}$ for the residual cardinality. Assume $\psi$ is trivial on $\vO_{\F}$ but non-trivial on $\vP_{\F}^{-1}$. Let $v_{\F}$ be the normalized additive valuation. For $m \in \Z_{\geq 0}$ let $U_{\F}^m := (1+\vP_{\F}^m) \cap \vO_{\F}^{\times}$. The collection $\left\{ \vP_{\F}^n \ \middle| \ n \in \Z_{\geq 0} \right\}$, resp. $\left\{ U_{\F}^m \ \middle| \ m \in \Z_{\geq 0} \right\}$ forms a system of neighborhoods of $0$, resp. $1$, making $\F$, resp. $\F^{\times}$ an example of \emph{totally disconnected topological group}. In general, let $G$ be a such group (multiplication written as $x \cdot y$) with a collection $\mathfrak{U} = \left\{ U_i \ \middle| \ i \in I \right\}$ of \emph{directed} system of open normal subgroup-neighborhoods at the identity element $\id \in G$. A function $f: G \to \C$ is \emph{locally constant} or \emph{continuous}, if
\begin{equation} \label{eq: ContFDef}
	\forall \ x \in G, \ \exists \ U \in \mathfrak{U}, \textrm{ s.t. } \forall \ y \in U, \ f(x \cdot y) = f(x).
\end{equation}
\begin{definition} \label{def: SmoothFDef}
	A \emph{continuous} function $f: G \to \C$ is \emph{smooth}, if in \eqref{eq: ContFDef} some $U \in \mathfrak{U}$ can be chosen uniformly for all $x \in G$. The space of smooth (resp. smooth with compact support) complex-valued functions on $G$ is denoted by $\Cont^{\infty}(G)$ (resp. $\Cont_c^{\infty}(G)$). We also write $\Sch(G) = \Cont_c^{\infty}(G)$.
\end{definition}

\noindent We thus have $\Sch(\Mat_d(\F))$ for any local field $\F$. For $\Phi \in \Sch(\Mat_d(\F))$ (hence any subspace of $\Sch(\Mat_d(\F))$) and $g_1,g_2 \in \GL_d(\F)$, we write $g_1.\Phi.g_2(X) := \Phi(g_2^{-1}Xg_1)$.

	\subsection{Voronoi--Hankel Spaces and $\pi$-Fourier Spaces}
	\label{sec: FucSps}
	
	Let $n \in \Z_{\geq 2}$. Let $\pi$ be a unitary irreducible representation of $\GL_n(\F)$ which is generic and $\RamCst$-tempered for some constant $0 \leq \RamCst < 1/2$. Let $W \in \Whi(\pi^{\infty},\psi)$ be a function in the Whittaker model of $\pi^{\infty}$. Let $w_n$ be the longest Weyl element of $\GL_n$. Then the function
	$$ \widetilde{W}(h) := W(w_n h^{\iota}), \quad \forall \ h \in \GL_n(\F) $$
is in $\Whi(\widetilde{\pi}^{\infty}, \psi^{-1})$, the Whittaker model of the smooth contragredient representation $\widetilde{\pi}^{\infty}$. We introduce some elementary operators on the space of functions on $\F^{\times}$:
\begin{itemize}
	\item For functions $\phi$ on $\F^{\times}$, its \emph{extension} by $0$ to $\F$ is denoted by $\Ext(\phi)$, and its \emph{inverse} is $\mathrm{Inv}(\phi)(t) := \phi(t^{-1})$; for functions $\phi$ on $\F$, its \emph{restriction} to $\F^{\times}$ is denoted by $\Rem(\phi)$, and the operator $\Inv$ is 
	$$ \Inv = \Ext \circ \mathrm{Inv} \circ \Rem. $$
	\item For $s \in \C$, $\mu \in \widehat{\F^{\times}}$ and functions $\phi$ on $\F$, we introduce the operator $\Mult_s(\mu)$ by
	$$ \Mult_s(\mu)(\phi)(t) = \phi(t) \mu(t) \norm[t]_{\F}^s. $$
	\item For $\delta \in \F^{\times}$ we introduce the operator $\Trans(\delta)$ by
	$$ \Trans(\delta)(\phi)(y) = \phi(y \delta). $$
\end{itemize}

\begin{definition} \label{def: VoronoiGLn}
	(0) If $n \geq 2$ and $0 \leq j \leq n-2$, the space of functions on $\F^{\times}$ (all containing $\Cont_c^{\infty}(\F^{\times})$)
\begin{equation} \label{eq: BasicVHS}
	\VorH(\pi) = \VorH(\pi, j) := \left\{ h(y) := \norm[y]_{\F}^{-\frac{n-1}{2}} \int_{\F^j} W \begin{pmatrix} y & & \\ \vec{x} & \id_j & \\ & & \id_{n-1-j} \end{pmatrix} \ud \vec{x} \ \middle| \ W \in \Whi(\pi^{\infty}, \psi) \right\}
\end{equation}
	is independent of $j$ or $\psi$ (as long as $\psi$ is non-trivial).

\noindent (1) Let $n \in \Z_{\geq 2}$. The transform from the function $H$ to the function $H^*$ defined by
	$$ H(y) = \int_{\F^{n-2}} W \begin{pmatrix} y & & \\ \vec{x} & \id_{n-2} & \\ & & 1 \end{pmatrix} \ud \vec{x}, \quad H^*(y) = \widetilde{W} \begin{pmatrix} y & \\ & w_{n-1} \end{pmatrix} $$
	is the \emph{Voronoi transform} for $\pi$, written as $\Vor_{\pi}$, namely $H^*(y) := \Vor_{\pi}(H)(y)$. We also introduce the \emph{Voronoi--Hankel transform} as
	$$ \VorH_{\pi} := \Trans((-1)^{n-1}) \circ \Mult_{-\frac{n-3}{2}} \circ \Vor_{\pi} \circ \Mult_{\frac{n-1}{2}}: \ \VorH(\pi) \to \norm \VorH(\widetilde{\pi}) $$
	so that the local functional equation can be written as (independently of the rank $n$)
\begin{equation} \label{eq: LocFEGLnGL1}
	\int_{\F^{\times}} \VorH_{\pi}(H')(t) \chi^{-1}(t) \norm[t]_{\F}^{-s} \ud^{\times} t = \gamma(s, \pi \times \chi, \psi) \int_{\F^{\times}} H'(t) \chi(t) \norm[t]_{\F}^s \ud^{\times}t, 
\end{equation}
	where we have written $H'(t) := H(t) \norm[t]_{\F}^{-\frac{n-1}{2}}$.
	
\noindent (2) Let $\Sch(\F)$ be the space of Schwartz--Bruhat functions. Let $\chi$ be a (quasi-character) of $\F^{\times}$. The Voronoi--Hankel transform $\VorH_{\chi}$ is the composition $\VorH_{\chi} = \Mult_1(\chi^{-1}) \circ \invOFour \circ \Mult_0(\chi^{-1})$ on
\begin{equation} \label{eq: LocFEGL1}
	\VorH_{\chi}: \chi \cdot \Sch(\F) \to \chi^{-1} \norm \cdot \Sch(\F). 
\end{equation}
	In this case we put $\VorH(\chi) := \chi \cdot \Sch(\F)$.
\end{definition}

\noindent Note that in the case $n=1$ we can rewrite the transform $\VorH_{\chi}$ in \eqref{eq: LocFEGL1} as
\begin{equation} \label{eq: LocFEGL1Bis}
	\Mult_{-1}(\chi) \circ \VorH_{\chi} = \invOFour \circ \Mult_0(\chi^{\iota}). 
\end{equation}

\noindent We extend the Voronoi--Hankel transform $\VorH_{\pi}$ to $\widetilde{\VorH}_{\pi}$ in terms of an analogue of \eqref{eq: LocFEGL1Bis}. This extension has its own interests: it shows that higher rank Voronoi--Hankel transforms are essentially Fourier transforms in higher dimensional affine spaces. In passing we complement Miller--Schmid's theory for non-archimedean local fields. In particular, we give in \S \ref{sec: MSTheory} an intermediate extension $\MS_{\pi}$ of $\VorH_{\pi}$ on $\Inv(\SSch(\F))$, the space of functions $\Inv(h)(t) := h(t^{-1})$ for $h \in \SSch(\F)$ (see Definition \ref{def: SSchNA}), and verify its consistency with $\widetilde{\VorH}_{\pi}$. 
\begin{theorem} \label{thm: ExtVorH}
	(1) Let $n \in \Z_{\geq 1}$. Let $\pi$ be an irreducible and generic representation of $\GL_n(\F)$. Let $\invOFour$ be the distributional inverse Fourier transform on $\Mat_n(\F)$. Let $I_n: \Cont(\F^{\times}) \to \Cont(\GL_n(\F))$ be given by $I_n(h)(g) := h(\det g)$. For any \emph{smooth matrix coefficient} $\beta$ of $\pi$ we have
	$$ \Mult_{-\frac{n+1}{2}}(\beta) \circ I_n \circ \MS_{\pi} = \invOFour \circ \Mult_{-\frac{n-1}{2}}(\beta^{\iota}) \circ I_n \mid_{\Inv(\SSch(\F))}. $$
	
\noindent (2) If in addition $\pi$ is \emph{tempered unitary}, then we have the equality on $\VorH(\pi)$
	$$ \Mult_{-\frac{n+1}{2}}(\beta) \circ I_n \circ \VorH_{\pi} = \invOFour \circ \Mult_{-\frac{n-1}{2}}(\beta^{\iota}) \circ I_n. $$
	Moreover, the following equation for a pair of functions $H, H^* \in \Cont(\F^{\times})$
	$$ \Mult_{-\frac{n+1}{2}}(\beta) \circ I_n(H^*) = \invOFour \circ \Mult_{-\frac{n-1}{2}}(\beta^{\iota}) \circ I_n(H), $$
	uniquely determines $\widetilde{\VorH}_{\pi}(H):=H^*$ as $\beta$ traverses the set $C(\pi^{\infty})$ of smooth matrix coefficients of $\pi$.
\end{theorem}

%

	The $\pi$-Schwartz space is introduced by Jiang--Luo \cite[(2.15)]{JiL23} as
\begin{equation} \label{eq: piFourSp}
	\Sch_{\pi}(\F^{\times}) := \mathrm{Span} \left\{ \norm[x]_{\F}^{\frac{n}{2}} \int_{\SL_n(\F)} (\Phi \cdot \beta) \left( \begin{pmatrix} x & \\ & \id_{n-1} \end{pmatrix} g \right) \ud g \ \middle| \ \Phi \in \Sch(\Mat_n(\F)), \beta \in C(\pi^{\infty}) \right\}.
\end{equation} 

\noindent We introduce the following auxiliary spaces of functions on $\F^{\times}$
\begin{align}
	S(\pi) &:= \norm_{\F}^{-\frac{1}{2}} \cdot \Sch_{\pi}(\F^{\times}), \label{eq: piFourSpVar} \\
	W(\pi) &:= \mathrm{Span} \left\{ \norm[x]_{\F}^{\frac{n-1}{2}} \int_{\SL_n(\F)} (\Phi \cdot W) \left( \begin{pmatrix} x & \\ & \id_{n-1} \end{pmatrix} g \right) \ud g \ \middle| \ \Phi \in \Sch(\Mat_n(\F)), W \in \Whi(\pi^{\infty}, \psi) \right\}. \label{eq: AuxSchWhiSp}
\end{align}
	Theorem \ref{thm: EquivDefs} (1) will follow from $S(\pi) = W(\pi) = \VorH(\pi,0) = \VorH(\pi,j)$. 
\begin{remark}
	The proof of Theorem \ref{thm: EquivDefs} (1) turns out to be technically different between the archimedean and non-archimedean cases. One major obstruction for a unified proof is the fact that $\F^{\times}$ modulo its subgroup of $n$-th perfect powers can be infinite for a local function field.
\end{remark}

\section{Equivalence of Definitions: Non-archimedean Case}
\label{sec: EquivDefsNA}

	\subsection{$W(\pi) = S(\pi)$}
		
	Any $\Phi \in \Sch(\Mat_n(\F))$ is invariant by some compact open subgroup $\Omega < \GL_n(\vO_{\F})$, i.e., $\Phi(\omega X) = \Phi(X) = \Phi(X \omega)$ for all $\omega \in \Omega, X \in \Mat_n(\F)$. The following linear form
	$$ \ell_{\Omega}: \Whi(\pi^{\infty}, \psi) \to \C, \quad W \mapsto \oint_{\Omega} W(x) \ud x $$
is a smooth functional on $V_{\pi}^{\infty} \simeq \Whi(\pi^{\infty}, \psi)$, hence can be viewed as a smooth vector $\widetilde{v} \in V_{\widetilde{\pi}}^{\infty}$ in the contragredient representation $\widetilde{\pi}$. Given $W \in \Whi(\pi^{\infty}, \psi)$, understood as a vector $v \in V_{\pi}^{\infty}$, the associated matrix coefficient is given by
	$$ \beta(g) = \Pairing{\pi(g).v}{\widetilde{v}} = \widetilde{v} \left( \pi(g).v \right) = \oint_{\Omega} W(xg) \ud x. $$
	Replacing by an open subgroup if necessary, we may assume that $W$ is invariant by $\Omega$, i.e., $W(gx) = W(g)$ for all $g \in \GL_n(\F), x \in \Omega$. We have
\begin{multline*}
	\int_{\SL_n(\F)} (\Phi \cdot W) \left( \begin{pmatrix} x & \\ & \id_{n-1} \end{pmatrix} g \right) \ud g = \int_{\substack{\GL_n(\F) \\ \det g = x}} \left( \int_{\Omega} \Phi \left( \omega^{-1} g \omega \right) \ud \omega \right) W \left( g \right) \ud g \\
	= \int_{\substack{\GL_n(\F) \\ \det g = x}} \Phi(g) \left( \int_{\Omega} W \left( \omega g \omega^{-1} \right) \ud \omega \right) \ud g = \int_{\substack{\GL_n(\F) \\ \det g = x}} \Phi(g) \left( \int_{\Omega} W \left( \omega g \right) \ud \omega \right) \ud g = \int_{\substack{\GL_n(\F) \\ \det g = x}} \Phi(g) \beta(g) \ud g,
\end{multline*}
	which lies in $\norm[x]_{\F}^{-\frac{n-1}{2}} \cdot S(\pi)$. We conclude the inclusion $W(\pi) \subseteq S(\pi)$ since $\Phi$ and $W$ are arbitrary.
	
	Fix an open compact subgroup $\Omega_0 < \GL_n(\vO_{\F})$. Define for any open compact subgroup $\Omega < \Omega_0$ and any $h \in \GL_n(\F)$ a linear form
	$$ \mu_{\Omega,h}: \Whi(\pi^{\infty}, \psi) \to \C, \quad W \mapsto \int_{\Omega} W(hx) \ud x. $$
	It is $\Omega$-invariant, hence $\GL_n(\vO_{\F})$-finite. If $W \in \Whi(\pi^{\infty}, \psi)$ is annihilated by all such $\mu_{\Omega,h}$, then evidently $W = 0$; equivalently the set $\{ \mu_{\Omega,h} \}$ span the space of $\GL_n(\vO_{\F})$-finite linear forms on $\Whi(\pi^{\infty}, \psi) \simeq V_{\pi}^{\infty}$. In particular, any $\beta \in C(\pi^{\infty})$ is in the span of $g \mapsto \mu_{\Omega,h}(\pi(g).W)$ and $\Omega,h$ and $W$ vary. Given any $\Phi \in \Sch(\Mat_n(\F))$, we take $\Omega_0$ to so small that $\Phi$ is (left and right) invariant by any element in $\Omega_0$. We have
\begin{multline*}
	\int_{\SL_n(\F)} \Phi \left( \begin{pmatrix} x & \\ & \id_{n-1} \end{pmatrix} g \right) \mu_{\Omega,h} \left( \pi\left( \begin{pmatrix} x & \\ & \id_{n-1} \end{pmatrix} g \right).W \right) \ud g = \int_{\substack{\GL_n(\F) \\ \det g = x}} \Phi(g) \left( \int_{\Omega} W(h \omega g) \ud \omega \right) \ud g \\
	= \int_{\Omega} \int_{\substack{\GL_n(\F) \\ \det g = x}} \Phi(g) W(h g \omega) \ud g \ud \omega = \int_{\Omega} \int_{\substack{\GL_n(\F) \\ \det g = x}} \Phi(h^{-1} g h) W(g h \omega) \ud g \ud \omega.
\end{multline*}
	Let $\Omega' = \Omega \cap \mathrm{Stab}(W)$ and write $\Omega = \sideset{}{_{i \in I}} \bigsqcup \omega_i \Omega'$ for a finite set $I$. Writing $\Phi_1(g) = \Phi(h^{-1} g h)$, we have
	$$ \int_{\Omega} \int_{\substack{\GL_n(\F) \\ \det g = x}} \Phi(h^{-1} g h) W(g h \omega) \ud g \ud \omega = \Vol(\Omega') \sideset{}{_{i \in I}} \sum \int_{\substack{\GL_n(\F) \\ \det g = x}} \Phi_1(g) \left( \pi(h \omega_i).W \right)(g) \ud g, $$
	which lies in $\norm[x]_{\F}^{-\frac{n-1}{2}} \cdot W(\pi)$. We conclude the inclusion $S(\pi) \subseteq W(\pi)$ since $\Phi$ and $\beta$ are arbitrary.
	
\begin{remark}
	The above crucial properties of $\mu_{\Omega,h}$ were used in \cite[pp. 421]{JPS83}.
\end{remark}

	\subsection{$W(\pi) = \VorH(\pi,0)$}
	\label{sec: AuxMeasSLn}
		
	To every $\Phi \in \Sch(\Mat_n(\F))$, a complex measure $\rho_{\Phi}$ on $\SL_n(\F)$ was introduced in \cite[(4.2)]{JPS83}. We recall its construction as follows. For $u_i,v_j \in \F$ and $\kappa \in \GL_n(\vO_{\F})$ we set
\begin{multline} \label{eq: AuxMeasSLnKerBis}
	\theta_{\Phi}(u_1; u_2,\dots,u_n; v_2,\dots,v_n; \kappa) = \\
	\int_{\F^{\frac{n(n-1)}{2}}} \Phi \left( \begin{pmatrix} u_1 & x_{1,2} & \cdots & x_{1,n} \\ & u_2 & \cdots & x_{2,n} \\ & & \ddots & \vdots \\ & & & u_n \end{pmatrix} \kappa \right) \psi \left( \sum_{i=2}^n v_i x_{i-1,i} \right) \prod_{i<j} \ud x_{i,j}. 
\end{multline}
	It is Schwartz for any given $\kappa$. Denote by $K_{\Phi}$ its $\psi$-Fourier transform, i.e.,
\begin{equation} \label{eq: AuxMeasSLnKer}
	K_{\Phi}(u; u_2,\dots,u_n; v_2,\dots,v_n; \kappa) := \int_{\F} \theta_{\Phi}(x; u_2,\dots,u_n; v_2,\dots,v_n; \kappa) \psi(-xu) \ud x. 
\end{equation}
	The complex measure $\rho_{\Phi}$ is then defined by
\begin{multline} \label{eq: AuxMeasSLn}
	\int_{\SL_n(\F)} F(h) \ud \rho_{\Phi}(h) = \\
	\int_{v \in \F} \int_{u_i \in \F^{\times}} \int_{\kappa \in \SL_n(\vO_{\F})} F \left[ \begin{pmatrix} u_2 \cdots u_n & & & \\ & 1 & & \\ & & \ddots & \\ & & & 1 \end{pmatrix}^{-1} \begin{pmatrix} 1 & v & & \\ & 1 & & \\ & & \ddots & \\ & & & 1 \end{pmatrix} \begin{pmatrix} 1 & & & \\ & u_2 & & \\ & & \ddots & \\ & & & u_n \end{pmatrix} \kappa \right] \cdot \\
	K_{\Phi}(v; u_2,\dots,u_n; u_2^{-1},\dots,u_n^{-1}; \kappa) \ud \kappa \left( \prod_{i=2}^n \norm[u_i]_{\F}^{i-1} \ud^{\times} u_i \right) \ud v.
\end{multline}
	It follows easily from the $\GL_n(\vO_{\F})$-finiteness of $\Phi$ that $\rho_{\Phi}$ has compact support.
\begin{lemma} \label{lem: AuxMeasSLn}
	For any $\Phi \in \Sch(\Mat_n(\F))$ and $W \in \Whi(\pi^{\infty}, \psi)$ we have for any $t \in \F^{\times}$
\begin{equation*}
	\int_{\SL_n(\F)} \left( \Phi \cdot W \right) \left( \begin{pmatrix} t & \\ & \id_{n-1} \end{pmatrix} g \right) \ud g = \norm[t]_{\F}^{1-n} \int_{\SL_n(\F)} W \left( \begin{pmatrix} t & \\ & \id_{n-1} \end{pmatrix} h \right) \ud \rho_{\Phi}(h).
\end{equation*}
\end{lemma}
\begin{proof}
	See the proof of Lemma \ref{lem: AuxMeasSLnId} below, valid in both archimedean and non-archimedean cases.
\end{proof}
	
\noindent We rewrite the equality in Lemma \ref{lem: AuxMeasSLn} as
	$$ \int_{\SL_n(\F)} \left( \Phi \cdot W \right) \left( \begin{pmatrix} t & \\ & \id_{n-1} \end{pmatrix} g \right) \ud g = \norm[t]_{\F}^{1-n} \cdot W_1 \left( \begin{pmatrix} t & \\ & \id_{n-1} \end{pmatrix} \right) $$
	where $W_1 \in \Whi(\pi^{\infty}, \psi)$ is given by $W_1(g) := \int_{\SL_n(\F)} W(gh) \ud \rho_{\Phi}(h)$. This proves $W(\pi) \subseteq \VorH(\pi,0)$ since $\Phi$ and $W$ are arbitrary.
		
	Let $W \in \Whi(\pi^{\infty}, \psi)$. Suppose $m \in \Z_{\geq 1}$ is sufficiently large so that $W(g\kappa) = W(g)$ for any $g \in \GL_n(\F)$ and $\kappa \in \gp{K}_n(m)$, the level $m$ principal congruence subgroup of $\GL_n(\vO_{\F})$, i.e., the kernel of the mod $\vP_{\F}^m$ homomorphism $\GL_n(\vO_{\F}) \rightarrow \GL_n(\vO_{\F}/\vP_{\F}^m)$. Let $\gp{K}_n(m)^1 := \gp{K}_n(m) \cap \SL_n(\vO_{\F})$. We take $\Phi = \id_S$ for the set $S$ given in the proof of \cite[Lemma (4.3.4)]{JPS79}, namely
	$$ S = \begin{pmatrix} \vP_{\F}^{-m} & \vO_{\F} & \cdots & \vO_{\F} \\ 
	\vP_{\F}^m & U_{\F}^m & \ddots & \vdots \\
	\vdots & \ddots & \ddots & \vdots \\ \vP_{\F}^m & \cdots & \ddots & U_{\F}^m \end{pmatrix}. $$
	From the proof of \cite[Lemma (4.3.4)]{JPS79} we know that $S$ satisfies the following properties:
\begin{itemize}
	\item[(1)] $S \cdot \gp{K}_n(m) = S$;
	\item[(2)] For $b \in \gp{B}_n(\F), \kappa \in \gp{K}_n(0)$, the relation $b \kappa \in S$ implies $\kappa \in (\gp{B}_n(\F) \cap \gp{K}_n(0)) \gp{K}_n(m)^1$.
\end{itemize}
	Applying the above property (2) and noting that the $\gp{B}_n(\F) \cap \gp{K}_n(0)$-part is absorbed by the integral over $\gp{B}_n(\F)$, we get
\begin{multline} \label{eq: GJParInt}
	\int_{\SL_n(\F)} \left( \Phi \cdot W \right) \left( \begin{pmatrix} t & \\ & \id_{n-1} \end{pmatrix} g \right) \ud g = \Vol \left( (\gp{B}_n(\F) \cap \gp{K}_n(0)) \gp{K}_n(m)^1 \right) \int_{x_{i,j} \in \F} \int_{u_i \in \F^{\times}} \\
	\left( \Phi \cdot W \right) \left[ \begin{pmatrix} t & \\ & \id_{n-1} \end{pmatrix} \begin{pmatrix} u_1 & x_{1,2} & \cdots & x_{1,n} \\ & u_2 & \cdots & x_{2,n} \\ & & \ddots & \vdots \\ & & & u_n \end{pmatrix} \right] \norm[u_1]_{\F}^{1-n} \left( \prod_{i=2}^n \norm[u_i]_{\F}^{i-n} \ud^{\times} u_i \right) \prod_{i<j} \ud x_{i,j} = \\
	\Vol \left( (\gp{B}_n(\F) \cap \gp{K}_n(0)) \gp{K}_n(m)^1 \right) \norm[t]_{\F}^{1-n} \int_{u_i \in \F^{\times}} \\
	W \left[ \begin{pmatrix} tu_1 & &  & \\ & u_2 & & \\ & & \ddots & \\ & & & u_n \end{pmatrix} \right] \theta_{\Phi}(tu_1, u_2, \dots, u_n; u_2^{-1}, \dots, u_n^{-1}; \id) \left( \prod_{i=2}^n \norm[u_i]_{\F}^{i-1} \ud^{\times} u_i \right).
\end{multline}
	But in the proof of \cite[Lemma (4.3.4)]{JPS79} we also know the formula
	$$ \theta_{\Phi}(u_1, u_2, \dots, u_n; v_2, \dots, v_n; \id) = \id_{\vP_{\F}^{-m}}(u_1) \prod_{j=2}^n \id_{U_{\F}^m}(u_j) \id_{\vO_{\F}}(v_j). $$
	It follows that
	$$ \int_{\SL_n(\F)} \left( \Phi \cdot W \right) \left( \begin{pmatrix} t & \\ & \id_{n-1} \end{pmatrix} g \right) \ud g = c_m \cdot \id_{\vP_{\F}^{-m}}(t) \norm[t]_{\F}^{1-n} W \begin{pmatrix} t & \\ & \id_{n-1} \end{pmatrix} $$
	for a constant $c_m > 0$ depending only on $m$. But $t \mapsto W \begin{pmatrix} t & \\ & \id_{n-1} \end{pmatrix}$ vanishes for $\norm[t]_{\F} \gg 1$. Taking $m$ sufficiently large so that $W \begin{pmatrix} t & \\ & \id_{n-1} \end{pmatrix} \neq 0$ implies $t \in \vP_{\F}^{-m}$, we see $\norm[t]_{\F}^{\frac{1-n}{2}} W \begin{pmatrix} t & \\ & \id_{n-1} \end{pmatrix} \in W(\pi)$, proving $\VorH(\pi,0) \subseteq W(\pi)$, since $W$ is arbitrary.

	\subsection{$\VorH(\pi,0) = \VorH(\pi,j)$}
		
	This independence of $j$ is stated in Definition \ref{def: VoronoiGLn} (1). The proof is quite similar with that of \cite[Proposition (4.1.4)]{JPS79}. We only outline the main steps. For $W \in \Whi(\pi^{\infty}, \psi)$ and $0 \leq j \leq n-2$, resp. $0 \leq j \leq n-3$ and $\Phi \in \Sch(\F)$, we introduce
	$$ h(t; W, j) := \norm[t]_{\F}^{-\frac{n-1}{2}} \int_{\F^j} W \begin{pmatrix} t & & \\ \vec{x} & \id_j & \\ & & \id_{n-1-j} \end{pmatrix} \ud \vec{x}, $$
	$$ h(t; W, j, \Phi) := \norm[t]_{\F}^{-\frac{n-1}{2}} \int_{\F} \int_{\F^j} W \begin{pmatrix} t & & & \\ \vec{x} & \id_j & & \\ z & & 1 & \\ & & & \id_{n-2-j} \end{pmatrix} \ud \vec{x} \Phi(z) \ud z. $$
	Let $[*]$ be the space of functions generated by those of the form $*$. We need to prove the inclusions
\begin{equation} \label{eq: AuxIncl} 
	[h(t; W, j)] \subset [h(t; W, j-1, \Phi)] \subset [h(t; W, j-1)] \subset [h(t; W, j-1, \Phi)] \subset [h(t; W, j)]. 
\end{equation}
	The first inclusion follows from \cite[Lemma (4.1.5)]{JPS79} by choosing $\Phi$ to be the characteristic function of a large compact open subgroup of $\F$. The second one follows from the smoothness of $W$. The third one follows by choosing $\Phi$ to be the characteristic function of a small compact open subgroup of $\F$. The last inclusion follows from the functional equation in \cite[Lemma (4.1.5)]{JPS79}.

\section{Equivalence of Definitions: Archimedean Case}
\label{sec: EquivDefsA}

	\subsection{Irreducible Representations of Casselman--Wallach Type}
	
	For $\F \in \{ \R, \C \}$, the relevant irreducible representations $\pi$ of $\GL_d(\F)$ are those considered in \cite{J09}. We shall call them \emph{of Casselman--Wallach type} and recall their construction and basic properties as follows.
	
	Let $d= \sum_{i=1}^r d_i$ be a partition of $d$ with $d_i \in \{ 1,2 \}$. Let $(\pi_i, V_i)$ be an irreducible unitary representation of $\GL_{d_i}(\F)$. In particular, $d_i=2$ forces $\F=\R$ and $\pi_i$ be a discrete series representation of $\GL_2(\R)$. Write $\vec{\pi} = (\pi_1, \dots, \pi_r)$. Let $\vec{u} = (u_1, \dots, u_r) \in \C^r$. Let $\gp{P}$ be the \emph{lower} parabolic subgroup of type $(d_1, \dots, d_r)$, with Levi-decomposition $\gp{P} = \gp{M} \gp{U}$ in $\GL_d$. Here $\gp{M}$ is the group of block diagonal matrices with sizes $d_1 \times d_1, \dots, d_r \times d_r$. We denote by $\delta_{\gp{P}}$ the module of the group $\gp{P}(\F)$. We denote by $(\pi_{\vec{\pi}, \vec{u}}, V_{\vec{\pi}, \vec{u}})$ the representation of $\GL_d(\F)$ induced by the representation $(\pi_1 \otimes \norm_{\F}^{u_1}, \dots, \pi_r \otimes \norm_{\F}^{u_r})$ of $\gp{P}(\F)$. Thus $V_{\vec{\pi}, \vec{u}}$ may be viewed as the space of functions $f$ on $\GL_d(\F)$ with values in the projective tensor product $V_1 \widehat{\otimes} V_2 \cdots \widehat{\otimes} V_r$ such that
\begin{itemize}
	\item for any $v \in \gp{U}(\F), m = \diag(m_1, \dots, m_r) \in \gp{M}(\F)$ and $g \in \GL_d(\F)$
	$$ f(vmg) = \delta_{\gp{P}}^{1/2}(m) \cdot \left( \pi_1(m_1) \norm[\det m_1]_{\F}^{u_1} \otimes \cdots \otimes \pi_r(m_r) \norm[\det m_r]_{\F}^{u_r} \right).f(g); $$
	\item $\int_{\gp{K}_d} \Norm[f(\kappa)]^2 \ud \kappa < \infty$ where $\gp{K}_d \in \{ \mathrm{O}_d(\R), \mathrm{U}_d(\C) \}$ is the standard maximal compact subgroup of $\GL_d(\F)$ and $\Norm[\cdot]$ is the norm on the Hilbert space $V_1 \widehat{\otimes} V_2 \cdots \widehat{\otimes} V_r$;
	\item $\pi_{\vec{\pi}, \vec{u}}(g)f(x) := f(xg)$ for any $x,g \in \GL_d(\F)$.
\end{itemize}
	Note that $(\pi_{\vec{\pi}, \vec{u}}, V_{\vec{\pi}, \vec{u}})$ is a Hilbert representation \cite[1.1.2]{Wal88} whose \emph{dual} representation is easily identified with $(\pi_{\vec{\pi}^{\vee}, -\vec{u}}, V_{\vec{\pi}^{\vee}, -\vec{u}})$, where $\vec{\pi}^{\vee} := (\pi_1^{\vee}, \dots, \pi_r^{\vee})$. Note also that the space of smooth vectors $V_{\vec{\pi}, \vec{u}}^{\infty}$ is identified with the space of smooth functions on $\gp{K}_d$. The representation $(\pi_{\vec{\pi}, \vec{u}}^{\infty}, V_{\vec{\pi}, \vec{u}}^{\infty})$ is a smooth Fr\'echet representation of $\GL_d(\F)$ having moderate growth \cite[11.5.2]{Wal92}, or a Casselman--Wallach representation \cite[\S 2]{J09}. Moreover, $\pi_{\vec{\pi}, \vec{u}}$ is irreducible if and only if $\pi_{\vec{\pi}^{\vee}, -\vec{u}}$ is irreducible. The relevant irreducible representations $\pi$ are precisely those irreducible $\pi_{\vec{\pi}, \vec{u}}$. They are generic.
	
	Wallach's Schwartz functions \cite[7.1.2]{Wal88} will serve us. We shall write $\Sch(\GL_d(\F))$ and $\Sch(\SL_d(\F))$ for the spaces of such functions on $\GL_d(\F)$ and $\SL_d(\F)$ respectively. Note that the ordinary Schwartz space $\Sch(\Mat_d(\F))$ is naturally a bi-module over $\Sch(\GL_d(\F))$ via the left $\rpL$ and right $\rpR$ multiplication, hence also over $\Sch(\SL_d(\F))$. Note also that $(\pi_{\vec{\pi}, \vec{u}}^{\infty}, V_{\vec{\pi}, \vec{u}}^{\infty})$ is a module over $\Sch(\GL_d(\F))$ by \cite[11.8.1]{Wal92}.

	\subsection{$W(\pi) = S(\pi)$}
	
	Our proof is based on the following results.
	
\begin{lemma} \label{lem: DualByFunctGL}
	For any $f \in \Sch(\GL_n(\F))$ the functional
	$$ \ell_f: \Whi(\pi^{\infty}, \psi) \to \C, \quad \ell_f(W) := \int_{\GL_n(\F)} f(x) W(x) \ud x $$
is well-defined and represented by some $\widetilde{v}_f \in V_{\pi^{\vee}}^{\infty}$. Conversely, every element in $V_{\pi^{\vee}}^{\infty}$ is of the form $\widetilde{v}_f$.
\end{lemma}
\begin{proof}
	Let $\ell: V_{\pi}^{\infty} \to \C$ be a Whittaker functional realizing the Whittaker model. There are elements $X_i$ in the universal enveloping algebra of the complexified Lie algebra of $\GL_n(\F)$ such that
	$$ \forall \ v \in V_{\pi}^{\infty}, \quad \norm[\ell(v)] \leq \sum_{i=1}^r \Norm[\ud \pi(X_i).v] $$
by definition. For any $f \in \Sch(\GL_n(\F))$ and $v \in V_{\pi}^{\infty}$ we have by the moderate growth of $\pi$
	$$ \extnorm{\int_{\GL_n(\F)} f(x) W_v(x) \ud x} = \extnorm{\ell \left( \pi(f).v \right)} \leq \sum_{i=1}^r \extNorm{\ud \pi(X_i)\pi(f).v} = \sum_{i=1}^r \extNorm{\pi(\ud \rpL(X_i).f).v} \leq \nu(f) \Norm[v] $$
for some semi-norm $\nu$ on $\Sch(\GL_n(\F))$. Therefore $\ell_f$ is represented by some vector $\widetilde{v}_f \in V_{\pi^{\vee}}$ as
	$$ \ell_f(v) = \extPairing{v}{\widetilde{v}_f}, \quad \forall \ v \in V_{\pi}^{\infty}. $$
It is easy to see $\widetilde{v}_f \in V_{\pi^{\vee}}^{\infty}$, and in fact $\ud \pi(X).\widetilde{v}_f = \widetilde{v}_{\ud \rpR(X).f}$. Moreover, since $\Sch(\GL_n(\F))$ is a convolution algebra by \cite[Theorem 7.1.1]{Wal92}, we have $\widetilde{v}_{f*h} = \pi^{\vee}(h^{\vee}).\widetilde{v}_f$ for any $f,h \in \Sch(\GL_n(\F))$, where $h^{\vee}(g) := h(g^{-1})$. The vector space spanned by $\widetilde{v}_f$ is thus a $\Sch(\GL_n(\F))$-submodule of $V_{\pi^{\vee}}^{\infty}$. The latter being an algebraically irreducible $\Sch(\GL_n(\F))$-module by \cite[Theorem 11.8.2]{Wal92}, we are done.
\end{proof}

\noindent Let $\GL_n(\R)_+$ be the subgroup of $\GL_n(\R)$ with positive determinant. We have a homomorphism of groups
	$$ \F^{\times} \times \SL_n(\F) \to \F^{\times} \SL_n(\F) = \begin{cases}
		\GL_n(\R)_+ & \text{if } \F=\R \ \& \ 2 \mid n \\
		\GL_n(\F) & \text{otherwise}
	\end{cases} $$
which is surjective with finite kernel isomorphic to $\mu_n(\F) := \left\{ \zeta \in \F^{\times} \ \middle| \ \zeta^n=1 \right\}$. The Lie algebra of $\GL_n(\F)$ clearly has a basis which is the union of bases of the Lie algebra of $\F^{\times}$ and $\SL_n(\F)$. Writing $\GL_n(\F) \ni g = z h$ with $z \in \F^{\times}$ and $h \in \SL_n(\F)$, the Harish-Chandra's norm \cite[\S 3.1]{J09} has
	$$ \Norm[g] = z\bar{z} \cdot \Tr(h {}^t\bar{h}) + (z\bar{z})^{-1} \cdot \Tr(h^{-1} {}^t\bar{h}^{-1}) \leq \Norm[z]_1 \cdot \Norm[h]_2, $$
where $\Norm_1$ and $\Norm_2$ are the Harish-Chandra's norms on $\F^{\times}$ and $\SL_n(\F)$ respectively. Since $\SL_n(\F)$ is a closed subset in $\Mat_n(\F)$ not containing $0$, we have $\Tr(h {}^t\bar{h}) \geq c_n$ for all $h \in \SL_n(\F)$ and some constant $c_n > 0$ depending only on $n$. We get $\Norm[g] \gg \Norm[z]_1$. The entries of $h^{-1}$ are polynomials of those of $h$ with degree $\leq n-1$. Therefore $\Tr(h^{-1} {}^t\bar{h}^{-1}) \ll \Tr(h {}^t\bar{h})^{n-1}$ with implied constant depending only on $n$. Now that $\Norm[g] \geq \Tr(h {}^t\bar{h})$ or $\Tr(h^{-1} {}^t\bar{h}^{-1})$ for any $z$, we get $\Norm[g] + \Norm[g]^{n-1} \gg \Norm[h]_2$. Hence the semi-norms defining $\Sch(\F^{\times} \SL_n(\F))$ are equivalent to those defining $\Sch(\F^{\times}) \widehat{\otimes} \Sch(\SL_n(\F)) = \Sch(\F^{\times} \times \SL_n(\F))$, yielding
	$$ \Sch(\F^{\times} \SL_n(\F)) \simeq \left\{ f \in \Sch(\F^{\times} \times \SL_n(\F)) \ \middle| \ f(\zeta z, \zeta^{-1} h) = f(z, h), \forall \zeta \in \mu_n(\F) \right\}. $$
	In other words any $f \in \Sch(\F^{\times} \SL_n(\F))$ is an average of some $\tilde{f} \in \Sch(\F^{\times} \times \SL_n(\F))$ in the sense
	$$ f(zh) = \sum_{\zeta \in \mu_n(\F)} \tilde{f}(\zeta z,\zeta^{-1} h). $$
	Therefore for any $\omega$ (quasi-)character of $\F^{\times}$ we have
\begin{equation} \label{eq: SchTrans}
	\int_{\F^{\times}} f(zh) \omega(z) \ud^{\times} z = \sum_{\zeta \in \mu_n(\F)} \omega(\zeta) \int_{\F^{\times}} \tilde{f}(z, \zeta h) \omega(z) \ud^{\times} z = \sum_{\zeta \in \mu_n(\F)} f_1(\zeta h) \omega(\zeta) 
\end{equation}
	for the function $f_1(h) := \int_{\F^{\times}} \tilde{f}(z, h) \omega(z) \ud^{\times} z \in \Sch(\SL_n(\F))$.

\begin{corollary} \label{cor: DualByFunctSL}
	For any $f \in \Sch(\SL_n(\F))$ the functional
	$$ \lambda_f: \Whi(\pi^{\infty}, \psi) \to \C, \quad \lambda_f(W) := \int_{\SL_n(\F)} f(x) W(x) \ud x $$
is well-defined and $\lambda_f \in V_{\pi^{\vee}}^{\infty}$. Conversely, every element in $V_{\pi^{\vee}}^{\infty}$ is of the form $\lambda_{f_1} + \lambda_{f_2} \circ \pi(\diag(-1,\id_{n-1}))$ if $2 \mid n$ and $\F=\R$, or else of the form $\lambda_f$.
\end{corollary}
\begin{proof}
	Let $\omega$ be the central character of $\pi$. Take $k \in \Sch(\F^{\times})$ with
	$$ \int_{\F^{\times}} k(z) \omega(z) \ud^{\times} z = 1. $$
	For $f \in \Sch(\SL_n(\F))$, the following formula defines a function $f_1 \in \Sch(\F^{\times} \SL_n(\F)) \subset \Sch(\GL_n(\F))$
	$$ f_1(zh) := \sum_{\zeta \in \mu_n(\F)} k(\zeta z) f(\zeta^{-1} h), \quad \forall \ z \in \F^{\times}, h \in \SL_n(\F). $$
	One easily verifies $\lambda_f = \ell_{f_1} \in V_{\pi^{\vee}}^{\infty}$ defined in Lemma \ref{lem: DualByFunctGL}. Conversely, represent the relevant element of $V_{\pi^{\vee}}^{\infty}$ in the form $\ell_f$ for some $f \in \Sch(\GL_n(\F))$. If we are not in the exceptional case ($2 \mid n$ and $\F=\R$), then \eqref{eq: SchTrans} shows that $\ell_f = \lambda_{f_1}$ for some $f_1 \in \Sch(\SL_n(\F))$. In the exceptional case, we argue for $f \mid_{\GL_n(\R)_+},\diag(-1,\id_{n-1}).f \mid_{\GL_n(\R)_+} \in \Sch(\GL_n(\R)_+)$ to find the required $f_1$ and $f_2$.
\end{proof}

	We now prove the equality $W(\pi) = S(\pi)$.
	
	Note $\Sch(\Mat_n(\F))$ is a smooth Fr\'echet representation of $\GL_n(\F) \times \GL_n(\F)$, hence also of $\SL_n(\F) \times \SL_n(\F)$. For any $\Phi \in \Sch(\Mat_n(\F))$ we can find finitely many $\Phi_j \in \Sch(\Mat_n(\F))$ and $f_j \in \Cont_c^{\infty}(\SL_n(\F))$ such that
	$$ \Phi(g) = \sum_j \int_{\SL_n(\F)} \Phi_j(x^{-1}g) f_j(x) \ud x. $$
	For any $W \in \Whi(\pi^{\infty}, \psi)$ we have
\begin{multline*}
	\int_{\substack{\GL_n(\F) \\ \det g = t}} \Phi(g) W(g) \ud g = \sum_j \int_{\substack{\GL_n(\F) \\ \det g = t}} \int_{\SL_n(\F)} \Phi_j(x^{-1}g) f_j(x) W(g) \ud x \ud g = \\
	 \sum_j \int_{\substack{\GL_n(\F) \\ \det g = t}} \int_{\SL_n(\F)} \Phi_j(g) f_j(x) W(xg) \ud x \ud g = \sum_j \int_{\substack{\GL_n(\F) \\ \det g = t}} \Phi_j(g) \lambda_{f_j}(\pi(g).W) \ud g.
\end{multline*}
	Since $g \mapsto \lambda_{f_j}(\pi(g).W)$ lies in $C(\pi^{\infty})$ by Corollary \ref{cor: DualByFunctSL}, we have proved $W(\pi) \subseteq S(\pi)$.
	
	Conversely for any $\beta \in C(\pi^{\infty})$ we can find at most two $f_j \in \Sch(\SL_n(\F))$ and $W \in \Whi(\pi^{\infty},\psi)$ with
	$$ \beta(g) = \sum_j \int_{\SL_n(\F)} f_j(x) W \left( x \begin{pmatrix} \delta_j & \\ & \id_{n-1} \end{pmatrix} g \right) \ud x $$
for $\delta_j \in \{ \pm 1 \}$ by Corollary \ref{cor: DualByFunctSL}. For any $\Phi \in \Sch(\Mat_n(\F))$ we have
\begin{multline*}
	\int_{\substack{\GL_n(\F) \\ \det g = t}} \Phi(g) \beta(g) \ud g = \sum_j \int_{\substack{\GL_n(\F) \\ \det g = t}} \int_{\SL_n(\F)} \Phi(g) f_j(x) W\left( x \begin{pmatrix} \delta_j & \\ & \id_{n-1} \end{pmatrix} g \right) \ud x \ud g \\
	= \sum_j \int_{\substack{\GL_n(\F) \\ \det g = t}} \left\{ \int_{\SL_n(\F)} \Phi \left( \begin{pmatrix} \delta_j^{-1} & \\ & \id_{n-1} \end{pmatrix} x^{-1}g \begin{pmatrix} \delta_j & \\ & \id_{n-1} \end{pmatrix} \right) f_j(x) \ud x \right\} W(g) \ud g.
\end{multline*}
	Since the function in $g$ defined by the inner integral lies in $\Sch(\Mat_n(\F))$, we have proved $S(\pi) \subseteq W(\pi)$.

	\subsection{$W(\pi) = \VorH(\pi,0)$}
	
	To every $\Phi \in \Sch(\Mat_n(\F))$, we define $\theta_{\Phi}$, $K_{\Phi}$ and $\rho_{\Phi}$ similarly to \eqref{eq: AuxMeasSLnKerBis}, \eqref{eq: AuxMeasSLnKer} and \eqref{eq: AuxMeasSLn}. Let $\gp{K}_n$ be the standard maximal compact subgroup of $\GL_n(\F)$. Write $\gp{K}_n^1 = \gp{K}_n \cap \SL_n(\F)$.
	
\begin{lemma} \label{lem: FuncExt}
	Define a subset $\Omega$ of $\Mat_n(\F)$ as
	$$ \Omega := \left\{ \begin{pmatrix} Z \\ Y \end{pmatrix} \ \middle| \ Z \in \Mat_{1 \times n}(\F), Y \in \Mat_{(n-1) \times n}(\F) \text{ with rank } n-1 \right\}. $$
\begin{itemize}
	\item[(1)] $\Omega$ is an open subset of $\Mat_n(\F)$ with an isomorphism as smooth manifolds
	$$ \sigma: \F^n \times \gp{B}_{n-1}^+(\F) \times \gp{K}_n^1 \simeq \Omega, \quad (Z, b, \kappa) \mapsto \begin{pmatrix} Z \\ \vec{0} \ b \end{pmatrix} \kappa, $$
	where $\gp{B}_{n-1}^+(\F)$ is the subgroup of $\gp{B}_{n-1}(\F)$ with positive diagonal entries.
	\item[(2)] For any $\phi \in \Sch(\F^n) \widehat{\otimes} \Cont_c^{\infty}(\gp{B}_{n-1}^+(\F) \times \gp{K}_n^1)$, the function $\Phi$ defined via $\sigma$ on $\Omega$ by
	$$ \Phi \left( \sigma(Z, b, \kappa) \right) = \phi(Z, b, \kappa) $$
	and equal to $0$ outside $\Omega$ lies in $\Sch(\Mat_n(\F))$.
\end{itemize}
\end{lemma}
\begin{proof}
	This is essentially \cite[Lemma (11.1.3)]{JPS79}: although our statement is slightly stronger, the proof given there goes through.
\end{proof}

\begin{lemma} \label{lem: AuxMeasSLnArch}
	(1) The function $K_{\Phi}$ lies in the (complete) tensor product $\Sch(\F^{2n-1}) \widehat{\otimes} \Cont^{\infty}(\gp{K}_n)$. Moreover, for any $ h = \diag(\delta_1, \dots, \delta_n) \in \gp{B}_n(\F) \cap \gp{K}_n = \gp{T}_n^0(\F)$ we have the relation
	$$ K_{\Phi}(u; u_2,\dots,u_n; v_2,\dots,v_n; h\kappa) = K_{\Phi} \left( \frac{u}{\delta_1}; \delta_2 u_2,\dots, \delta_n u_n; \frac{v_2}{\delta_2},\dots,\frac{v_n}{\delta_n}; \kappa \right). $$
	
\noindent (2) The measure $\rho_{\Phi}$ is represented by a function $f_{\Phi}$ on $\gp{P}_n^1(\F) \cdot \gp{K}_n^1$ with 
	$$ \gp{P}_n^1(\F) := \left\{ p(v; u_2, \cdots, u_n) := \begin{pmatrix} \frac{1}{u_2 \cdots u_n} & v & & \\ & u_2 & & \\ & & \ddots & \\ & & & u_n  \end{pmatrix} \ \middle| \ v \in \F, \ u_j \in \F^{\times} \right\}, $$
	as $\rho_{\Phi} = f_{\Phi} \ud p \ud \kappa$, where
\begin{itemize}
	\item $\ud p$ and $\ud \kappa$ are left Haar measures on $\gp{P}_n^1(\F)$ and $\gp{K}_n^1$ respectively;
	\item $f_{\Phi}$ lies in the (complete) tensor product $\Sch(\F \times (\F^{\times})^{n-1}) \widehat{\otimes} \Cont^{\infty}(\gp{K}_n^1)$ with respect to the affine coordinates $p(\cdot)$ of $\gp{P}_n^1(\F)$ satisfying for any $h = \diag(\delta_1, \dots, \delta_n) \in \gp{P}_n^1(\F) \cap \gp{K}_n^1$
	$$ f_{\Phi}(p(v; u_2, \cdots, u_n); h\kappa) = f_{\Phi}(p(\delta_2 v; \delta_2 u_2, \cdots, \delta_n u_n); \kappa). $$
\end{itemize}
	
\noindent (3) For any $f_1 \in \Cont_c^{\infty}(\F \times (\F^{\times})^{n-1})$ and $f_2 \in \Cont^{\infty}(\gp{K}_n^1)$ we can find $\Phi \in \Sch(\Mat_n(\F))$ with
	$$ f_{\Phi}(p(v; u_2, \cdots, u_n); \kappa) = \oint_{\gp{T}_n^{0,1}(\F)} f_1(\delta_2 v; \delta_2 u_2, \cdots, \delta_n u_n) \cdot f_2(h^{-1} \kappa) \ud h, $$
	where $\gp{T}_n^{0,1}(\F) = \gp{P}_n^1(\F) \cap \gp{K}_n^1$ is the subgroup of elements in $\gp{T}_n^{0}(\F)$ with determinant $1$.
\end{lemma}
\begin{proof}
	(1) For any $\Phi \in \Sch(\Mat_n(\F))$ the following function lies in $\Sch(\Mat_n(\F)) \widehat{\otimes} \Cont^{\infty}(\gp{K}_n)$
	$$ \Mat_n(\F) \times \gp{K}_n \to \C, \quad (X, \kappa) \mapsto \Phi(X \kappa) $$
	since $\Sch(\Mat_n(\F))$ is a smooth Fr\'echet representation of $\GL_n(\F)$. Consequently we see $K_{\Phi}$ lies in the space $\Sch(\F^{2n-1}) \widehat{\otimes} \Cont^{\infty}(\gp{K}_n)$ by its formula \eqref{eq: AuxMeasSLnKer}. The stated equation also follows readily from \eqref{eq: AuxMeasSLnKer}.
	
\noindent (2) Choosing $\ud \kappa$ on $\gp{K}_n^1$ to be the restriction of $\ud \kappa$ on $\gp{K}_n$, and choosing
	$$ \ud p(v; u_2, \cdots, u_n) = \norm[u_2 \cdots u_n]_{\F} \ud v \ud^{\times} u_2 \cdots \ud^{\times} u_n $$
	we easily deduce the following formula from \eqref{eq: AuxMeasSLn}
	$$ f_{\Phi}(p(v; u_2, \cdots, u_n); \kappa) = K_{\Phi} \left( (u_3 \cdots u_n) v; u_2,\dots,u_n; u_2^{-1},\dots,u_n^{-1}; \kappa \right) \cdot \prod_{i=3}^n \norm[u_i]_{\F}^{i-1}. $$
	The desired property of $f_{\Phi}$ follows readily from the property of $K_{\Phi}$ established in (1).
	
\noindent (3) We first show that there is $\Psi \in \Sch(\F^n) \widehat{\otimes} \Cont_c^{\infty}(\gp{B}_{n-1}(\F))$ satisfying
\begin{multline} \label{eq: AuxF}
	\int_{\F} \int_{\F^{\frac{n(n-1)}{2}}} \Psi \begin{pmatrix} u_1 & x_{1,2} & \cdots & x_{1,n} \\ & u_2 & \cdots & x_{2,n} \\ & & \ddots & \vdots \\ & & & u_n \end{pmatrix} \psi \left( -u_1v + \sum_{i=2}^n u_i^{-1} x_{i-1,i} \right) \left( \prod_{i<j} \ud x_{i,j} \right) \ud u_1 \\
	= f_1((u_3 \cdots u_n)^{-1}v; u_2, \dots, u_n) \cdot \prod_{i=3}^n \norm[u_i]_{\F}^{1-i}.
\end{multline}
	We take an auxiliary function $\phi \in \Cont_c^{\infty}(\F)$ satisfying $\int_{\F} \phi(x) \ud x = 1$; define $\Psi$ in \eqref{eq: AuxF} by the formula
\begin{equation} \label{eq: AuxFF}
	\left( \prod_{j > i+1} \phi(x_{i,j}) \right) \cdot \left( \prod_{i=3}^n \norm[u_i]_{\F}^{1-i} \int_{\F} f_1 \left( \frac{t}{u_3 \cdots u_n}; u_2, \dots, u_n \right) \psi(t u_1) \ud t \right) \cdot \left( \prod_{i=2}^n \phi(x_{i-1,i}) \psi(-u_i^{-1} x_{i-1,i}) \right); 
\end{equation}
	and readily verifies that $\Psi \in \Sch(\F) \widehat{\otimes} \Cont_c^{\infty}(\F^{n-1}) \widehat{\otimes} \Cont_c^{\infty}(\gp{B}_{n-1}(\F))$ and \eqref{eq: AuxF} holds true. We then define a function $\Phi$ on $\Omega$ (given in Lemma \ref{lem: FuncExt}) as
\begin{equation} \label{eq: PhiF}
	\Phi \left( \begin{pmatrix} u_1 & x_{1,2} & \cdots & x_{1,n} \\ & u_2 & \cdots & x_{2,n} \\ & & \ddots & \vdots \\ & & & u_n \end{pmatrix} \kappa \right) = \oint_{\gp{T}_n^{0,1}(\F)} \Psi \left( \begin{pmatrix} u_1 & x_{1,2} & \cdots & x_{1,n} \\ & u_2 & \cdots & x_{2,n} \\ & & \ddots & \vdots \\ & & & u_n \end{pmatrix} h \right) f_2(h^{-1} \kappa) \ud h, 
\end{equation}
	where $u_2, \dots, u_n > 0$ and $\kappa \in \gp{K}_n^1$. It is clear that $\Phi \circ \sigma \in \Sch(\F^n) \widehat{\otimes} \Cont_c^{\infty}(\gp{B}_{n-1}^+(\F) \times \gp{K}_n^1)$. Hence it extends to a function in $\Sch(\Mat_n(\F))$ by Lemma \ref{lem: FuncExt} (2). We denote this extension by $\Phi$. Since $\Cont_c^{\infty}(\gp{B}_{n-1}(\F)) \subset \Sch(\F^{\frac{n(n-1)}{2}})$, $\Psi$ is the restriction of a function in $\Sch(\Mat_n(\F))$, still denoted by $\Psi$. We rewrite \eqref{eq: AuxF} as
\begin{equation} \label{eq: AuxFBis}
	f_{\Psi}(p(v; u_2, \cdots, u_n); \id_n) = f_1(v; u_2, \cdots, u_n).
\end{equation}
	We can compute $f_{\Phi}$, first for $u_2, \dots, u_n > 0$ and $\kappa \in \gp{K}_n^1$, with the help of (2) as
\begin{multline*}
	f_{\Phi}(p(v; u_2, \cdots, u_n); \kappa) = \oint_{\gp{T}_n^{0,1}(\F)} f_{\Psi}(p(v; u_2, \cdots, u_n); h) f_2(h^{-1} \kappa) \ud h \\
	= \oint_{\gp{T}_n^{0,1}(\F)} f_{\Psi}(p(\delta_2 v; \delta_2 u_2, \cdots, \delta_n u_n); \id_n) f_2(h^{-1} \kappa) \ud h = \oint_{\gp{T}_n^{0,1}(\F)} f_1(\delta_2 v; \delta_2 u_2, \cdots, \delta_n u_n) \cdot f_2(h^{-1} \kappa) \ud h.
\end{multline*}
	For general $u_j, j \geq 2$, we write $u_j = u_j^0 \delta_j^0$ with $u_j^0 > 0$ and $\norm[\delta_j^0]_{\F}=1$. Writing $h_0 := \diag(\delta_1^0, \delta_2^0, \dots, \delta_n^0)$ with $\delta_1^0 \delta_2^0 \cdots \delta_n^0 = 1$, we obtain the desired equation by its above special case and (2) again
\begin{multline*} 
	f_{\Phi}(p(v; u_2, \cdots, u_n); \kappa) = f_{\Phi}(p((\delta_2^0)^{-1} v; u_2^0, \cdots, u_n^0); h_0\kappa) \\
	= \oint_{\gp{T}_n^{0,1}(\F)} f_1(\delta_2 (\delta_2^0)^{-1} v; \delta_2 u_2^0, \cdots, \delta_n u_n^0) \cdot f_2(h^{-1} h_0 \kappa) \ud h \\
	= \oint_{\gp{T}_n^{0,1}(\F)} f_1(\delta_2 v; \delta_2 \delta_2^0 u_2^0, \cdots, \delta_n \delta_n^0 u_n^0) \cdot f_2(h^{-1} \kappa) \ud h \\
	= \oint_{\gp{T}_n^{0,1}(\F)} f_1(\delta_2 v; \delta_2 u_2, \cdots, \delta_n u_n) \cdot f_2(h^{-1} \kappa) \ud h,
\end{multline*}
	where in the third line we made the change of variables $h \leadsto hh_0$.
\end{proof}

\begin{lemma} \label{lem: AuxMeasSLnId}
	For any $\Phi \in \Sch(\Mat_n(\F))$, $W \in \Whi(\pi^{\infty}, \psi)$ and $t \in \F^{\times}$, we have
\begin{equation*}
	\int_{\SL_n(\F)} \left( \Phi \cdot W \right) \left( \begin{pmatrix} t & \\ & \id_{n-1} \end{pmatrix} g \right) \ud g = \norm[t]_{\F}^{1-n} \int_{\SL_n(\F)} W \left( \begin{pmatrix} t & \\ & \id_{n-1} \end{pmatrix} h \right) \ud \rho_{\Phi}(h).
\end{equation*}
\end{lemma}
\begin{proof}
	We shall deduce the result from \cite[Lemma (4.3.1)]{JPS79}, which states that
\begin{equation} \label{eq: AuxMeasSLnBis}
	\int_{\GL_n(\F)} \left( \Phi \cdot H \right) \left( g \right) \ud g = \int_{\F^{\times}} \left( \int_{\SL_n(\F)} H \left( \begin{pmatrix} t & \\ & \id_{n-1} \end{pmatrix} h \right) \ud \rho_{\Phi}(h) \right) \norm[t]_{\F}^{1-n} \ud^{\times} t, 
\end{equation}
	whenever $H$ is a right smooth function on $\GL_n(\F)$, left $(\gp{N}_n(\F), \psi)$-covariant and making the left hand side of \eqref{eq: AuxMeasSLnBis} absolutely convergent for all $\Phi \in \Sch(\Mat_n(\F))$. For simplicity we call such $H$ ``nice''. We claim that for any $\phi \in \Cont_c^{\infty}(\F^{\times})$ and any $W \in \Whi(\pi^{\infty}, \psi)$ the function $H(g) := W(g) \phi(\det g)$ is nice. The smoothness and covariance are easy to verify. In the last paragraph before \cite[\S (4.4)]{JPS79} it was used that $W(g) \norm[\det g]_{\F}^{s+(n-1)/2}$ is nice if $\Re s \gg 1$ is sufficiently large (with implied constant depending only on $\pi$). Since $\phi(t) \ll \norm[t]_{\F}^N$ for any $N \in \Z$, we see the ``niceness'' of $H(g)$ and conclude the proof of the claim. Inserting such $H$ into \eqref{eq: AuxMeasSLnBis} we get
\begin{multline*}
	\int_{\F^{\times}} \phi(t) \left\{ \int_{\SL_n(\F)} \left( \Phi \cdot W \right) \left( \begin{pmatrix} t & \\ & \id_{n-1} \end{pmatrix} g \right) \ud g \right\} \ud^{\times} t = \\
	\int_{\F^{\times}} \phi(t) \norm[t]_{\F}^{1-n} \left\{ \int_{\SL_n(\F)} W \left( \begin{pmatrix} t & \\ & \id_{n-1} \end{pmatrix} h \right) \ud \rho_{\Phi}(h) \right\} \ud^{\times} t.
\end{multline*}
	The inner integrals on both sides being clearly continuous functions in $t$ and equal as distributions on $\F^{\times}$, they must be equal point-wise in $t$.
\end{proof}

	We now prove the equality $S(\pi) = \VorH(\pi,0)$.
	
	For any $\Phi \in \Sch(\Mat_n(\F))$ and $W \in \Whi(\pi^{\infty}, \psi)$, we rewrite the equality in Lemma \ref{lem: AuxMeasSLnId} as
	$$ \int_{\SL_n(\F)} \left( \Phi \cdot W \right) \left( \begin{pmatrix} t & \\ & \id_{n-1} \end{pmatrix} g \right) \ud g = \norm[t]_{\F}^{1-n} \cdot W_1 \left( \begin{pmatrix} t & \\ & \id_{n-1} \end{pmatrix} \right) $$
	where $W_1(g) := \int_{\SL_n(\F)} W(gh) \ud \rho_{\Phi}(h)$. But $\Whi(\pi^{\infty}, \psi)$ is a smooth representation with moderate growth for both $\gp{P}_n^1(\F)$ and $\gp{K}_n^1$, we deduce $W_1 \in \Whi(\pi^{\infty}, \psi)$ by Lemma \ref{lem: AuxMeasSLnArch} (2). This proves $W(\pi) \subseteq \VorH(\pi,0)$.
	
	Conversely, for any $W \in \Whi(\pi^{\infty}, \psi)$ we may apply Dixmier-Malliavin's theorem to $\gp{P}_n^1(\F)$ and $\gp{K}_n^1$ step-by-step, and get for finitely many $f_i \in \Cont_c^{\infty}(\gp{P}_n^1(\F))$, $g_i \in \Cont^{\infty}(\gp{K}_n^1)$ and $W_i \in \Whi(\pi^{\infty}, \psi)$ that
	$$ W(g) = \sum_i \int_{\gp{P}_n^1(\F)} \int_{\gp{K}_n^1} W_i(g x \kappa) f_i(x) g_i(\kappa) \ud \kappa \ud x. $$
	By Lemma \ref{lem: AuxMeasSLnArch} (3) we find $\Phi_i \in \Sch(\Mat_n(\F))$ such that 
	$$ f_{\Phi_i}(x; \kappa) = \int_{\gp{P}_n^1(\F) \cap \gp{K}_n^1} f_i(xh) g_i(h^{-1}\kappa) \ud h. $$
	By Lemma \ref{lem: AuxMeasSLnId} we get (up to a constant depending on the normalization of $\ud h$)
\begin{multline*} 
	\norm[t]_{\F}^{1-n} W \left( \begin{pmatrix} t & \\ & \id_{n-1} \end{pmatrix} \right) = \norm[t]_{\F}^{1-n} \sum_i \int_{\SL_n(\F)} W_i \left( \begin{pmatrix} t & \\ & \id_{n-1} \end{pmatrix} g \right) \ud \rho_{\Phi_i}(g) \\
	= \sum_i \int_{\SL_n(\F)} \left( \Phi_i \cdot W_i \right) \left( \begin{pmatrix} t & \\ & \id_{n-1} \end{pmatrix} g \right) \ud g.
\end{multline*}
	This proves $\VorH(\pi,0) \subseteq W(\pi)$.

	\subsection{$\VorH(\pi,0) = \VorH(\pi,j)$}
	
	The proof is quite similar to the non-archimedean case. We only indicate the differences in technique. Note that, the relevant main tool, i.e., the analogue of \cite[Lemma (4.1.5)]{JPS79} in the archimedean case is the specialization to $n'=1$ of \cite[Lemma 6.2]{J09}. We define $h(t; W, j)$ and $h(t; W, j, \Phi)$ the same way as in the non-archimedean case, and are reduced to proving the inclusions \eqref{eq: AuxIncl}. Only the first and the third inclusions need different argument, and the third one is simpler than the first. For the first inclusion, we apply Dixmier--Malliavin's theorem \cite[Lemma 6.1]{J09} to write
	$$ W(g) = \sum_i \int_{\F} W_i \left( g \begin{pmatrix} 1 & 0 & z & 0 \\ 0 & \id_j & 0 & 0 \\ 0 & 0 & 1 & 0 \\ 0 & 0 & 0 & \id_{n-j-2} \end{pmatrix} \right) \phi_i(z) \ud z $$
	for a finite collection of $W_i \in \Whi(\pi^{\infty}, \psi)$ and $\phi_i \in \Cont_c^{\infty}(\F)$. An application of \cite[Lemma 6.2]{J09} yields the desired $[h(t; W, j)] \subset [h(t; W, j-1, \Phi)]$.

\section{Some Extensions of Voronoi--Hankel Transforms}
\label{sec: ExtVorH}

	\subsection{Miller--Schmid Type Theory: Non-Archimedean Case}
	\label{sec: MSTheory}

	Let $\F$ be a non-archimedean local field. We extend Miller--Schmid's theory \cite{MS04} to this setting via a Paley--Wiener theory as follows.
	
\begin{definition} \label{def: SSchNA}
	(1) A \emph{finite function} on a locally compact group is a continuous function whose translates span a finite dimensional vector space.
	
\noindent (2) We write $\SSch(\F)$ for the space spanned by $\eta \cdot \Sch(\F)$ as $\eta$ runs through finite functions on $\F^{\times}$.
\end{definition}
\begin{remark}
	The space of finite functions is spanned by functions of the shape $\chi \cdot v_{\F}^k$, where $\chi$ is a quasi-character of $\F^{\times}$ and $k \in \Z_{\geq 0}$. It is clear that $\SSch(\F) \subset \Cont^{\infty}(\F^{\times})$, since $\chi \cdot v_{\F}^k \cdot \id_{\vO_{\F}} \in \Cont^{\infty}(\F^{\times})$.
\end{remark}

\begin{definition} \label{eq: SMellinNA}
	Write $\SMel(\F) = \Cont_c(\widehat{\vO_{\F}^{\times}}, \C(X))$, where $\C(X)$ is the fractional field of $\C[X]$.
\end{definition}

\begin{proposition} \label{prop: PaleyWienerNA}
	The Mellin transform for $f \in \Cont^{\infty}(\F^{\times})$ defined for $\xi \in \widehat{\vO_{\F}^{\times}}$ and $s \in \C$ by
	$$ \Mellin{f}(\xi,s) := \int_{\F^{\times}} f(t) \xi(t) \norm[t]_{\F}^s \ud^{\times} t $$
whenever the integral is absolutely convergent, induces a bijection between $\SSch(\F)$ and $\SMel(\F)$ up to the change of variables $X := q_{\F}^{-s}$.
\end{proposition}
\begin{proof}
	Clearly the Mellin transform is injective on $\SSch(\F)$. It suffices to identify its image as $\SMel(\F)$. Note that $\SSch(\F)$ is a smooth $\vO_{\F}^{\times}$-module. For every $\xi \in \widehat{\vO_{\F}^{\times}}$ we introduce
	$$ \Cont_c^{\infty}(\F^{\times}; \xi) := \left\{ f \in \Cont_c^{\infty}(\F^{\times}) \ \middle| \ f(tx) = \xi(t)f(x), \ \forall t \in \vO_{\F}^{\times} \right\}, $$
	$$ \SSch(\F, \xi) := \left\{ f \in \SSch(\F) \ \middle| \ f(tx) = \xi(t)f(x), \ \forall t \in \vO_{\F}^{\times} \right\}, $$
	then a (smooth/profinite) version of Peter--Weyl's theorem for $\vO_{\F}^{\times}$ implies
	$$ \SSch(\F) = \sideset{}{_{\xi}} \bigoplus \SSch(\F, \xi). $$
	It suffices to show that the Mellin transform maps $\SSch(\F, \xi)$ onto $\C(X)$ viewed as the subspace of functions in $\SMel(\F)$ supported in the singleton $\{ \xi^{-1} \}$. Note that the Mellin transform identifies $\Cont_c^{\infty}(\F^{\times}; \xi)$ with $\C[X,X^{-1}]$, and takes convolution on $\F^{\times}$ to multiplication. Introduce
	$$ \Sch^*(\F; \xi) := \left\{ f \in \xi \cdot\Sch(\F) \ \middle| \ f(tx) = \xi(t)f(x), \ \forall t \in \vO_{\F}^{\times} \right\} = \Cont_c^{\infty}(\F^{\times};\xi) \bigoplus \C \xi \id_{\vO_{\F}}, $$
	$$ [X]_0 := 1, \quad [X]_k := (X+1)\cdots(X+k), \ \forall k \in \Z_{\geq 1}. $$
	Note that $\Cont_c^{\infty}(\F^{\times};\xi)$ is a commutative algebra with identity $\xi \id_{\vO_{\F}^{\times}}$, and we have the equalities
	$$ \left( \xi \id_{\varpi_{\F} \vO_{\F}^{\times}} - \xi \id_{\vO_{\F}^{\times}} \right) * \left( \xi \id_{\vO_{\F}^{\times}} \right) = \xi \id_{\vO_{\F}^{\times}} \quad \Rightarrow \quad \Sch^*(\F; \xi) = \Cont_c^{\infty}(\F^{\times};\xi) *  \left( \xi \id_{\vO_{\F}} \right); $$
	$$ \norm^{s_0} (\phi * f) = \left( \norm^{s_0} \phi \right) * \left( \norm^{s_0} f \right), \quad v_{\F} (\phi * f) = (v_{\F} \phi)*f + \phi * (v_{\F} f), \quad \forall \phi \in \Cont_c^{\infty}(\F^{\times}), f \in \SSch(\F). $$
	We can therefore write
\begin{multline*} 
	\SSch(\F,\xi) = \sum_{s_0 \in \C, k \in \Z_{\geq 0}} v_{\F}^k \norm^{s_0} \Sch^*(\F;\xi) = \sum_{s_0 \in \C, k \in \Z_{\geq 0}} [v_{\F}]_k \norm^{s_0} \Cont_c^{\infty}(\F^{\times};\xi)*(\xi \id_{\vO_{\F}}) \\
	= \sum_{s_0 \in \C, k \in \Z_{\geq 0}} \Cont_c^{\infty}(\F^{\times};\xi)* \left( [v_{\F}]_k \norm^{s_0} \xi \id_{\vO_{\F}} \right).
\end{multline*}
	Writing $\beta = q_{\F}^{-s_0}$, a simple computation shows (for $\Re(s) \gg 1$)
	$$ \Mellin{[v_{\F}]_k \norm^{s_0} \xi \id_{\vO_{\F}}}(\xi^{-1},s) = \left[ X \frac{\ud}{\ud X} \right]_k \left( \frac{1}{1-\beta X} \right) = \frac{k!}{(1-\beta X)^k}. $$
	Therefore the image under the Mellin transform of $\SSch(\F,\xi)$ is identified with
	$$ \sum_{\beta \in \C^*, k \in \Z_{\geq 0}} \C[X,X^{-1}] \frac{k!}{(1-\beta X)^k} = \C(X) $$
	by the partial fractional expansion for the \textrm{PID} $\C[X,X^{-1}]$.
\end{proof}
	
	Consider a generic admissible irreducible $\GL_r(\F)$-representation $\pi$. Note that $\VorH(\pi)$ given in \eqref{eq: BasicVHS} is a $\Cont_c^{\infty}(\F^{\times})$-submodule of $\SSch(\F)$ (see \cite[Theorem (2.7)]{JPS83} and \cite[Proposition (2.2)]{JPS79}). The \emph{Voronoi--Hankel transform} in Definition \ref{def: VoronoiGLn}
\begin{equation} \label{eq: VorHInitial}
	\VorH_{\pi, \psi}: \VorH(\pi) \to \norm \VorH(\widetilde{\pi})
\end{equation}
	is an isomorphism determined by the local functional equation
\begin{equation} \label{eq: VorHLocFE}
	\Mellin{\VorH_{\pi, \psi}(h)}(\xi^{-1},-s) = \gamma(s, \pi \times \xi, \psi) \cdot \Mellin{h}(\xi,s), \quad \forall \ h \in \VorH(\pi), \xi \in \widehat{\vO_{\F}^{\times}}.
\end{equation}
	
\noindent We propose a Miller--Schmid type extension of $\VorH_{\pi,\psi}$ to functions with simple singularities at the infinity.
	
\begin{proposition} \label{prop: MillerSchmidtNA}
	(1) For any $h \in \SSch(\F)$ there is a unique $\MS_{\pi,\psi} \circ \Inv(h) \in \SSch(\F)$ such that
	$$ \Mellin{\MS_{\pi, \psi}(\Inv(h))}(\xi^{-1},-s) = \gamma(s, \pi \times \xi, \psi) \cdot \Mellin{\Inv(h)}(\xi,s), \quad \forall \ \xi \in \widehat{\vO_{\F}^{\times}}. $$
The Mellin transforms on both sides are absolutely convergent for $\Re(s) \ll -1$. 

\noindent (2) The two transforms $\VorH_{\pi,\psi}$ and $\MS_{\pi,\psi}$ coincide on $\VorH(\pi) \cap \Inv(\SSch(\F)) = \Cont_c^{\infty}(\F^{\times})$.
\end{proposition}
\begin{proof}
	(1) The existence and uniqueness of $\MS_{\pi,\psi} \circ \Inv(h) \in \SSch(\F)$ is a direct consequence of Proposition \ref{prop: PaleyWienerNA}: the right hand side is absolutely convergent for $\Re(s) \ll -1$ and defines a rational function in $X$. 

\noindent (2) It suffices to note that for $h \in \Cont_c^{\infty}(\F^{\times})$ the Mellin transform $\Mellin{h}(\xi,s)$ is absolutely convergent for all $s \in \C$.
\end{proof}

\begin{remark} \label{rmk: MSArch}
	For archimedean $\F \in \{ \R, \C \}$, we have analogues of $\SMel(\F)$ (see Definition \ref{eq: SMellinNA}). The precise definitions and the corresponding Paley--Wiener theories can be found in \cite[Definition 6.21 \& Corollary 6.39]{MS04} for the real case and \cite[Lemma 2.8]{Qi20} for the complex case. The analogues of Proposition \ref{prop: MillerSchmidtNA} can be stated in the same way, or stated in some different but equivalent way in \cite[Lemma 6.19]{MS04} for the real case and in \cite[Theorem 3.12]{Qi20} for the complex case.
\end{remark}

	\subsection{Extended Voronoi--Hankel Transforms}
	
	We first recall the Godement--Jacquet theory in the local setting. We follow \cite{GoJ11} with a slight modification. Namely, we present the theory with $\beta^{\iota}(g) = \beta(g^{\iota})$ instead of $\check{\beta}(g) = \beta(g^{-1})$.
	
	Let $\F$ be a local field. Let $\pi$ be an irreducible and \emph{generic} representation of $\GL_n(\F)$, with dual representation $\widetilde{\pi}$ (also called \emph{contra-gredient representation}). Let $\Pairing{\cdot}{\cdot}$ be the natural pairing on $V_{\pi} \times V_{\widetilde{\pi}}$. A \emph{smooth matrix coefficient} $\beta$ of $\pi$ is a function on $\GL_n(\F)$ so that for some $v \in V_{\pi}^{\infty}$ and $\widetilde{v} \in V_{\widetilde{\pi}}^{\infty}$ we have $\beta(g) = \Pairing{\pi(g).v}{\widetilde{v}}$. The set of smooth matrix coefficients is denoted by $C(\pi^{\infty})$.
\begin{theorem} \label{thm: GJLoc}
	For any Schwartz--Bruhat function $\Phi \in \Sch(\Mat_n(\F))$ and any $\beta \in C(\pi^{\infty})$, the Godement--Jacquet zeta function is defined by
	$$ \Zeta(s, \Phi, \beta) := \int_{\GL_n(\F)} \Phi(g) \beta(g) \norm[\det g]_{\F}^{s+\frac{n-1}{2}} \ud g. $$
\begin{itemize}
	\item[(1)] The above integral defining $\Zeta(s, \Phi, \beta)$ is absolutely convergent in $\Re s > s_0$ for some $s_0 \in \R$.
	\item[(2)] The function $s \mapsto \Zeta(s, \Phi, \beta)$ has a meromorphic continuation to $s \in \C$, and the quotient function $s \mapsto \Zeta(s, \Phi, \beta)/L(s,\pi)$ is entire. If $\F$ is archimedean, then the function $s \mapsto \Zeta(s, \Phi, \beta)$ is rapidly decreasing in any region $a \leq \Re(s) \leq b$ away from the possible poles.
	\item[(3)] We have the local functional equation
	$$ \frac{\Zeta(1-s, \widehat{\Phi}, \beta^{\iota})}{L(1-s, \widetilde{\pi})} = \varepsilon(s, \pi, \psi) \frac{\Zeta(s, \Phi, \beta)}{L(s, \pi)}, \quad \text{or} \quad \Zeta(1-s, \widehat{\Phi}, \beta^{\iota}) = \gamma(s, \pi, \psi) \Zeta(s, \Phi, \beta). $$
\end{itemize}
\end{theorem}
\begin{proof}
	See \cite[Theorem 15.4.4]{GoJ11} for the non-archimedean case. As for the archimedean case, \cite[Theorem 15.9.1]{GoJ11} is a weaker version where $\Sch(\Mat_n(\F))$ is replaced by the subspace of \emph{standard} Schwartz functions $\Sch_0(\Mat_n(\F))$, and $C(\pi^{\infty})$ is replaced with its $\gp{K}_n$-finite subset. The proof of \cite[Proposition 4.4]{J09} shows the equivalence and implies our version: any relevant $\pi$ is a subrepresentation of a principal one. Note that the denseness of $\Sch_0(\Mat_n(\F))$ can be found in \cite[Exercise \Rmnum{3}.6 \& \Rmnum{3}.7]{HT92}.
\end{proof}
\begin{remark} \label{rmk: FourTransConv}
	We take the convention of Fourier transforms in \cite[(0.2)]{JPS79} as
	$$ \invOFour(\Phi)(X) = \widehat{\Phi}(X) := \int_{\Mat_n(\F)} \Phi(Y) \psi \left( \Tr (X Y^T) \right) \ud Y, $$
	since it is with this convention that the consistency of the Godement--Jacquet theory and the $\GL_n \times \GL_1$ Rankin--Selberg theory has been checked, i.e., they give the same gamma factors.
\end{remark}

\begin{proposition} \label{prop: GJZAbsCTemp}
	If $\pi$ is unitary and \emph{tempered}, then the integral $\Zeta(s, \Phi, \beta)$ is absolutely convergent in $\Re(s) > 0$ for any $\beta \in C(\pi^{\infty})$.
\end{proposition}
\begin{proof}
	This is essentially an easy consequence of the decay of matrix coefficient
\begin{equation} \label{eq: DecayMCSmooth}
	\norm[\beta(g)] \ll \Xi_n(g)
\end{equation}
 	where $\Xi_n$ is the Harish-Chandra's Xi-function for $\GL_n(\F)$. The estimation \eqref{eq: DecayMCSmooth} can be found in \cite[Theorem 1.2]{Sun09} (see also \cite[Theorem 1.1]{He09}), extending the $\gp{K}$-finite case \cite[Theorem 2]{CHH88}. Note that 
	$$ \Xi_n(g) = \Pairing{\Pi_0(g).e_0}{e_0} $$
	is the matrix coefficient of a unitary spherical vector $e_0$ in the representation $\Pi_0 = \Ind_{\gp{B}_n(\F)}^{\GL_n(\F)} \id$ parabolically induced from the trivial character of the Borel subgroup $\gp{B}_n(\F)$.  To conclude, an iterated application of \cite[(15.7.13)]{GoJ11} together with the Weil--Schwartz envelop shows for any $\sigma > 0$
	$$ \int_{\GL_n(\F)} \extnorm{\Phi(g)} \cdot \Xi_n(g) \cdot \norm[\det g]_{\F}^{\sigma+\frac{n-1}{2}} \ud g \ll \int_{(\F^{\times})^{\oplus n}} \phi(t_1,\dots,t_n) \sideset{}{_{j=1}^n} \prod \norm[t_j]_{\F}^{\sigma} \ud^{\times} t_j < \infty $$
	for some positive Schwartz-Bruhat function $\phi \in \Sch(\F^n)$.
\end{proof}

\begin{proposition} \label{prop: BdWhiTemp}
	Let $\pi$ be unitary and \emph{tempered}. For any $W \in \Whi(\pi, \psi)$, there is a Schwartz--Bruhat function $0< \phi \in \Sch(\F^{n-1})$ so that
	$$ \extnorm{W \left( n \begin{pmatrix} a_1 & & \\ & \ddots & \\ & & a_n \end{pmatrix} \kappa \right)} \leq \prod_{j=1}^{n-1} \norm[t_j]_{\F}^{\frac{j(n-j)}{2}} \left( 1+ (\log \norm[t_j]_{\F})^2 \right)^d \cdot \phi(t_1,\dots,t_{n-1}) $$
for any $a_j = t_j \cdots t_n \in \F^{\times}$, $n \in \gp{N}_n(\F)$, $\kappa \in \gp{K}_n$, and some $d \in \Z_{\geq 0}$ depending only on $\pi$.
\end{proposition}
\begin{proof}
	In the archimedean case, the desired bound is a consequence of \cite[Proposition 3.5]{J09} and the Weil--Schwartz envelop. Alternatively we may derive it from \cite[Theorem 2]{J04} and \cite[Proposition 3.5]{J09}. In the non-archimedean case, we may take $a_n=1$ by the unitarity of the central character and $n=\kappa=\id$ by the $\gp{K}_n$-finiteness of $W$. By \cite[Proposition (2.2)]{JPS79} we have
	$$ W \left( \begin{pmatrix} a_1 & & \\ & \ddots & \\ & & a_n \end{pmatrix} \right) = \omega_{\pi}(a_n) \sum_{\lambda} \lambda(t_1,\dots,t_{n-1}) \phi_{\lambda}(t_1,\dots,t_{n-1}) $$
for a finite number of \emph{finite functions} $\lambda$ and some $\phi_{\lambda} \in \Sch(\F^{n-1})$. We may assume that each $\lambda$ is \emph{decomposable} in the sense that $\lambda(t_1,\dots,t_{n-1}) = \lambda_1(t_1) \cdots \lambda_{n-1}(t_{n-1})$ for finite functions $\lambda_j$ on $\F^{\times}$. The bound \cite[Proposition (2.5)]{JS83} implies that the exponent of $\lambda_j$ is $\leq j(n-j)/2$, namely 
	$$ \norm[\lambda_j(t_j)] \ll \norm[t_j]_{\F}^{\frac{j(n-j)}{2}-\epsilon}, \quad \forall \ \epsilon > 0. $$
	We conclude since $\lambda_j$ is a sum of products of a quasi-character and a power of the additive valuation.
\end{proof}

\begin{corollary} \label{cor: RSAbsCTemp}
	Let $\pi$ be unitary and \emph{tempered} representation. For any $h \in \VorH(\pi)$ and $\xi \in \widehat{\F^{\times}}$ the Mellin transform $\Mellin{h}(\xi,s)$ is absolutely convergent in $\Re(s) > 0$.
\end{corollary}

\begin{remark} \label{rmk: NonTempBd}
	It should be possible to extend the estimation \eqref{eq: DecayMCSmooth} to the non-tempered case, based on the work of Oh \cite{Oh02}, and a strategy in the rank one case of Venkatesh \cite[Lemma 9.1]{Ve10} which he attributes to Shalom \cite{Shal00}. Similarly, the bound of the Whittaker functions in Proposition \ref{prop: BdWhiTemp} should also be extensible. We believe that the absolute convergence region should be $\Re(s) > \RamCst$ in both Proposition \ref{prop: GJZAbsCTemp} and Corollary \ref{cor: RSAbsCTemp} for a unitary and $\RamCst$-tempered representation.
\end{remark}

\begin{proof}[Proof of Theorem \ref{thm: ExtVorH}]
	First consider the case of a general irreducible smooth and generic $\pi$. Take any $\Phi \in \Cont_c^{\infty}(\GL_n(\F))$, $\beta \in C(\pi^{\infty})$ and $h \in \SSch(\F)$. By the Plancherel for Mellin transform, we have
\begin{align*}
	&\quad \int_{\GL_n(\F)} \widehat{\Phi}(g) \beta^{\iota}(g) \norm[\det g]_{\F}^{\frac{n+1}{2}} \cdot \Inv(h)(\det g) \ud g \\
	&= \sideset{}{_{\xi \in \widehat{\Omega_{\F}}}} \sum \int_{\sigma+i\R(\F)} \left( \int_{\GL_n(\F)} \widehat{\Phi}(g) \beta^{\iota}(g) \xi(\det g)^{-1} \norm[\det g]_{\F}^{1-s+\frac{n-1}{2}} \ud g \right) \left( \int_{\F^{\times}} h(t^{-1}) \xi(t) \norm[t]_{\F}^s \ud^{\times} t \right) \ud_{\F} s \\
	&= \sideset{}{_{\xi \in \widehat{\Omega_{\F}}}} \sum \int_{\sigma+i\R(\F)} \Zeta(1-s, \widehat{\Phi}, (\beta \otimes \xi)^{\iota}) \cdot \Mellin{\Inv(h)}(\xi,s) \ud_{\F} s \\
	&= \sideset{}{_{\xi \in \widehat{\Omega_{\F}}}} \sum \int_{\sigma+i\R(\F)} \Zeta(s, \Phi, \beta \otimes \xi) \cdot \Mellin{\MS_{\pi,\psi}(\Inv(h))}(\xi^{-1},-s) \ud_{\F} s \quad (\text{for } \sigma \ll -1),
\end{align*}
	where we have applied the local functional equations in Proposition \ref{prop: MillerSchmidtNA} and Theorem \ref{thm: GJLoc}. Since $\Phi \in \Cont_c^{\infty}(\GL_n(\F))$ the zeta integral $\Zeta(s, \Phi, \beta \otimes \xi)$ is absolutely convergent for any $s \in \C$. Applying the Plancherel for Mellin transform again we get
\begin{equation} \label{eq: VH-MS-Comp}
	\int_{\GL_n(\F)} \widehat{\Phi}(g) \beta^{\iota}(g) \norm[\det g]_{\F}^{\frac{n+1}{2}} \cdot \Inv(h)(\det g) \ud g = \int_{\GL_n(\F)} \Phi(g) \beta(g) \norm[\det g]_{\F}^{\frac{n-1}{2}} \cdot \MS_{\pi,\psi}(\Inv(h))(\det g) \ud g. 
\end{equation}
	To see that \eqref{eq: VH-MS-Comp} characterizes $\MS_{\pi,\psi}(\Inv(h))$, take any $t_0 \in \F^{\times}$ and $g_0 \in \GL_n(\F)$ with $t_0 = \det g_0$. There is always a $\beta \in C(\pi^{\infty})$ such that $\beta(g_0) \neq 0$, since $C(\pi^{\infty})$ is stable by translation by $\GL_n(\F)$. Then $g \mapsto \Phi(g) \beta(g) \norm[\det g]_{\F}^{\frac{n-1}{2}}$ can be any smooth function of compact support in any neighborhood of $g_0$ on which $\beta \neq 0$. Therefore \eqref{eq: VH-MS-Comp} characterizes $\MS_{\pi,\psi}(\Inv(h))$ in small neighborhoods of any $t_0 \in \F^{\times}$, since $\GL_n(\F) \simeq \F^{\times} \times \SL_n(\F)$ as topological spaces, hence it uniquely determines $\MS_{\pi,\psi}(\Inv(h))$.
	
	Now consider the special case of unitary and tempered $\pi$. Replacing $\Inv(h)$ with any $h \in \VorH(\pi)$, $\Phi \in \Cont_c^{\infty}(\GL_n(\F))$ with any $\Phi \in \Sch(\Mat_n(\F))$, and re-taking the above argument, we see that
\begin{align*}
	&\quad \int_{\GL_n(\F)} \widehat{\Phi}(g) \beta^{\iota}(g) \norm[\det g]_{\F}^{\frac{n+1}{2}} \cdot h(\det g) \ud g \\
	&= \sideset{}{_{\xi \in \widehat{\Omega_{\F}}}} \sum \int_{\sigma+i\R(\F)} \Zeta(1-s, \widehat{\Phi}, (\beta \otimes \xi)^{\iota}) \cdot \Mellin{h}(\xi,s) \ud_{\F} s \\
	&= \sideset{}{_{\xi \in \widehat{\Omega_{\F}}}} \sum \int_{\sigma+i\R(\F)} \Zeta(s, \Phi, \beta \otimes \xi) \cdot \Mellin{\VorH_{\pi,\psi}(h)}(\xi^{-1},-s) \ud_{\F} s
\end{align*}
	holds for $0 < \sigma < 1$ by Proposition \ref{prop: GJZAbsCTemp} and Corollary \ref{cor: RSAbsCTemp}. Therefore we obtain
\begin{equation} \label{eq: ExtVHComp}
	\int_{\GL_n(\F)} \widehat{\Phi}(g) \beta^{\iota}(g) \norm[\det g]_{\F}^{\frac{n+1}{2}} \cdot h(\det g) \ud g = \int_{\GL_n(\F)} \Phi(g) \beta(g) \norm[\det g]_{\F}^{\frac{n-1}{2}} \cdot \VorH_{\pi,\psi}(h)(\det g) \ud g
\end{equation}
	and conclude the proof.
\end{proof}

	\subsection{Multiplicativity in Tempered Case}
	\label{sec: Mult}

	We first draw some a consequence of Theorem \ref{eq: EquivGamma}.

\begin{corollary} \label{cor: VHTransBis}
	Let $h$ be given via $\Phi \in \Sch(\Mat_n(\F))$ and $\beta \in C(\pi^{\infty})$ as
	$$ h(x) = \norm[x]_{\F}^{\frac{n-1}{2}} \int_{\SL_n(\F)} (\Phi \cdot \beta) \left( \begin{pmatrix} x & \\ & \id_{n-1} \end{pmatrix} g \right) \ud g. $$
	Then we have
	$$ \VorH_{\pi}(h)(y) = \norm[y]_{\F}^{\frac{n+1}{2}} \int_{\SL_n(\F)} (\widehat{\Phi} \cdot \beta^{\iota}) \left( \begin{pmatrix} y & \\ & \id_{n-1} \end{pmatrix} g \right) \ud g $$
	where $\beta^{\iota}(g) := \beta(g^{\iota})$ and the ``hat'' Fourier transform is defined by
\begin{equation} \label{eq: HatFourTrans}
	\widehat{\Psi}(X) = \invOFour_{\psi}(\Psi)(X) = \int_{\Mat(n \times m, \F)} \Psi(Y) \psi \left( \Tr(XY^T) \right) \ud Y. 
\end{equation}
\end{corollary}
\begin{proof}
	By the local functional equation of Godement--Jacquet zeta-integrals we have
\begin{multline} \label{eq: GJLocFE} 
	\int_{\GL_n(\F)} (\widehat{\Phi} \cdot \beta^{\iota})(g) \chi^{-1}(\det g) \norm[\det g]_{\F}^{1-s+\frac{n-1}{2}} \ud g = \\
	\gamma(s, \pi \times \chi, \psi) \int_{\GL_n(\F)} (\Phi \cdot \beta)(g) \chi(\det g) \norm[\det g]_{\F}^{s+\frac{n-1}{2}} \ud g.
\end{multline}
	Comparing \eqref{eq: GJLocFE} with \eqref{eq: LocFEGLnGL1} and \eqref{eq: LocFEGL1} we get
	$$ \int_{\GL_n(\F)} (\widehat{\Phi} \cdot \beta^{\iota})(g) \chi^{-1}(\det g) \norm[\det g]_{\F}^{1-s+\frac{n-1}{2}} \ud g = \int_{\F^{\times}} \VorH_{\pi}(h)(t) \chi^{-1}(t) \norm[t]_{\F}^{-s} \ud^{\times} t, $$
	both sides being absolutely convergent for $\Re(s) \ll -1$. We conclude the desired equality by Mellin inversion in this absolutely convergent region.
\end{proof}

%

	We shall use the above alternative definition of $\VorH(\pi) = S(\pi)$ to prove Theorem \ref{thm: Mult}. As in \cite[(15.7.2)]{GoJ11} we can find smooth functions satisfying
	$$ f: \GL_n(\F) \to V_{\pi_1}^{\infty} \otimes V_{\pi_2}^{\infty}, \quad f \left( \begin{pmatrix} g_1 & X \\ & g_2 \end{pmatrix} g \right) = \norm[\det g_1]_{\F}^{\frac{n_2}{2}} \norm[\det g_2]_{\F}^{-\frac{n_1}{2}} (\pi_1 \otimes \pi_2)(g_1,g_2).f(g), $$
	$$ \widetilde{f}: \GL_n(\F) \to V_{\widetilde{\pi}_1}^{\infty} \otimes V_{\widetilde{\pi}_2}^{\infty}, \quad \widetilde{f} \left( \begin{pmatrix} g_1 & X \\ & g_2 \end{pmatrix} g \right) = \norm[\det g_1]_{\F}^{\frac{n_2}{2}} \norm[\det g_2]_{\F}^{-\frac{n_1}{2}} (\widetilde{\pi}_1 \otimes \widetilde{\pi}_2)(g_1,g_2).\widetilde{f}(g), $$
	so that the matrix coefficient $\beta \in C(\pi^{\infty})$ has the form
	$$ \beta(g) = \int_{\gp{K}_n} \Pairing{f(\kappa g)}{\widetilde{f}(\kappa)} \ud \kappa, $$
	where the group $\gp{K}_n$ is the standard maximal (connected) compact subgroup of $\GL_n(\F)$. For $\kappa_j \in \gp{K}_n$ we introduce
	$$ \kappa_2.\Phi.\kappa_1(X) := \Phi(\kappa_1 X \kappa_2), \quad \beta_{\kappa_1,\kappa_2}(g_1,g_2) := \Pairing{(\pi_1 \otimes \pi_2)(g_1,g_2).f(\kappa_2)}{\widetilde{f}(\kappa_1)}; $$
	$$ T(g_1,g_2; \Phi) := \int_{\Mat(n_1 \times n_2, \F)} \Phi \left( \begin{pmatrix} g_1 & X \\ & g_2 \end{pmatrix} \right) \ud X \in \Sch(\Mat_{n_1}(\F) \times \Mat_{n_2}(\F)). $$
\begin{lemma} \label{lem: BdMCIndRep}
	We have the following bound valid uniformly in $\kappa_1, \kappa_2 \in \gp{K}_n$
	$$ \extnorm{\beta_{\kappa_1,\kappa_2}(g_1,g_2)} \ll \Xi_{n_1}(g_1) \Xi_{n_2}(g_2), $$
	where $\Xi_n$ is the Harish-Chandra's function for $\GL_n(\F)$ (see \cite[\S 4.5.3]{Wal88}).
\end{lemma}
\begin{proof}
	The stated bound is valid for any fixed $\kappa_1, \kappa_2 \in \gp{K}_n$ by (the smooth version of) the decay of matrix coefficients \cite{Sun09}. Its uniform extension to all $\kappa_1, \kappa_2 \in \gp{K}_n$ is obvious in the non-archimedean case since smoothness of $f$ and $\widetilde{f}$ is the same as $\gp{K}_n$-finiteness. In the archimedean case, we note that the implied constant in \cite[Theorem 1.2]{Sun09} is some Sobolev norm $\Sob(\cdot)$. Hence we get $\extnorm{\beta_{\kappa_1,\kappa_2}(g_1,g_2)} \leq \Sob(f(\kappa_2)) \Sob(\widetilde{f}(\kappa_1)) \Xi_{n_1}(g_1) \Xi_{n_2}(g_2)$. By smoothness, the images of $\kappa_2 \mapsto \Sob(f(\kappa_2))$ and $\kappa_1 \to \Sob(\widetilde{f}(\kappa_1))$ are compact, hence bounded, proving the stated bound in the archimedean case.
\end{proof}
	
\noindent Computation similar to those in \cite[\S 15.7]{GoJ11} leads to
\begin{multline} \label{eq: MultTFDep}
	h(t) = \norm[t]_{\F}^{\frac{n-1}{2}} \cdot \int_{\substack{\GL_n(\F) \\ \det g = t}} \Phi(g) \beta(g) \ud g \\
	= \int_{\substack{\gp{K}_n^2 \times \GL_{n_1}(\F) \times \GL_{n_2}(\F) \\ \det(\kappa_1^{-1}\kappa_2) \det(g_1) \det(g_2) = t}} T(g_1,g_2; \kappa_2.\Phi.\kappa_1^{-1}) \beta_{\kappa_1,\kappa_2}(g_1,g_2) \norm[\det g_1]_{\F}^{\frac{n_1-1}{2}} \norm[\det g_2]_{\F}^{\frac{n_2-1}{2}} \ud g_1 \ud g_2 \ud \kappa_1 \ud \kappa_2.
\end{multline}
	Consider the function
	$$ h_1(t; g_2, \kappa_1, \kappa_2) := \int_{\substack{\GL_{n_1}(\F) \\ \det(g_1) = t}} T(g_1,g_2; \kappa_2.\Phi.\kappa_1^{-1}) \beta_{\kappa_1,\kappa_2}(g_1,g_2) \norm[\det g_1]_{\F}^{\frac{n_1-1}{2}} \norm[\det g_2]_{\F}^{\frac{n_2-1}{2}} \ud g_1. $$
	We have $t \mapsto h_1(t; g_2, \kappa_1, \kappa_2) \in \VorH(\pi_1)$ for any given $g_2, \kappa_1, \kappa_2$. By Corollary \ref{cor: VHTransBis} its image under $\VorH_{\pi_1}$ is
	$$ h_1^*(t; g_2, \kappa_1, \kappa_2) := \int_{\substack{\GL_{n_1}(\F) \\ \det(g_1) = t}} \invOFour_1(T)(g_1,g_2; \kappa_2.\Phi.\kappa_1^{-1}) \beta_{\kappa_1,\kappa_2}(g_1^{\iota},g_2) \norm[\det g_1]_{\F}^{\frac{n_1+1}{2}} \norm[\det g_2]_{\F}^{\frac{n_2-1}{2}} \ud g_1, $$
	where $\invOFour_1$ is the partial inverse $\psi$-Fourier transform with respect to the variable $g_1$. Introduce $\delta = \delta(g_2,\kappa_1,\kappa_2) := \det(g_2)^{-1} \det(\kappa_1\kappa_2^{-1})$. The image of the function $t \mapsto h_1(t \delta; g_2, \kappa_1, \kappa_2)$ under $\VorH_{\pi_1}$ is $t \mapsto h_1^*(t \delta^{-1}; g_2, \kappa_1, \kappa_2)$ by the local functional equation \eqref{eq: LocFEGLnGL1}. Consequently for any $\Phi_1 \in \Sch(\Mat_{n_1}(\F))$ and $\beta_1 \in C(\pi_1^{\infty})$ we have
\begin{multline} \label{eq: 1stExtVHId}
	\norm[\delta]_{\F}^{-\frac{n_1+1}{2}}\int_{\substack{\GL_{n_1}(\F)^{\times 2} \\ \delta \det g_1' = \det g_1}} \widehat{\Phi}_1(g_1') \beta_1^{\iota}(g_1') T(g_1,g_2; \kappa_2.\Phi.\kappa_1^{-1}) \beta_{\kappa_1,\kappa_2}(g_1,g_2) \norm[\det g_1]_{\F}^{n_1} \norm[\det g_2]_{\F}^{\frac{n_2-1}{2}} \ud g_1 \ud g_1' \\
	= \int_{\GL_{n_1}(\F)} \widehat{\Phi}_1(g_1') \beta_1^{\iota}(g_1') \norm[\det g_1']_{\F}^{\frac{n_1+1}{2}} h_1(\delta \det g_1'; g_2, \kappa_1, \kappa_2) \ud g_1' \\
	= \int_{\GL_{n_1}(\F)} \Phi_1(g_1') \beta_1(g_1') \norm[\det g_1']_{\F}^{\frac{n_1-1}{2}} h_1^*(\delta^{-1} \det g_1'; g_2, \kappa_1, \kappa_2) \ud g_1' \\
	= \norm[\delta]_{\F}^{\frac{n_1-1}{2}}\int_{\substack{\GL_{n_1}(\F)^{\times 2} \\ \delta^{-1} \det g_1' = \det g_1}} \Phi_1(g_1') \beta_1(g_1') \invOFour_1(T)(g_1,g_2; \kappa_2.\Phi.\kappa_1^{-1}) \beta_{\kappa_1,\kappa_2}(g_1^{\iota},g_2) \norm[\det g_1]_{\F}^{n_1} \norm[\det g_2]_{\F}^{\frac{n_2-1}{2}} \ud g_1 \ud g_1'.
\end{multline}

\begin{lemma} \label{lem: 1stExtVHAux}
	Both sides of \eqref{eq: 1stExtVHId} are integrable for $\ud g_2 \ud \kappa_1 \ud \kappa_2$ on $\GL_{n_2}(\F) \times \gp{K}_n \times \gp{K}_n$.
\end{lemma}
\begin{proof}
	Elementary estimation for Schwartz functions (e.g. \cite[\S 2]{Wu24+}), shows the existence of some positive $\widetilde{\Phi} \in \Sch(\Mat_{n_1}(\F) \times \Mat_{n_2}(\F))$ so that uniformly in $\kappa_1,\kappa_2$ we have
	$$ T(g_1,g_2; \kappa_2.\Phi.\kappa_1^{-1}) \ll \widetilde{\Phi}(g_1,g_2). $$
	Together with Lemma \ref{lem: BdMCIndRep} we deduce
	$$ h_1(t; g_2, \kappa_1, \kappa_2) \ll \norm[t]_{\F}^{\frac{n_1-1}{2}} \int_{\substack{\GL_{n_1}(\F) \\ \det(g_1) = t}} \widetilde{\Phi}(g_1,g_2) \Xi_{n_1}(g_1) \ud g_1 \cdot \Xi_{n_2}(g_2) \norm[\det g_2]_{\F}^{\frac{n_2-1}{2}}. $$
	By Theorem \ref{thm: EquivDefs} (1), the first component lies in $\VorH(\id^{\boxplus n_1})$ since $\Xi_{n_1}(\cdot)$ is a matrix coefficient of $\id^{\boxplus n_1}$. We can thus appeal to the general bound of Whittaker functions for tempered representations in Proposition \ref{prop: BdWhiTemp} and deduce the existence of some positive $\phi \in \Sch(\F \times \Mat_{n_2}(\F))$ and $d \in \Z_{\geq 0}$ such that
	$$ \norm[t]_{\F}^{\frac{n_1-1}{2}} \int_{\substack{\GL_{n_1}(\F) \\ \det(g_1) = t}} \widetilde{\Phi}(g_1,g_2) \Xi_{n_1}(g_1) \ud g_1 \ll \left[ 1 + \left( \log \norm[t]_{\F} \right)^2 \right]^d \phi(t, g_2). $$
	We also have for any $0 < \epsilon < 1$ the existence of some $0<\widetilde{\phi} \in \Sch(\Mat_{n_2}(\F))$ such that
	$$ \phi(t, g_2) \ll \norm[t]_{\F}^{-\epsilon} \widetilde{\phi}(g_2). $$
	Inserting these bounds we see
\begin{multline*}
	\int_{\GL_{n_1}(\F) \times \GL_{n_2}(\F) \times \gp{K}_n \times \gp{K}_n} \widehat{\Phi}_1(g_1') \beta_1^{\iota}(g_1') \norm[\det g_1']_{\F}^{\frac{n_1+1}{2}} h_1(\delta \det g_1'; g_2, \kappa_1, \kappa_2) \ud g_1' \ud g_2 \ud \kappa_1 \ud \kappa_2 \\
	\ll \int_{\GL_{n_1}(\F) \times \GL_{n_2}(\F) \times \gp{K}_n \times \gp{K}_n} \extnorm{\widehat{\Phi}_1(g_1')} \Xi_{n_1}(g_1') \cdot \norm[\det g_1']_{\F}^{\frac{n_1+1}{2}} \cdot \left[ 1 + \left( \log (\norm[\det g_1']_{\F}) - \log (\norm[\det g_2]_{\F}) \right)^2 \right]^d \cdot \\
	\phi \left( \det g_1' \cdot (\det g_2)^{-1} \cdot \det(\kappa_1 \kappa_2^{-1}), g_2 \right) \Xi_{n_2}(g_2) \norm[\det g_2]_{\F}^{\frac{n_2-1}{2}} \ud g_1' \ud g_2 \ud \kappa_1 \ud \kappa_2 \\
	\ll \sum_{d_1,d_2} \int_{\GL_{n_1}(\F)} \extnorm{\widehat{\Phi}_1(g_1')} \Xi_{n_1}(g_1') \norm[\det g_1']_{\F}^{\frac{n_1+1}{2}-\epsilon} \extnorm{ \log (\norm[\det g_1']_{\F}) }^{d_1} \ud g_1 \cdot \\
	\int_{\GL_{n_2}(\F)} \widetilde{\phi}(g_2) \Xi_{n_2}(g_2) \norm[\det g_2]_{\F}^{\frac{n_2-1}{2}+\epsilon} \extnorm{ \log (\norm[\det g_2]_{\F}) }^{d_2} \ud g_2 < \infty,
\end{multline*}
	where $(d_1,d_2)$ runs through a finite subset of $\Z_{\geq 0} \times \Z_{\geq 0}$ depending on $d$, and we have applied Proposition \ref{prop: GJZAbsCTemp} to bound the integrals in the last two lines. The integrability of the other side of \eqref{eq: 1stExtVHId} is quite similar and left as an exercise to the reader.
\end{proof}

\noindent Integrating both sides of \eqref{eq: 1stExtVHId} we get
\begin{multline}
	\int_{\GL_{n_1}(\F)} \widehat{\Phi}_1(g_1') \beta_1^{\iota}(g_1') \norm[\det g_1']_{\F}^{\frac{n_1+1}{2}} \cdot h(\det g_1') \ud g_1' = \\
	\int_{\GL_{n_1}(\F)} \Phi_1(g_1') \beta_1(g_1') \norm[\det g_1']_{\F}^{\frac{n_1-1}{2}} \cdot h^{\sharp}(\det g_1') \ud g_1'
\end{multline}
	for the function $h^{\sharp}(\cdot)$ defined by (absolutely convergent integral)
\begin{multline*} 
	h^{\sharp}(t) := \\
	\int_{\substack{\gp{K}_n^2 \times \GL_{n_1}(\F) \times \GL_{n_2}(\F) \\ \det(\kappa_1\kappa_2^{-1}) \det(g_1) \det(g_2)^{-1} = t}} \invOFour_1(T)(g_1,g_2; \kappa_2.\Phi.\kappa_1^{-1}) \beta_{\kappa_1,\kappa_2}(g_1^{\iota},g_2) \norm[\det g_1]_{\F}^{\frac{n_1+1}{2}} \norm[\det g_2]_{\F}^{\frac{n_2-1}{2}} \ud g_1 \ud g_2 \ud \kappa_1 \ud \kappa_2. 
\end{multline*}
 	Therefore we get $\widetilde{\VorH}_{\pi_1}(h) = h^{\sharp}$ by Theorem \ref{thm: ExtVorH}. We deduce
\begin{multline*}
	h^{\sharp}(t^{-1}) = \\
	\int_{\substack{\gp{K}_n^2 \times \GL_{n_1}(\F) \times \GL_{n_2}(\F) \\ \det(g_2) = t\det(\kappa_1\kappa_2^{-1}) \det(g_1)}} \invOFour_1(T)(g_1,g_2; \kappa_2.\Phi.\kappa_1^{-1}) \beta_{\kappa_1,\kappa_2}(g_1^{\iota},g_2) \norm[\det g_1]_{\F}^{\frac{n_1+1}{2}} \norm[\det g_2]_{\F}^{\frac{n_2-1}{2}} \ud g_1 \ud g_2 \ud \kappa_1 \ud \kappa_2.
\end{multline*}
	We exchange the roles of $g_2$ and $g_1$, and consider the function
	$$ h_2(t; g_1, \kappa_1, \kappa_2) := \int_{\substack{\GL_{n_2}(\F) \\ \det(g_2) = t}} \invOFour_1(T)(g_1,g_2; \kappa_2.\Phi.\kappa_1^{-1}) \beta_{\kappa_1,\kappa_2}(g_1^{\iota},g_2) \norm[\det g_1]_{\F}^{\frac{n_1+1}{2}} \norm[\det g_2]_{\F}^{\frac{n_2-1}{2}} \ud g_2. $$
	We have $t \mapsto h_2(t; g_1, \kappa_1, \kappa_2) \in \VorH(\pi_2)$ for any given $g_1, \kappa_1, \kappa_2$. By Corollary \ref{cor: VHTransBis} its image under $\VorH_{\pi_2}$ is
\begin{multline*} 
	h_2^*(t; g_1, \kappa_1, \kappa_2) := \int_{\substack{\GL_{n_2}(\F) \\ \det(g_2) = t}} \widehat{T}(g_1,g_2; \kappa_2.\Phi.\kappa_1^{-1}) \beta_{\kappa_1,\kappa_2}(g_1^{\iota},g_2^{\iota}) \norm[\det g_1]_{\F}^{\frac{n_1+1}{2}} \norm[\det g_2]_{\F}^{\frac{n_2+1}{2}} \ud g_2 \\
	= \int_{\substack{\GL_{n_2}(\F) \\ \det(g_2) = t}} T(g_1^T,g_2^T; \kappa_1.{}^t\widehat{\Phi}.\kappa_2^{-1}) \beta_{\kappa_1,\kappa_2}(g_1^{\iota},g_2^{\iota}) \norm[\det g_1]_{\F}^{\frac{n_1+1}{2}} \norm[\det g_2]_{\F}^{\frac{n_2+1}{2}} \ud g_2,
\end{multline*}
	where $\widehat{T} = \invOFour_2 \invOFour_1(T)$ is the inverse $\psi$-Fourier transform of $T$ in $(g_1,g_2) \in \Mat_{n_1}(\F) \times \Mat_{n_2}(\F)$, and we have applied the following functional equation (which is \cite[Lemma 15.7.12]{GoJ11} in the non-archimedean case, but the proof is valid for archimedean $\F$ as well; note the adaptation to our definition \eqref{eq: HatFourTrans}):
\begin{equation} \label{eq: TecFEIndRep}
	\widehat{T}(g_1,g_2; \Phi) = T(g_1^T,g_2^T; {}^t\widehat{\Phi}), \quad {}^t\widehat{\Phi}(X) := \widehat{\Phi}(X^T).
\end{equation}

\noindent Introduce $\delta = \delta(g_1,\kappa_1,\kappa_2) := \det(g_1) \det(\kappa_1\kappa_2^{-1})$. The image of the function $t \mapsto h_2(t \delta; g_1, \kappa_1, \kappa_2)$ under $\VorH_{\pi_2}$ is $t \mapsto h_2^*(t \delta^{-1}; g_2, \kappa_1, \kappa_2)$ by the local functional equation \eqref{eq: LocFEGLnGL1}. Consequently for any $\Phi_2 \in \Sch(\Mat_{n_2}(\F))$ and $\beta_2 \in C(\pi_2^{\infty})$ we have
\begin{multline} \label{eq: 2ndExtVHId}
	\norm[\delta]_{\F}^{-\frac{n_2+1}{2}}\int_{\substack{\GL_{n_2}(\F)^{\times 2} \\ \delta \det g_2' = \det g_2}} \invOFour_1(T)(g_1,g_2; \kappa_2.\Phi.\kappa_1^{-1}) \beta_{\kappa_1,\kappa_2}(g_1^{\iota},g_2) \norm[\det g_1]_{\F}^{\frac{n_1+1}{2}} \norm[\det g_2]_{\F}^{n_2} \ud g_2 \ud g_2' \\
	= \int_{\GL_{n_2}(\F)} \widehat{\Phi}_2(g_2') \beta_2^{\iota}(g_2') \norm[\det g_2']_{\F}^{\frac{n_2+1}{2}} h_2(\delta \det g_2'; g_1, \kappa_1, \kappa_2) \ud g_2' \\
	= \int_{\GL_{n_2}(\F)} \Phi_2(g_2') \beta_2(g_2') \norm[\det g_2']_{\F}^{\frac{n_2-1}{2}} h_2^*(\delta^{-1} \det g_2'; g_1, \kappa_1, \kappa_2) \ud g_2' \\
	= \norm[\delta]_{\F}^{\frac{n_2-1}{2}}\int_{\substack{\GL_{n_2}(\F)^{\times 2} \\ \delta^{-1} \det g_2' = \det g_2}} \Phi_2(g_2') \beta_2(g_2') T(g_1^T,g_2^T; \kappa_1.{}^t\widehat{\Phi}.\kappa_2^{-1}) \beta_{\kappa_1,\kappa_2}(g_1^{\iota},g_2^{\iota}) \norm[\det g_1]_{\F}^{\frac{n_1+1}{2}} \norm[\det g_2]_{\F}^{n_2} \ud g_2 \ud g_2'.
\end{multline}

\noindent The proof of the following lemma is quite similar to that of Lemma \ref{lem: 1stExtVHAux}, and left to the reader.
\begin{lemma} \label{lem: 2ndExtVHAux}
	Both sides of \eqref{eq: 2ndExtVHId} are integrable for $\ud g_1 \ud \kappa_1 \ud \kappa_2$ on $\GL_{n_1}(\F) \times \gp{K}_n \times \gp{K}_n$.
\end{lemma}

\noindent Integrating both sides of \eqref{eq: 2ndExtVHId} we get
\begin{multline}
	\int_{\GL_{n_2}(\F)} \widehat{\Phi}_2(g_2') \beta_2^{\iota}(g_2') \norm[\det g_2']_{\F}^{\frac{n_2+1}{2}} \cdot h^{\sharp}((\det g_2')^{-1}) \ud g_2' = \\
	\int_{\GL_{n_2}(\F)} \Phi_2(g_2') \beta_2(g_2') \norm[\det g_2']_{\F}^{\frac{n_2-1}{2}} \cdot h^*(\det g_2') \ud g_2'
\end{multline}
	for the function $h^*(\cdot)$ defined by (absolutely convergent integral)
\begin{multline*}
	h^*(t) := \\
	\int_{\substack{\gp{K}_n^2 \times \GL_{n_1}(\F) \times \GL_{n_2}(\F) \\ \det(\kappa_1\kappa_2^{-1}) \det(g_1) \det(g_2) = t}}  T(g_1^T,g_2^T; \kappa_1.{}^t\widehat{\Phi}.\kappa_2^{-1}) \beta_{\kappa_1,\kappa_2}(g_1^{\iota},g_2^{\iota}) \norm[\det g_1]_{\F}^{\frac{n_1+1}{2}} \norm[\det g_2]_{\F}^{\frac{n_2+1}{2}} \ud g_1 \ud g_2 \ud \kappa_1 \ud \kappa_2 = \\
	\int_{\substack{\gp{K}_n^2 \times \GL_{n_1}(\F) \times \GL_{n_2}(\F) \\ \det(\kappa_1\kappa_2^{-1}) \det(g_1) \det(g_2) = t}}  T(g_1,g_2; \kappa_1.{}^t\widehat{\Phi}.\kappa_2^{-1}) \beta_{\kappa_1,\kappa_2}(g_1^{-1},g_2^{-1}) \norm[\det g_1]_{\F}^{\frac{n_1+1}{2}} \norm[\det g_2]_{\F}^{\frac{n_2+1}{2}} \ud g_1 \ud g_2 \ud \kappa_1 \ud \kappa_2 = \\
	\int_{\substack{\GL_n(\F) \\ \det g = t}} {}^t\widehat{\Phi}(g) \beta(g^{-1}) \ud g \cdot \norm[t]_{\F}^{\frac{n+1}{2}} = \int_{\substack{\GL_n(\F) \\ \det g = t}} \widehat{\Phi}(g) \beta^{\iota}(g) \ud g \cdot \norm[t]_{\F}^{\frac{n+1}{2}} = \VorH_{\pi}(h)(t),
\end{multline*}
	which concludes the proof of Theorem \ref{thm: Mult}, since $h^* = \widetilde{\VorH}_{\pi_2} \circ \Inv (h^{\sharp})$ by Theorem \ref{thm: ExtVorH}.

\section{Some Functional Analysis: non-Archimedean Case}
\label{sec: FANA}

	From now on $\F$ is a non-archimedean local field with residual field $\fF_q$.

	\subsection{Germs of Partial Gauss Integrals}
	
\begin{definition} \label{def: GermParGI}
	Let $\omega$ be a (quasi-)character of $\F^{\times}$. For every integer $l > \max(\cond(\omega),1)$ we define a function $G_l(\cdot; \omega, d)$ supported in $\varpi_{\F}^{-ld} \vO_{\F}^{\times}$ given by
	$$ G_l(\varpi_{\F}^{-ld} y; \omega, d) := 
	\begin{cases} 
		\psi(\varpi_{\F}^{-l} y) \omega(\varpi_{\F}^{-l} y) & \text{if } d = 1 \\
		\int_{(\vO_{\F}^{\times})^{d-1}} \psi \left( \frac{y(t_2 \cdots t_d)^{-1} + t_2 + \cdots + t_d}{\varpi_{\F}^l} \right) \omega \left( \frac{t_d}{\varpi_{\F}^l} \right) \ud t_2 \cdots \ud t_d & \text{if } d > 1
	\end{cases}, \quad \forall \ y \in \vO_{\F}^{\times}. $$
\end{definition}

\begin{lemma} \label{lem: GermParGI}
	(1) For every (quasi-)character $\chi$ of $\F^{\times}$ we have
	$$ \int_{\F^{\times}} G_l(y; \omega, d) \chi(y) \frac{\ud y}{\norm[y]_{\F}} = \varepsilon(1, \chi^{-1}, \psi)^{d-1} \varepsilon(1, \chi^{-1} \omega^{-1}, \psi) \cdot \id_{\cond(\chi) = l}. $$

\noindent (2) Suppose $d \geq 2$ and $l \geq 2 \cond(\omega)$. For any $y \in \vO_{\F}^{\times}$ we have the formula
\begin{multline*} 
	G_l(\varpi_{\F}^{-ld}y; \omega, d) = \Vol(U_{\F}^{\lceil \frac{l}{2} \rceil})^{d-1} \cdot \sideset{}{_{\substack{z \in \vO_{\F}^{\times}/U_{\F}^{\lceil \frac{l}{2} \rceil} \\ z^d \in y U_{\F}^{\lfloor \frac{l}{2} \rfloor}}}} \sum \omega \left( \tfrac{z}{\varpi_{\F}^l} \right) \psi \left( \tfrac{z}{\varpi_{\F}^l} \right) \cdot \\
	\begin{cases}
		\psi \left( \frac{y}{\varpi_{\F}^l z} \right) & \text{if } d=2 \\
		\sum_{\delta_j \in U_{\F}^{\lfloor \frac{l}{2} \rfloor}/U_{\F}^{\lceil \frac{l}{2} \rceil}} \psi \left( \tfrac{z}{\varpi_{\F}^l} \left( \frac{y}{z^d \delta_2 \cdots \delta_{d-1}} + \delta_2 + \cdots + \delta_{d-1} \right) \right) & \text{if } d>2
	\end{cases}. 
\end{multline*}
	
\noindent (3) Let $\omega_j$ be two (quasi-)characters of $\F^{\times}$ and $d_j \in \Z_{\geq 1}$. Suppose $l_j \geq \max(2 \cond(\omega_j), 2)$. We have
	$$ G_{l_1}(\cdot; \omega_1, d_1) * G_{l_2}(\cdot; \omega_2, d_2) = 
	\begin{cases}
		G_l(\cdot; \omega_1 \omega_2, d_1+d_2) & \text{if } l_1 = l_2 = l \\
		0 & \text{if } l_1 \neq l_2
	\end{cases}, $$
	where the convolution is taken over $\F^{\times}$ in the sense $f_1*f_2(x) = \int_{\F^{\times}} f_1(xy^{-1}) f_2(y) \ud^{\times} y$.
\end{lemma}
\begin{proof}
	(1) A direct computation gives
	$$ \int_{\F^{\times}} G_l(y; \omega, d) \chi(y) \frac{\ud y}{\norm[y]_{\F}} = \left( \int_{\varpi_{\F}^{-l} \vO_{\F}^{\times}} \psi(t) \chi(t) \frac{\ud t}{\norm[t]_{\F}} \right)^{d-1} \cdot \int_{\varpi_{\F}^{-l} \vO_{\F}^{\times}} \psi(t) \chi\omega(t) \frac{\ud t}{\norm[t]_{\F}}. $$
	The desired formula then follows from \cite[23.5 Exercise]{BuH06} if $\cond(\chi), \cond(\chi \omega) > 0$; and from direct computation if either $\cond(\chi)=0$ or $\cond(\chi \omega)=0$.
	
\noindent (2) Write $t_1 := y(t_2 \cdots t_d)^{-1}$. Choose any $2 \leq j \leq d$. Multiplying $t_j$ by $1+u$ and integrating over $u \in \vP_{\F}^{\lceil l/2 \rceil}$ we rewrite and get
	$$ G_l(\varpi_{\F}^{-ld} y; \omega, d) = \int_{\substack{t_1 \cdots t_d = y \\ t_j \in \vO_{\F}^{\times}}} \psi \left( \frac{t_1 + \cdots + t_d}{\varpi_{\F}^l} \right) \omega \left( \frac{t_d}{\varpi_{\F}^l} \right) \cdot \left( \oint_{\vP_{\F}^{\lceil l/2 \rceil}} \psi \left( \frac{(t_j-t_1)u}{\varpi_{\F}^l} \right) \ud u \right) \ud t_1 \cdots \ud t_d. $$
	The non-vanishing of the inner integral implies $t_j-t_1 \in \vP_{\F}^{\lfloor l/2 \rfloor}$, under which conditions the integrand is constant on each coset of $U_{\F}^{\lceil l/2 \rceil}$ for every variable $t_j$; in particular $y=t_1 \cdots t_d \equiv t_d^d \pmod{\vP_{\F}^{\lfloor l/2 \rfloor}}$. Writing the integral as a discrete sum by integrating over $U_{\F}^{\lceil l/2 \rceil}$ for each $t_j$ we get
	$$ G_l(\varpi_{\F}^{-ld} y; \omega, d) = \Vol(U_{\F}^{\lceil \frac{l}{2} \rceil})^{d-1} \cdot \sideset{}{_{\substack{t_j \in \vO_{\F}^{\times} / U_{\F}^{\lceil l/2 \rceil}, t_j-t_d \in \vP_{\F}^{\lfloor l/2 \rfloor} \\ t_1 \cdots t_d = y \pmod{\vP_{\F}^{\lfloor l/2 \rfloor}}}}} \sum \psi \left( \frac{t_1 + \cdots + t_d}{\varpi_{\F}^l} \right) \omega \left( \frac{t_d}{\varpi_{\F}^l} \right). $$
	The desired formula follows by the change of variables $z=t_d$, $\delta_j = z^{-1}t_j$ for $2 \leq j \leq d-1$.
	
\noindent (3) From (1) we deduce that, for any $z \in \varpi_{\F}^{-l_j d_j} \vO_{\F}^{\times}$, the following function
	$$ \vO_{\F}^{\times} \to \C, \quad y \mapsto G_{l_j}(z y; \omega_j, d_j) $$
	is a linear combination of characters $\chi$ of $\vO_{\F}^{\times}$ with $\cond(\chi) = l_j$. If $l_1 \neq l_2$ we deduce the vanishing of the convolution by orthogonality of characters of $\vO_{\F}^{\times}$. If $l_1=l_2=l$, the convolution is clearly supported on $\varpi_{\F}^{-l(d_1+d_2)} \vO_{\F}^{\times}$. At $\varpi_{\F}^{-l(d_1+d_2)} y$ with $y \in \vO_{\F}^{\times}$ it is by definition
	$$ \int_{\substack{t_1 \cdots t_{d_1+d_2} = y \\ t_j \in \vO_{\F}^{\times}}} \psi \left( \frac{t_1 + \cdots + t_{d_1+d_2}}{\varpi_{\F}^l} \right) \omega_1 \left( \frac{t_{d_1}}{\varpi_{\F}^l} \right) \omega_2 \left( \frac{t_{d_1+d_2}}{\varpi_{\F}^l} \right) \ud t_1 \cdots \ud t_{d_1+d_2}. $$
	Multiplying $t_{d_1}$ and $t_{d_1+d_2}$ by $(1+u)^{-1}$ and $1+u$ respectively, and integrating over $u \in \vP_{\F}^{\lceil l/2 \rceil}$ we see that in the domain of integration we may impose the condition $t_{d_1}-t_{d_1+d_2} \in \vP_{\F}^{\lfloor l/2 \rfloor}$ as in the above proof of (2). Consequently we have $\omega_1(t_{d_1}) = \omega_1(t_{d_1+d_2})$ and identify the above integral with $G_l(\varpi_{\F}^{-l(d_1+d_2)} y; \omega_1 \omega_2, d_1+d_2)$, proving the desired equality.
\end{proof}

	\subsection{Some Paley--Wiener Results}
	
	We studied $\SSch(\F)$ in Definition \ref{def: SSchNA}, an analogue of the space for $\F=\R$ with the same notation in Miller--Schmid's work \cite[Definition 6.4]{MS04}. We give an extension of it.
	
\begin{definition} \label{def: SSchNAExt}
	A continuous function $f \in \Cont(\F^{\times})$ has \emph{$\mathrm{Kl}_{\omega}(d)$-type singularity} at $\infty$, where $\omega$ is a (quasi-)character of $\F^{\times}$ and $d \in \Z_{\geq 1}$, if there are constants $N \in \Z_{\geq 1}$ and $C \in \C$ so that
	$$ f(y) = C \cdot G_{\geq N}(y; \omega, d), \quad \forall \ y: v_{\F}(y) \leq -Nd, $$
	where we have written $G_{\geq N}(\cdot) =  \sideset{}{_{l \geq N}} \sum G_l(\cdot)$. We write $\SSch^{(d)}(\F, \omega)$ for the space of continuous functions with simple singularity at $0$ and $\mathrm{Kl}_{\omega}(d)$-type singularity at $\infty$. We also write $\SSch^{(0)}(\F; \omega) = \SSch(\F)$.
\end{definition}

\begin{remark}
	The space of finite functions is spanned by functions of the shape $\chi \cdot v_{\F}^k$, where $\chi$ is a quasi-character of $\F^{\times}$ and $k \in \Z_{\geq 0}$. We have for any $n \in \Z$ and any $N \in \Z_{\geq 2}$ sufficiently large
	$$ \SSch(\F) = \Cont_c^{\infty}(\F^{\times}) \bigoplus_{\chi, k} \C \cdot \chi v_{\F}^k \id_{\vP_{\F}^n}, \quad \SSch^{(d)}(\F, \omega) = \SSch(\F) \bigoplus \C \cdot G_{\geq N}(y; \omega, d). $$
\end{remark}

\begin{definition} \label{def: PVConv}
	For any locally integrable function $f \in \intL_{\mathrm{loc}}^1(\F^{\times})$ we introduce the \emph{principal value integral}
	$$ \int_{\F^{\times}}^{\mathrm{pv}} f(y) \ud^{\times} y := \lim_{N \to \infty} \sum_{n=-N}^N \int_{\varpi_{\F}^n \vO_{\F}^{\times}} f(y) \ud^{\times} y $$
	whenever the limit exists. For two functions $f_1,f_2 \in \intL_{\mathrm{loc}}^1(\F^{\times})$ we define the \emph{regularized convolution}
	$$ f_1 *_{\mathrm{pv}} f_2 (x) := \int_{\F^{\times}}^{\mathrm{pv}} f_1(xy^{-1}) f_2(y) \ud^{\times} y. $$
\end{definition}

	We introduced $\SMel(\F)$ in Definition \ref{eq: SMellinNA} as the image of $\SSch(\F)$ under the Mellin transform. We extend it to $\SMel^{(d)}(\F; \omega)$ as the image of $\SSch^{(d)}(\F, \omega)$ under some \emph{regularized} Mellin transform.
	
\begin{definition} \label{def: RegMellinT}
	The regularized Mellin transform of $f \in \Cont(\F^{\times})$ at $\xi \in \widehat{\vO_{\F}^{\times}}$ and $s \in \C$ is given by
	$$ \RMellin{f}(\xi,s) = \int_{\F^{\times}}^{\mathrm{pv}} f(t) \xi(t) \norm[t]_{\F}^s \ud^{\times} t $$
whenever the limit exists.
\end{definition}

\begin{definition} \label{def: SSchMellin}
	For each $d \in \Z_{\geq 1}$, (quasi-)character $\omega$ and $N \in \Z_{\geq 1}$ we introduce function
 	$$ \widehat{\vO_{\F}^{\times}} \to \C(q_{\F}^s), \quad \xi \mapsto E_{\geq N}(\xi,q_{\F}^{-s}; \omega, d) := \varepsilon(1-s, \xi^{-1}, \psi)^{d-1} \varepsilon(1-s, \xi^{-1} \omega^{-1}, \psi) \id_{\cond(\xi) \geq N}, $$
	where $\xi$ is regarded as a character of $\F^{\times}$ via $\xi(\varpi_{\F})=1$. Write $\SMel^{(d)}(\F; \omega) := \SMel(\F) \bigoplus \C \cdot E_{\geq N}(\xi,q_{\F}^s; \omega, d)$, which is independent of the choice of $N$.
\end{definition}

\begin{lemma} \label{lem: SSchPW}
	For any $f \in \SSch^{(d)}(\F, \omega)$ the regularized Mellin transform is well-defined for $\Re s \gg 1$ uniformly with respect to $\xi$, and has meromorphic continuation to $s \in \C$. It establishes an isomorphism as vector spaces between $\SSch^{(d)}(\F, \omega)$ and $\SMel^{(d)}(\F; \omega)$.
\end{lemma}
\begin{proof}
	In view of Lemma \ref{lem: GermParGI} this follows readily from Proposition \ref{prop: PaleyWienerNA}. 
\end{proof}

\noindent For convenience of further reference, we record some regularized Mellin transform in a table:
\begin{table}[h!]
\centering
\begin{tabular}{| c | c |}
	\hline
	$f$ & $\RMellin{f}$ \\
	\hline
	$G_{\geq N}(y; \omega, d) \ (N \geq \max(2\cond(\omega),2))$ & $E_{\geq N}(\xi,X; \omega, d)$ \\
	\hline
	$\chi(y) \id_{\vP_{\F}^n}(y)$ & $\id_{\chi^{-1} \mid_{\vO_{\F}^{\times}}}(\xi) \cdot \chi(\varpi_{\F})^n X^{n} (1-\chi(\varpi_{\F}) X)^{-1}$ \\
	\hline
\end{tabular}
\caption{Some regularized Mellin transform.}
\label{table: RM}
\end{table}

	\subsection{Some Convolution Algebra}
	
\begin{proposition} \label{prop: SSchAlg}
	(1) For any $d_j \in \Z_{\geq 1}$ and characters $\omega_j$ the regularized convolution is a bilinear map 
	$$ \SSch^{(d_1)}(\F, \omega_1) \times \SSch^{(d_2)}(\F, \omega_2) \to \SSch^{(d_1+d_2)}(\F, \omega_1 \omega_2), \quad (f_1, f_2) \mapsto f_1 *_{\mathrm{pv}} f_2. $$
	
\noindent (2) The regularized Mellin transform convert $*_{\mathrm{pv}}$ to point-wise multiplication. In other words for any $f_j \in \SSch^{(d_j)}(\F, \omega_j)$ we have the equality as functions in $\SMel^{(d_1+d_2)}(\F; \omega_1 \omega_2)$
	$$ \RMellin{f_1 *_{\mathrm{pv}} f_2}(\xi, s) = \RMellin{f_1}(\xi,s) \cdot \RMellin{f_2}(\xi,s). $$
	
\noindent (3) Let $\SSch(\F, \widetilde{\F^{\times}})$ be the span of all $\SSch^{(d)}(\F, \omega)$ for $d \in \Z_{\geq 0}$ and (quasi-)characters $\omega$. Then $(\SSch(\F, \widetilde{\F^{\times}}), *_{\mathrm{pv}})$ is a commutative algebra, which contains $(\SSch(\F), *_{\mathrm{pv}})$ as an ideal.
\end{proposition}
\begin{proof}
	Note that all assertions concerning only elements in $\SSch(\F)$ are consequences of Proposition \ref{prop: PaleyWienerNA}: the bijection $\SSch(\F) \to \SMel(\F)$ established there is easily seen to convert $*_{\mathrm{pv}}$ into point-wise multiplication. We can omit them in the rest of the proof. 
	
\noindent The regularized convolution is clearly bilinear by definition. To verify $f_1 *_{\mathrm{pv}} f_2 \in \SSch^{(d_1+d_2)}(\F, \omega_1 \omega_2)$ it suffices to verify it for some sets of generators of $\SSch^{(d_j)}(\F, \omega_j)$. By Fourier inversion on $\vO_{\F}^{\times}$ we see that $\Cont_c^{\infty}(\F^{\times})$ is generated by $\chi \id_{\varpi_{\F}^n \vO_{\F}^{\times}} = \chi \id_{\vP_{\F}^n} - \chi \id_{\vP_{\F}^{n+1}}$ for $\chi \in \widehat{\F^{\times}}$ and $n \in \Z$. Therefore $\SSch(\F)$ is spanned by $\chi v_{\F}^k \id_{\vP_{\F}^n}$. Lemma \ref{lem: GermParGI} implies for $N_j \geq \max(2\cond(\omega_j),2)$ and $N \geq \max(2\cond(\omega),2)$ the equations
\begin{equation} \label{eq: 1stRegConvAux}
	G_{\geq N_1}(\cdot; \omega_1, d_1) *_{\mathrm{pv}} G_{\geq N_2}(\cdot; \omega_2, d_2) = G_{\geq \max(N_1,N_2)}(\cdot; \omega_1 \omega_2, d_1+d_2);
\end{equation}
\begin{multline} \label{eq: 2ndRegConvAux}
	\left( G_{\geq N}(\cdot; \omega, d) *_{\mathrm{pv}} \chi \id_{\vP_{\F}^n} \right)(x) = \chi(x) \sum_{l \geq N, \lceil \frac{n-v_{\F}(x)}{d} \rceil} \int_{\varpi_{\F}^{-ld} \vO_{\F}^{\times}} G_l(y; \omega, d) \chi^{-1}(y) \ud^{\times} y \\
	= \chi(x) \id_{\vP_{\F}^{n-d\cond(\chi)}}(x) \cdot \id_{\cond(\chi) \geq N} \cdot \varepsilon(1, \chi, \psi)^{d-1} \varepsilon(1, \chi \omega^{-1}, \psi),
\end{multline}
	which indeed lies in $\SSch(\F)$. Replacing $\chi$ by $\chi \norm_{\F}^s$ in the above formula and differentiating $k$ times with respect to $s$ at $s=0$ we deduce $G_{\geq N}(\cdot; \omega, d) *_{\mathrm{pv}} \chi v_{\F}^k \id_{\vP_{\F}^n} \in \SSch(\F)$ and complete the proof of (1). Combining Table \ref{table: RM} with \eqref{eq: 1stRegConvAux} \& \eqref{eq: 2ndRegConvAux}, we verify (2) for the relevant special elements by:
\begin{multline*}
	E_{\geq N_1}(\xi,q_{\F}^{-s}; \omega_1, d_1) \cdot E_{\geq N_2}(\xi,q_{\F}^{-s}; \omega_2, d_2) = \\
	\varepsilon(1-s, \xi^{-1}, \psi)^{d_1+d_2-2} \varepsilon(1-s, \xi^{-1} \omega_1^{-1}, \psi) \varepsilon(1-s, \xi^{-1} \omega_2^{-1}, \psi) \id_{\cond(\xi) \geq \max(N_1,N_2)} = \\
	\varepsilon(1-s, \xi^{-1}, \psi)^{d_1+d_2} (\omega_1\omega_2)(c) \id_{\cond(\xi) \geq \max(N_1,N_2)} = E_{\geq \max(N_1,N_2)}(\xi, q_{\F}^{-s}; \omega_1 \omega_2, d_1+d_2),
\end{multline*}
	where $\xi^{-1}(1+x) = \psi(cx)$ for all $x \in \vP_{\F}^{\lfloor \cond(\xi)/2 \rfloor}$ and some $c \in \F^{\times}$, and we have applied the stability theorem for $\GL_1$ \cite[\S 23.8]{BuH06}; and
\begin{multline*}
	E_{\geq N}(\xi,q_{\F}^{-s}; \omega, d) \cdot \id_{\chi^{-1} \mid_{\vO_{\F}^{\times}}}(\xi) \cdot \chi(\varpi_{\F})^n q_{\F}^{-ns} (1-\chi(\varpi_{\F}) q_{\F}^{-s})^{-1} = \\
	\id_{\geq N}(\cond(\xi)) \id_{\chi^{-1} \mid_{\vO_{\F}^{\times}}}(\xi) \cdot \varepsilon(1-s, \xi^{-1}, \psi)^{d-1} \varepsilon(1-s, \xi^{-1}\omega^{-1}, \psi) \cdot \chi(\varpi_{\F})^n q_{\F}^{-ns} (1-\chi(\varpi_{\F}) q_{\F}^{-s})^{-1} = \\
	\id_{\cond(\chi) \geq N} \cdot \chi(\varpi_{\F})^{n-d\cond(\chi)} q_{\F}^{-(n-d\cond(\chi))s} (1-\chi(\varpi_{\F}) q_{\F}^{-s})^{-1} \cdot \varepsilon(1, \chi, \psi)^{d-1} \varepsilon(1, \chi \omega^{-1}, \psi),
\end{multline*}
	where we have applied $\varepsilon(s, \xi^{-1}, \psi) = \varepsilon(s, \chi, \psi) \cdot \chi(\varpi_{\F})^{-\cond(\chi)}$, $\varepsilon(s, \chi, \psi) = \varepsilon(1/2, \chi, \psi) q_{\F}^{\cond(\chi)(1/2-s)}$ and $\cond(\chi \omega^{-1}) = \cond(\chi)$ under the condition $\cond(\chi) \geq N \geq \max(2\cond(\omega),2)$ (hence $> \cond(\omega)$). Applying the same deformation technique to $\chi$ as before, we complete the proof of (2). (3) follows readily from \eqref{eq: 2ndRegConvAux}. In particular, the (non-obvious) associativity (and commutativity) of $*_{\mathrm{pv}}$ follows from the one for the point-wise multiplication for functions in $(\xi, q_{\F}^{-s})$.
\end{proof}

	It would be elegant to view the results in this section in terms of distributions. In general the convolution of two distributions need not exist, unless one has compact support (see \cite[Definition 6.36]{Ru91}). The space $\SSch(\F, \widetilde{\F^{\times}})$ provides an example where the convolution makes sense for distributions with non-compact support. Hence we may re-formulate Proposition \ref{prop: SSchAlg} in the following equivalent way.

\begin{proposition} \label{prop: SSchAlgBis}
	Functions in $\SSch(\F, \widetilde{\F^{\times}})$ are distributions on $\Cont_c^{\infty}(\F^{\times})$, and the regularized convolution on $\SSch(\F, \widetilde{\F^{\times}})$ extends the convolution in the sense of distributions.
\end{proposition}

\section{Cuspidal Inducing Datum and Related Computation}
\label{sec: CID}

	\subsection{Supercuspidal Representations via Compact Induction}
	
	The supercuspidal representations of $\gp{G} := \GL_d(\F)$ are classified by Bushnell--Kutzko \cite{BK93}. According to this classification, for every supercuspidal $\pi$ of $\GL_d(\F)$ we can find an open subgroup $\gp{J}$ which is compact modulo the center, and an irreducible admissible representation $(\rho, V_{\rho})$ of $\gp{J}$ such that
\begin{equation} \label{eq: ScBHModel}
	\pi = \cInd_{\gp{J}}^{\gp{G}} \rho, \quad \text{resp. its contragredient} \quad \widetilde{\pi} = \cInd_{\gp{J}}^{\gp{G}} \widetilde{\rho} 
\end{equation}
is compactly induced from $\rho$, resp. its contragredient $\widetilde{\rho}$. A matrix coefficient $\beta = \beta_{f_1,f_2}$ associated with a pair of vectors $(f_1, f_2) \in V_{\pi}^{\infty} \times V_{\widetilde{\pi}}^{\infty}$ is expressed in terms of the $\gp{J}$-invariant pairing $\extPairing{\cdot}{\cdot}_{\rho}$ on $V_{\rho} \times V_{\widetilde{\rho}}$ as
	$$ \beta_{f_1,f_2}(g) = \int_{\gp{J} \backslash \gp{G}} \extPairing{f_1(xg)}{f_2(x)}_{\rho} \ud x. $$
\begin{definition} \label{def: FibreOfDet}
	For any open subgroup $\gp{H}$ of $\gp{G}$ and any $t \in \F^{\times}$, let $\gp{H}(t)$ be the subset of elements $x \in \gp{H}$ with $\det x = t$. Note that $\gp{H}(t)$, if non-empty, is a coset of the subgroup $\gp{H}(1)$, and can be endowed with a Haar measure $\ud h_1$ of $\gp{H}(1)$ determined by
	$$ \int_{\gp{H}} f(h) \ud h = \int_{\F^{\times}} \left( \int_{\gp{H}(1)} f(a_t h_1) \ud h_1 \right) \ud^{\times} t, \quad \forall f \in \Cont_c^{\infty}(\gp{H}), $$
	where $a_t \in \gp{H}$ is any element with $\det(a_t) = t$.
\end{definition}
\begin{lemma} \label{lem: NoTrivialRp}
	The restriction of $\rho$ to $\gp{J}(1)$ does not contain the trivial representation.
\end{lemma}
\begin{proof}
	Suppose this is not the case. Since $\gp{J}(1)$ is a normal subgroup of $\gp{J}$, the non-zero subspace of $\gp{J}(1)$-fixed vectors is stable by $\gp{J}$, hence has to be $V_{\rho}$. Therefore the action of $\rho$ factors through $\gp{J}(1)$, hence is an irreducible representation of the finite abelian group $\gp{J}/\gp{J}(1) < \gp{G}/\gp{G}(1) \simeq \F^{\times}$. Consequently $\rho$ has dimension $1$ and is the restriction of $\chi \circ \det$ for some (quasi)-character $\chi$ of $\F^{\times}$. The map
	$$ V_{\pi}^{\infty} \to \C, \quad f \mapsto \int_{\gp{J} \backslash \gp{G}} f(g) \chi^{-1}(\det g) \ud g $$
is a non-zero $\gp{G}$-intertwining map from $\pi$ to $\chi \circ \det$, whose kernel is a non-trivial sub-representation of $\pi$, contradicting the irreducibility of $\pi$.
\end{proof}

	\subsection{Principal Orders and Related Properties}

	By induction in stages, we may take $\gp{J}$ to be \emph{maximal} in \eqref{eq: ScBHModel}. In this case, a such pair $(\gp{J}, \rho)$ is called a \emph{cuspidal inducing datum} by Bushnell--Henniart \cite[\S 15.8 Definition]{BuH06}. (While if $\gp{J}$ is taken to be minimal, a such pair $(\gp{J}, \rho)$ is called a \emph{cuspidal type}.) The maximal compact-modulo-center subgroups are classified in the Bruhat--Tits theory. In the case of $\GL_d(\F)$, they are realized as normalizers of the \emph{principal orders}. We recall their properties without proofs, referring the reader to \cite[\S 1]{BuF85} or \cite[\S 1]{Bu87} for details. Let $d = ef$ with $e,f \in \Z_{\geq 1}$. For any subset $S$ of $\F$ write $[S]_f$ for the set of $f \times f$ matrices with entries in $S$. The standard principal order $\oA$ with \emph{ramification index} $e(\oA) = e$ and \emph{block size} $f(\oA) = f$ is
\begin{equation} \label{eq: POrder}
	\oA = \begin{pmatrix} [\vO_{\F}]_f & [\vO_{\F}]_f & [\vO_{\F}]_f & \cdots & [\vO_{\F}]_f \\
	[\vP_{\F}]_f & [\vO_{\F}]_f & [\vO_{\F}]_f & \cdots & [\vO_{\F}]_f \\
	\vdots & \ddots & \ddots & \ddots & \vdots \\
	[\vP_{\F}]_f & \cdots & [\vP_{\F}]_f & [\vO_{\F}]_f & [\vO_{\F}]_f \\
	[\vP_{\F}]_f & \cdots & [\vP_{\F}]_f & [\vP_{\F}]_f & [\vO_{\F}]_f \end{pmatrix}.
\end{equation}
	Call it the \emph{standard order with parameter} $(e,f)$. Its Jacobson radical is given via a prime element by
\begin{equation} \label{eq: POrderRad}
	\VP = \Pi \oA = \oA \Pi, \quad
	\Pi = \begin{pmatrix} & \id_{d-f} \\
	\varpi_{\F} \id_f & \end{pmatrix}.
\end{equation}
	The normalizer subgroup of $\oA$ is
\begin{equation} \label{eq: POrderNmSubGp}
	\nG(\oA) = \left\{ g \in \gp{G} \ \middle| \ g \oA g^{-1} = \oA \right\} = \Pi^{\Z} \oA^{\times}.
\end{equation}
	A filtration of open normal subgroups at $\id_d$ is given by
\begin{equation} \label{eq: POrderNbhd1}
	U_{\oA}^n := (\id_d + \VP^n) \cap \oA^{\times}, \quad n \in \Z_{\geq 0}.
\end{equation}
	For any $\vO_{\F}$-lattice $L$ in $\Mat_d(\F)$, write $L^* := \left\{ x \in \Mat_d(\F) \ \middle| \ \Tr(x L) \subset \vO_{\F} \right\}$. We have
\begin{equation} \label{eq: POrderRadDual}
	(\VP^n)^* = \VP^{1-e-n}, \quad \forall n \in \Z.
\end{equation}
	If $\gp{J} = \nG(\oA)$, there is $\cond(\rho) \in \Z_{\geq 1}$, called the \emph{conductor exponent}, such that
\begin{equation} \label{eq: CIDCondExp}
	U_{\oA}^{\cond(\rho)} \subset \ker(\rho), \quad U_{\oA}^{\cond(\rho)-1} \not\subset \ker(\rho).
\end{equation}
	If $\cond(\pi)$ is the conductor exponent of $\pi$, then we have (see \cite[Theorem 3]{Bu87})
\begin{equation} \label{eq: CondExpRel}
	\cond(\rho) = \cond(\pi)/f + 1 - e(\oA).
\end{equation}
	Note that $\oA$ is uniquely determined by $\pi$, hence we may write $e(\oA) = e(\pi)$ and $f(\oA) = f(\pi)$. In fact, $f(\pi)$ is the number of unramified characters $\chi$ such that $\pi \simeq \pi \otimes (\chi \circ \det)$ by the last paragraph of the proof of \cite[(6.2.4) Theorem]{BK93}. Note also that $f(\pi) \mid \cond(\pi)$ by \cite[15.2 Corollary]{Bu17}. Moreover we have
\begin{equation} \label{eq: CChCondBd}
	\VP^n \cap \F = \vP_{\F}^{\lceil n/e \rceil} \quad \Rightarrow \quad U_{\oA}^n \supset U_{\F}^{\lceil n/e \rceil} \quad \Rightarrow \quad \lceil \cond(\rho)/e \rceil \geq \cond(\omega).
\end{equation}

	\subsection{Norm One Subgroups}
	
	We will also need for $n \in \Z_{\geq 0}$ some properties of the following subgroups
\begin{equation} \label{eq: POrderNbhd1Var}
	U_{\oA}^n(1) = U_{\oA}^n \cap \SL_d(\F), \quad \widetilde{U}_{\oA}^n := \vO_{\F}^{\times} U_{\oA}^n, \quad \widetilde{U}_{\oA}^n(1) = \widetilde{U}_{\oA}^n \cap \SL_d(\F).
\end{equation}
\begin{lemma} \label{lem: DetTrRel}
	Let $m \in \Z_{\geq 0}$. For every $X \in \VP^m$ we have $\det(\id_d + X) \equiv 1 + \Tr(X) \pmod{\vP_{\F}^{\lceil 2m/e \rceil}}$. In particular, we have $\det U_{\oA}^m = U_{\F}^{\lceil m/e \rceil}$ and $U_{\oA}^m = \diag(U_{\F}^{\lceil m/e \rceil},1,\dots,1) \times U_{\oA}^m(1)$.
\end{lemma}
\begin{proof}
	The case of $m=0$ is obvious. Let $m \geq 1$. The case of $e=1$ is simple and left to the reader. Suppose $e \geq 2$. Write $m=el+m_1$ for some $l \in \Z_{\geq 0}$ and $0 < m_1 \leq e$. Introduce the standard order $\oA_1$, resp. $\oA_1'$ with parameter $(m_1,f)$, resp. $(e-m_1,f)$. Write the Jacobson radical as $\VP_1$, resp. $\VP_1'$. A direct computation gives, for $d_1 := m_1 f$ and $d_1' := (e-m_1)f$, the formula
\begin{equation} \label{eq: PowerJRad}
	\VP^m = \varpi_{\F}^l \begin{pmatrix} \Mat_{d_1' \times d_1}(\vP_{\F}) & \oA_1' \\ \varpi_{\F} \oA_1 & \Mat_{d_1 \times d_1'}(\vP_{\F}) \end{pmatrix}.
\end{equation}
	Write $X = (x_{ij})$ and $Y = (y_{ij})$ with $y_{ij} = \delta_{i,j}+x_{ij}$, where $\delta_{i,j} = 1$ if $i=j$ or $0$ if $i \neq j$. Observing \eqref{eq: PowerJRad} we deduce $x_{ii} \in \vP_{\F}^{l+1}$. Note that $\lceil 2m/e \rceil \leq 2l+2$. Let $S_d$ be the permutation group of $\left\{ 1, \dots, d \right\}$. Then
\begin{multline} \label{eq: DetMod}
	\det Y = \prod_{i=1}^d (1+x_{ii}) + \sum_{\id \neq \sigma \in S_d} \sgn(\sigma) \prod_{i=1}^d y_{i,\sigma(i)} \\
	\equiv 1+\Tr(X) + \sum_{\id \neq \sigma \in S_d} \sgn(\sigma) \prod_{i=1}^d y_{i,\sigma(i)} \pmod{\vP_{\F}^{\lceil 2m/e \rceil}}.
\end{multline}
	Consider any $\id \neq \sigma \in S_d$. It has at most $d-2$ fixed points. 
\begin{itemize}
	\item If $m_1 \leq e/2$, then $\lceil 2m/e \rceil = 2l+1$. There exists $j$ with $\sigma(j) < j$. By \eqref{eq: PowerJRad}, we have $y_{j,\sigma(j)} \in \vP_{\F}^{l+1}$ and $y_{i,\sigma(i)} \in \vP_{\F}^l$ for any $i \neq j$ with $\sigma(i) \neq i$. 
	\item If $e/2 < m_1 \leq e$, then $\lceil 2m/e \rceil = 2l+2$. If at least two $j$ has $\sigma(j) < j$, then each such $y_{j,\sigma(j)}$ lies in $\vP_{\F}^{l+1}$. Otherwise $\sigma$ is the transposition of $j$ and $\sigma(j)$, yielding $y_{j,\sigma(j)}y_{\sigma(j),j} \in \vP_{\F}^{2l+2}$ by \eqref{eq: PowerJRad}.
\end{itemize}
	The summand for $\id \neq \sigma \in S_d$ in \eqref{eq: DetMod} lies in $\vP_{\F}^{\lceil 2m/e \rceil}$ in both cases and we conclude.
\end{proof}
\begin{proposition} \label{prop: POrderNbhdQ}
	Let $m,n \in \Z_{\geq 1}$ satisfy $m < n \leq 2m$. We have
	$$ U_{\oA}^m / U_{\oA}^n \simeq \VP^m/\VP^n, \quad U_{\oA}^m(1) / U_{\oA}^n(1) \simeq \VP^m[0]/\VP^n[0], $$
where for any $\vO_{\F}$-lattice $L$ in $\Mat_d(\F)$, $L[0]$ denotes the sub-$\vO_{\F}$-module of trace $0$ elements in $L$.
\end{proposition}
\begin{proof}
	It suffices to treat the case $n=2m$, as the general case follows easily from this special one. The first equation is easily verified by the following explicit realization of the isomorphism
\begin{equation} \label{eq: POrderNbhdQExp}
	\VP^m/\VP^{2m} \to U_{\oA}^m / U_{\oA}^{2m}, \quad X + \VP^{2m} \mapsto (\id_d + X)U_{\oA}^{2m}. 
\end{equation}
	We shall prove that the restriction of the isomorphism \eqref{eq: POrderNbhdQExp} to $\VP^m[0]/\VP^{2m}[0]$ gives the desired second equation. This is equivalent to proving the following two claims:
\begin{itemize}
	\item[(1)] For any $X \in \VP^m[0]$ we have $(\id_d + X)U_{\oA}^{2m} \cap U_{\oA}^m(1) \neq \emptyset$;
	\item[(2)] For any $\id_d + X \in U_{\oA}^m(1)$ we have $(X + \VP^{2m}) \cap \VP^m[0] \neq \emptyset$.
\end{itemize}
	To prove (1), consider $X \in \VP^m[0]$. By Lemma \ref{lem: DetTrRel} we have $\det (\id_d + X) \in U_{\F}^{\lceil 2m/e \rceil} = \det U_{\oA}^{2m}$. Therefore we can find $\kappa \in U_{\oA}^{2m}$ with $\det (\id_d + X) = \det \kappa$, and $(\id_d + X) \kappa^{-1} \in (\id_d + X)U_{\oA}^{2m} \cap U_{\oA}^m(1)$ as required. To prove (2), take any $\id_d + X \in U_{\oA}^m(1)$. We have $X \in \VP^m$, and Lemma \ref{lem: DetTrRel} implies $1 = \det (\id_d + X) \equiv 1 + \Tr(X) \pmod{\vP_{\F}^{\lceil 2m/e \rceil}}$, or $\Tr(X) \in \vP_{\F}^{\lceil 2m/e \rceil}$. Replacing $m$ by $2m$ in \eqref{eq: PowerJRad} shows that $\diag(\vP_{\F}^{\lceil 2m/e \rceil}, \dots, \vP_{\F}^{\lceil 2m/e \rceil}) \subset \VP^{2m}$. In particular, we have $Y:=\diag(-\Tr(X),0,\dots,0) \in \VP^{2m}$ and $X+Y \in \VP^m[0]$ as required.
\end{proof}
\begin{remark}
	Carefully chasing the proof , we see that an explicit formula can be given as
\begin{equation} \label{eq: POrderNbhdQExpBis}
	\VP^m[0]/\VP^n[0] \simeq U_{\oA}^m(1) / U_{\oA}^n(1), \quad X+\VP^n[0] \mapsto (\id_d+X) \diag(\det(\id_d+X)^{-1},1,\dots,1) U_{\oA}^n(1) 
\end{equation}
	for the second group isomorphism in Proposition \ref{prop: POrderNbhdQ}. In particular, we have $\diag(\det(\id_d+X),1,\dots,1) \in U_{\oA}^{2m}$ for any $X \in \VP^m[0]$ by \eqref{eq: PowerJRad}.
\end{remark}

	\subsection{Some Technical Computation}

\begin{lemma} \label{lem: AddCharAvg}
	Suppose $d \geq 2$. Let $\oA$ be given in \eqref{eq: POrder}. Suppose $g \in \Pi^{-k} \oA^{\times}$ for some $k \in \Z_{\geq 1}$ satisfying
	$$ n := \lceil (k+1-e)/2 \rceil \geq m := \lfloor (k+1-e)/2 \rfloor \geq 1. $$
\begin{itemize}
	\item[(0)] If $k=el$ for some $l \in \Z_{\geq 1}$ then we have $\lceil m/e \rceil = \lfloor l/2 \rfloor$ and $\lceil n/e \rceil = \lceil l/2 \rceil$.
	\item[(1)] The non-vanishing of $I(g) := \oint_{U_{\oA}^n(1)} \psi(\Tr(g \kappa)) \ud \kappa \neq 0$ implies
	\begin{itemize}
		\item $k = el$ for some $l \in \Z_{\geq 1}$,
		\item $g \in \varpi_{\F}^{-l} \widetilde{U}_{\oA}^m$,
	\end{itemize}
	Under these conditions $I(g) = \psi(\Tr(g))$ is constant on left cosets of $U_{\oA}^n(1)$.
	\item[(2)] Write $g \in \varpi_{\F}^{-l} \widetilde{U}_{\oA}^m$ as $g = z \delta$ with $z \in \varpi_{\F}^{-l} \vO_{\F}^{\times}$ and $\delta \in U_{\oA}^m$. We have
\begin{multline*} 
	\frac{1}{[U_{\oA}^m(1) : U_{\oA}^n(1)]} \sum_{\kappa \in U_{\oA}^m(1) / U_{\oA}^n(1)} \psi \left( \Tr(g \kappa) \right) = \sum_{\delta_j \in U_{\F}^{\lfloor l/2 \rfloor} / U_{\F}^{\lceil l/2 \rceil}} \psi \left( z \left( \frac{\det \delta}{\delta_2 \cdots \delta_d} + \sideset{}{_{j=2}^d} \sum \delta_j \right) \right) \cdot I(e,f;l),
\end{multline*}
	$$ \text{with} \quad I(e,f;l) :=
	\begin{cases}
		1 & \text{if } 2 \mid l, 2 \nmid e \\
		q_{\F}^{-\frac{ef^2}{2}} & \text{if } 2 \mid l, 2 \mid e \\
		q_{\F}^{1-\frac{ef(f+1)}{2}} & \text{if } 2 \nmid l
	\end{cases}. $$
\end{itemize}
\end{lemma}
\begin{proof}
	(0) We give details for the first equation, the second being similar. We compute by definition
	$$ \tfrac{m}{e} = \begin{cases}
		\lfloor \tfrac{l}{2} \rfloor & \text{if } e=1 \\
		\tfrac{l}{2} - \tfrac{e-1}{2e} & \text{if } 2 \nmid e > 1, 2 \mid l \\
		\tfrac{l-1}{2} & \text{if } 2 \mid e \text{ or } 2 \nmid e > 1, 2 \nmid l	
	\end{cases} \quad \Rightarrow \quad \lceil \tfrac{m}{e} \rceil = \lfloor \tfrac{l}{2} \rfloor. $$
	
\noindent (1) The function $U_{\oA}^n \to \C, \kappa \mapsto \psi(\Tr(g \kappa))$ is right invariant by $U_{\oA}^{2n}$, since we have $g \kappa \VP^{2n} \subseteq \VP^{1-e} = \VP^*$ by \eqref{eq: POrderRadDual} and $\psi$ is trivial on $\vO_{\F}$. Therefore the integral can be regarded as taken over $\VP^n[0] / \VP^{2n}[0]$ by Proposition \ref{prop: POrderNbhdQ}, especially \eqref{eq: POrderNbhdQExpBis}. Together with \eqref{eq: POrderRadDual} we get
	$$ \oint_{U_{\oA}^n(1)} \psi(\Tr(g \kappa)) \ud \kappa = \psi(\Tr(g)) \oint_{\VP^n[0]} \psi(\Tr(g x)) \ud x = \psi(\Tr(g)) \cdot \id_{\F + \VP^{1-e-n}}(g). $$
	Writing $g = \Pi^{-k} u$ with $u \in \oA^{\times}$ we have
	$$ \Pi^{-k} u \in \F + \VP^{1-e-n} \ \Rightarrow \ u \in \F \Pi^k + \VP^{m} \ \Rightarrow \ \F \Pi^k \cap \oA^{\times} \neq \emptyset \ \Rightarrow \ e \mid k. $$
	Writing $k = el$ and observing the formula \eqref{eq: PowerJRad} we conclude by
	$$ \varpi_{\F}^{-l} u = \Pi^{-k} u \in \F + \VP^{1-e-n} \ \Leftrightarrow \ u \in \left( \F + \VP^{m} \right) \cap \oA^{\times} = \vO_{\F}^{\times} U_{\oA}^m = \widetilde{U}_{\oA}^m. $$
	
\noindent (2) Multiplying $\delta$ from right by some element in $U_{\oA}^m(1)$, we may always assume $\delta = \diag(1+u,1,\dots,1)$ for some $u \in \vP_{\F}^{\lceil m/e \rceil} = \vP_{\F}^{\lfloor l/2 \rfloor}$ by Lemma \ref{lem: DetTrRel}.

\noindent (\rmnum{1}) The case for $m=n$, which is equivalent to $2 \mid l \ \& \ 2 \nmid e$, is easy. 

\noindent (\rmnum{2}) In the case $2 \mid l \ \& \ 2 \mid e$ we write $l=2l_1 \ \& \ e=2e_1$ and $d_1=e_1f$. Let $\oA_1$ be the standard order with parameter $(e_1,f)$, and $\VP_1$ its Jacobson radical. We have $m=e(l_1-1)+e_1$ and by \eqref{eq: PowerJRad} the formula
	$$ \VP^m = \varpi_{\F}^{l_1-1} \begin{pmatrix} \Mat_{d_1}(\vP_{\F}) & \oA_1 \\ \varpi_{\F} \oA_1 & \Mat_{d_1}(\vP_{\F}) \end{pmatrix}, \quad \VP^n = \VP^{m+1} = \varpi_{\F}^{l_1-1} \begin{pmatrix} \Mat_{d_1}(\vP_{\F}) & \VP_1 \\ \varpi_{\F} \VP_1 & \Mat_{d_1}(\vP_{\F}) \end{pmatrix}. $$
	Consequently a system of representatives for $\VP^m[0]/\VP^n[0]$ can be chosen as
	$$ X = \varpi_{\F}^{l_1-1} \begin{pmatrix} & X_1' \\ \varpi_{\F}X_1 & \end{pmatrix} $$
where $X_1$ and $X_1'$ run through a system of representative for $\oA_1/\VP_1$. In particular we have $[U_{\oA}^m(1) : U_{\oA}^n(1)] = \norm[\oA_1/\VP_1]^2 = q_{\F}^{e f^2}$. Applying Lemma \ref{lem: DetTrRel} to $\oA_1$ we have
\begin{multline*} 
	(\id_d+X) \begin{pmatrix} \id_{d_1} & \\ -\varpi_{\F}^{l_1}X_1 & \id_{d_1} \end{pmatrix} = \begin{pmatrix} \id_{d_1} - \varpi_{\F}^{l-1} X_1' X_1 & \varpi_{\F}^{l_1-1} X_1' \\ & \id_{d_1} \end{pmatrix} \Rightarrow \\
	\det(\id_d+X) = \det(\id_{d_1}-\varpi_{\F}^{l-1} X_1' X_1) \equiv 1-\varpi_{\F}^{l-1}\Tr(X_1' X_1) \pmod{\vP_{\F}^{2(l-1)}}.
\end{multline*}
	Since $2(l-1) \geq l$ in the present case, we deduce from \eqref{eq: POrderNbhdQExpBis} that the left hand side is equal to
\begin{multline*}
	\sum_{X_1,X_1'} \psi \left( z \Tr \left( \begin{pmatrix} 1+u & \\ & \id_{d-1} \end{pmatrix} \begin{pmatrix} \id_{d_1} & \varpi_{\F}^{l_1-1} X_1' \\ \varpi_{\F}^{l_1} X_1 & \id_{d_1} \end{pmatrix} \begin{pmatrix} 1 + \varpi_{\F}^{l-1} \Tr (X_1'X_1) & \\ & \id_{d-1} \end{pmatrix} \right) \right) \\
	= \psi(z(u+d)) \sum_{X_1,X_1'} \psi(\varpi_{\F}^{l-1}z \Tr(X_1'X_1)).
\end{multline*}
	For a fixed $X_1'$ the sum over $X_1$ is non-vanishing only if $X_1' \in \VP_1$, since $(\varpi_{\F}^{-1}\oA_1)^* = \VP_1$ by \eqref{eq: POrderRadDual}; while in the case $X_1' \in \VP_1$ the sum over $X_1$ is equal to $\norm[\oA_1/\VP_1] = q_{\F}^{e_1f^2}$, proving the stated formula.
	
\noindent (\rmnum{3}) In the case $2 \nmid l$ we write $l=2l_1+1$. We have $m=el_1 \geq 1$. Hence we have $l_1 \geq 1$ and by \eqref{eq: PowerJRad}
	$$ \VP^m = \varpi_{\F}^{l_1} \oA, \quad \VP^n = \VP^{m+1} = \varpi_{\F}^{l_1} \VP. $$
	A system of representatives for $\VP^m[0]/\VP^n[0]$ can be chosen as \emph{block} diagonal matrices, i.e.,
	$$ X = (x_{ij}) \in \begin{pmatrix} [\vP_{\F}^{l_1}]_f & & \\ & \ddots & \\ & & [\vP_{\F}^{l_1}]_f \end{pmatrix} / \begin{pmatrix} [\vP_{\F}^{l_1+1}]_f & & \\ & \ddots & \\ & & [\vP_{\F}^{l_1+1}]_f \end{pmatrix}, \quad \sum_{i=1}^d x_{ii} = 0. $$
	In particular we have $[U_{\oA}^m(1) : U_{\oA}^n(1)] = q_{\F}^{e f^2-1}$. Then with $S_d$ as in \eqref{eq: DetMod} we have
\begin{multline*}
	\det(\id_d+X) = \prod_{i=1}^d(1+x_{ii}) + \sum_{\id \neq \sigma \in S_d} \sgn(\sigma) \prod_{i=1}^d (\delta_{i,\sigma(i)}+x_{i,\sigma(i)}) \\
	\equiv 1 + \sum_{1 \leq i < j \leq d} (x_{ii}x_{jj}-x_{ij}x_{ji}) \pmod{\vP_{\F}^{3l_1}}.
\end{multline*}
	In fact those $\id \neq \sigma$ with $< d-2$ fixed points contribute to $\vP_{\F}^{3l_1}$, while those $\id \neq \sigma$ with $d-2$ fixed points are precisely the transpositions. Since $3l_1 \geq l$ in the present case, we deduce from \eqref{eq: POrderNbhdQExpBis} that the left hand side is equal to
\begin{multline*}
	\sum_X \psi \left( z \Tr \left( \begin{pmatrix} 1+u & \\ & \id_{d-1} \end{pmatrix} \left( \id_d + X \right) \begin{pmatrix} 1 - \sum_{1 \leq i < j \leq d} (x_{ii}x_{jj}-x_{ij}x_{ji}) & \\ & \id_{d-1} \end{pmatrix} \right) \right) \\
	= \sum_X \psi \left( z \left( (1+u)(1+x_{11}) \left( 1 - \sum_{1 \leq i < j \leq d} (x_{ii}x_{jj}-x_{ij}x_{ji}) \right) + \sum_{i=2}^d (1+x_{ii}) \right) \right) \\
	= \sum_X \psi \left( z \left\{ (1+u)(1+x_{11}) + \sum_{i=2}^d (1+x_{ii}) - \sum_{1 \leq i < j \leq d} x_{ii}x_{jj} \right\} \right) \cdot \psi \left( z \sum_{1 \leq i < j \leq d} x_{ij}x_{ji} \right).
\end{multline*}
	The first factor in the summand involves only the diagonal entries of $X$ while the second involves only the off-diagonal ones. We can perform the summations separately. Obviously the second factor contributes a factor of $q_{\F}^{ef(f-1)/2}$. To transform the first factor, we make the change of variables
	$$ \delta_i = 1+x_{ii}, \ \forall \ 1 \leq i \leq d-1; \quad \delta_d := 1+x_{dd} - \sum_{1 \leq i < j \leq d} x_{ii}x_{jj}. $$
	It is easy to see $\prod_{i=1}^{d} \delta_i \equiv 1 \pmod{\vP_{\F}^{3l_1}}$. Replacing $\delta_1$ by $\left( \prod_{i=2}^{d} \delta_i \right)^{-1}$ does not change the first factor, the sum of which is then identified with
	$$ \sum_{\delta_i \in U_{\F}^{l_1}/U_{\F}^{l_1+1}} \psi \left( z \left\{ \frac{1+u}{\prod_{i=2}^{d} \delta_i} + \sum_{i=2}^{d} \delta_i \right\} \right). $$
	The stated formula then follows readily.
\end{proof}

\section{Voronoi--Hankel Kernel Function: non-Archimedean Case}

	\subsection{Supercuspidal Case}
	
	Let $\pi$ be given by \eqref{eq: ScBHModel}. For any $\phi \in \Cont_c^{\infty}(\F^{\times})$, we define a bi-linear form
\begin{equation} \label{eq: AuxHermF}
	I_{\phi}: V_{\pi}^{\infty} \times V_{\widetilde{\pi}}^{\infty} \to \C, \quad (f_1,f_2) \mapsto \int_{\gp{G}} \phi(\det(g)) \psi(\Tr(g)) \beta_{f_1,f_2}(g) \ud g. 
\end{equation}
	Note that for any $y \in \gp{G}$ we have by definition and the change of variables $x \mapsto xy^{-1}$
	$$ \beta_{\pi(y).f_1,\widetilde{\pi}(y).f_2}(g) = \int_{\gp{J}} \extPairing{f_1(xgy)}{f_2(xy)}_{\rho} \ud x = \int_{\gp{J}} \extPairing{f_1(xy^{-1}gy)}{f_2(x)}_{\rho} \ud x = \beta_{f_1,f_2}(y^{-1}gy). $$
	It follows that $I_{\phi}(\pi(y).f_1,\widetilde{\pi}(y).f_2) = I_{\phi}(f_1,f_2)$ since $\Tr(\cdot)$ and $\det(\cdot)$ are invariant by conjugation in $\gp{G}$. Therefore $I_{\phi}$ is a $\gp{G}$-invariant bi-linear form on $V_{\pi}^{\infty} \times V_{\widetilde{\pi}}^{\infty}$, which must be a constant multiple of the standard pairing on $V_{\pi}^{\infty} \times V_{\widetilde{\pi}}^{\infty}$ by irreducibility. Namely for some $c_{\pi}(\phi) \in \C$ we have
\begin{equation} \label{eq: HermFComp}
	I_{\phi}(f_1,f_2) = c_{\pi}(\phi) \cdot \int_{\gp{J} \backslash \gp{G}} \extPairing{f_1(x)}{f_2(x)}_{\rho} \ud x = c_{\pi}(\phi) \cdot \beta_{f_1,f_2}(\id). 
\end{equation}
	On the other hand, inserting the definition of $\beta_{f_1,f_2}$ and changing the order of integration we find
\begin{multline} \label{eq: AuxHermFBis}
	I_{\phi}(f_1,f_2) = \int_{\gp{G}} \phi(\det(g)) \psi(\Tr(g)) \left( \int_{\gp{J} \backslash \gp{G}} \extPairing{f_1(xg)}{f_2(x)}_{\rho} \ud x \right) \ud g \\
	= \int_{\gp{J} \backslash \gp{G}} \left( \int_{\gp{G}} \phi(\det(x^{-1}g)) \psi(\Tr(x^{-1}g)) \extPairing{f_1(g)}{f_2(x)}_{\rho} \ud g \right) \ud x \\
	= \int_{(\gp{J} \backslash \gp{G})^2} \left( \int_{\gp{J}} \phi(\det(x_2^{-1}xx_1)) \psi(\Tr(x_2^{-1}xx_1)) \extPairing{\rho(x).f_1(x_1)}{f_2(x_2)}_{\rho} \ud x \right) \ud x_1 \ud x_2.
\end{multline}
	Note that the integral over $(\gp{J} \backslash \gp{G})^2$ is a finite discrete sum, since $\gp{J}$ is open in $\gp{G}$ and $f_j$'s have compact support modulo $\gp{J}$. Taking $f_j$ to be supported in a single coset and comparing \eqref{eq: HermFComp} with \eqref{eq: AuxHermFBis} we deduce
	$$ \int_{\gp{J}} \phi(\det(x_2^{-1}xx_1)) \psi(\Tr(x_2^{-1}xx_1)) \extPairing{\rho(x).f_1(x_1)}{f_2(x_2)}_{\rho} \ud x = c_{\pi}(\phi) \cdot \id_{\gp{J}}(x_1x_2^{-1}) \cdot \extPairing{f_1(x_1)}{f_2(x_1)}_{\rho}. $$
	Equivalently, taking into account the invariance of the trace and the determinant we have
\begin{equation} \label{eq: OrthRel}
	\int_{\gp{J}} \phi(\det(xy)) \psi(\Tr(xy)) \extPairing{\rho(x).v_1}{v_2}_{\rho} \ud x = c_{\pi}(\phi) \cdot \id_{\gp{J}}(y) \cdot \extPairing{v_1}{\widetilde{\rho}(y).v_2}_{\rho}
\end{equation}
	for any $(v_1,v_2) \in V_{\rho} \times V_{\widetilde{\rho}}$. Summing \eqref{eq: OrthRel} for $y=\id_d$ and $(v_1,v_2)$ in dual bases we get
\begin{equation} \label{eq: OrthRelConst}
	c_{\pi}(\phi) = \frac{1}{d_{\rho}} \int_{\gp{J}} \phi(\det(x)) \psi(\Tr(x)) \chi_{\rho}(x) \ud x,
\end{equation}
	where $\chi_{\rho}$ is the character of $\rho$ and $d_{\rho} = \chi_{\rho}(\id_d)$ is the dimension/degree of $\rho$. We summarize the above discussion in the following lemma.
	
\begin{lemma} \label{lem: VHKernScAux}
	Let $\pi$ be a supercuspidal representation given by \eqref{eq: ScBHModel}. For any smooth matrix coefficient $\beta$ and any $\phi \in \Cont_c^{\infty}(\F^{\times})$ we have
	$$ \int_{\gp{G}} \phi(\det(g)) \psi(\Tr(g)) \beta(g) \ud g = \beta(\id) \cdot \frac{1}{d_{\rho}} \int_{\gp{J}} \phi(\det(x)) \psi(\Tr(x)) \chi_{\rho}(x) \ud x. $$
\end{lemma}

\begin{proposition} \label{prop: VHKernSc}
	Let $\pi$ be a supercuspidal representation given by \eqref{eq: ScBHModel}, with central character $\omega$.
\begin{itemize}
	\item[(1)] The Voronoi--Hankel transform $\VorH_{\pi}$ is of convolution type with convolution kernel defined by
	$$ \VHF_{\pi}(t) = \norm[t]_{\F}^{\frac{d+1}{2}} \frac{1}{d_{\widetilde{\rho}}} \int_{\gp{J}(t)} \psi(\Tr(g)) \chi_{\widetilde{\rho}}(g) \ud g. $$
	\item[(2)] The function $\VHF_{\pi}(t)$ is locally constant (but not smooth in the sense of Definition \ref{def: SmoothFDef}) on $\F^{\times}$, and vanishes if $\norm[t]_{\F} \leq 1$. 
	\item[(3)] Write $l_0 = \lceil (2\cond(\pi)+f(\pi))/d \rceil - 1$. If $v_{\F}(t) \leq -dl_0$ and $\VHF_{\pi}(t) \neq 0$, then we have $v_{\F}(t) = -dl$ for some $l \geq l_0$ and
	$$ \VHF_{\pi}(t) = \norm[t]_{\F} \cdot G_l(t; \omega^{-1}, d). $$
	\item[(4)] The function $\VHF_{\pi}(t)$ satisfies for $\Re(s) \gg 1$ and any $\chi \in \widehat{\F^{\times}}$
	$$ \int_{\F^{\times}} \VHF_{\pi}(t) \chi^{-1}(t) \norm[t]_{\F}^{-s} \ud^{\times} t = \varepsilon(s, \pi \times \chi, \psi). $$
	Moreover, the above equality holds for all $s \in \C$ by regularized Mellin transform in Definition \ref{def: RegMellinT}
	$$ \int_{\F^{\times}}^{\mathrm{pv}} \VHF_{\pi}(t) \chi^{-1}(t) \norm[t]_{\F}^{-s} \ud^{\times} t = \varepsilon(s, \pi \times \chi, \psi). $$
	\item[(5)] If $v_{\F}(t) > -dl_0$, then for any $\delta \in U_{\F}^{l_0-1}$ we have $\VHF_{\pi}(t \delta) = \VHF_{\pi}(t)$.
	
\end{itemize}
\end{proposition}
\begin{proof}
	(1) Since $\pi$ is supercuspidal we have $\VorH(\pi) = \Cont_c^{\infty}(\F^{\times})$ (e.g. see the proof of \cite[Proposition (2.2)]{JPS79} taking into account that the Jacquet modules of $\pi$ are $0$). For $\phi \in \Cont_c^{\infty}(\F^{\times})$, recall $\VorH_{\pi}(\phi) \in \Cont_c^{\infty}(\F^{\times})$ is characterized by the following equation (for any $\Phi \in \Sch(\Mat_d(\F))$ and smooth matrix coefficient $\beta$ of $\pi$)
\begin{equation} \label{eq: VorHViaHDFour}
	\int_{\gp{G}} \widehat{\Phi}(g) \beta^{\iota}(g) \norm[\det g]_{\F}^{\frac{d+1}{2}} \cdot \phi(\det g) \ud g = \int_{\gp{G}} \Phi(g) \beta(g) \norm[\det g]_{\F}^{\frac{d-1}{2}} \cdot \VorH_{\pi}(\phi)(\det g) \ud g
\end{equation}
	by Theorem \ref{thm: ExtVorH}. Note that both $\beta$ and $\beta^{\iota}$ have compact support modulo the center. We can therefore replace $\widehat{\Phi}$ by its defining integral and change the order of integration to obtain
\begin{multline} \label{eq: InsertHatFour}
	\int_{\gp{G}} \widehat{\Phi}(g) \beta^{\iota}(g) \norm[\det g]_{\F}^{\frac{d+1}{2}} \cdot \phi(\det g) \ud g = \\
	\int_{\Mat_d(\F)} \Phi(X) \left( \int_{\gp{G}} \psi(\Tr(g X^T)) \beta^{\iota}(g) \norm[\det g]_{\F}^{\frac{d+1}{2}} \phi(\det g) \ud g \right) \ud X. 
\end{multline}
	If $0 \neq X \notin \GL_d(\F)$, then we can find $g_1,g_2 \in \GL_d(\F)$ and $1 \leq s \leq d-1$ such that
	$$ X = g_1 \begin{pmatrix} \id_s & \\ & 0_{d-s} \end{pmatrix} g_2 \quad \Rightarrow \quad g_2^T \begin{pmatrix} \id_s & U \\ & \id_{d-s} \end{pmatrix} g_2^{\iota} X^T = X^T, \ \forall \  U \in \Mat_{s \times (d-s)}(\F). $$
	But $g_2^T \begin{pmatrix} \id_s & U \\ & \id_{d-s} \end{pmatrix} g_2^{\iota}$ runs through the unipotent radical of a parabolic subgroup of $\gp{G}$, along a right coset of which the integral of $\beta^{\iota}$ is vanishing by cuspidality while $\det$ is constant. Hence the inner integral on the right hand side of \eqref{eq: InsertHatFour} is vanishing for $X \notin \GL_d(\F)$. Writing $X = x \in \GL_d(\F)$ we get
\begin{multline} \label{eq: ElimSingEle}
	\int_{\gp{G}} \widehat{\Phi}(g) \beta^{\iota}(g) \norm[\det g]_{\F}^{\frac{d+1}{2}} \cdot \phi(\det g) \ud g = \\
	\int_{\gp{G}} \Phi(x) \norm[\det x]_{\F}^d \left( \int_{\gp{G}} \psi(\Tr(g x^T)) \beta^{\iota}(g) \norm[\det g]_{\F}^{\frac{d+1}{2}} \phi(\det g) \ud g \right) \ud x \\
	= \int_{\gp{G}} \Phi(x) \norm[\det x]_{\F}^{\frac{d-1}{2}} \left( \int_{\gp{G}} \psi(\Tr(g)) \beta^{\iota}(g x^{\iota}) \norm[\det g]_{\F}^{\frac{d+1}{2}} \phi(\det g \cdot (\det x)^{-1}) \ud g \right) \ud x.
\end{multline}
	Note that $g \mapsto \beta^{\iota}(g x^{\iota})$ is a smooth matrix coefficient of $\widetilde{\pi}$, and $t \mapsto \norm[t]_{\F}^{\frac{d+1}{2}} \phi(t \cdot (\det x)^{-1})$ lies in $\Cont_c^{\infty}(\F^{\times})$, we can apply Lemma \ref{lem: VHKernScAux} to the inner integral on the right hand side of \eqref{eq: ElimSingEle} and get
\begin{multline} \label{eq: ScVHFInt}
	\int_{\gp{G}} \widehat{\Phi}(g) \beta^{\iota}(g) \norm[\det g]_{\F}^{\frac{d+1}{2}} \cdot \phi(\det g) \ud g = \\
	\int_{\gp{G}} \Phi(x) \beta^{\iota}(x^{\iota}) \norm[\det x]_{\F}^{\frac{d-1}{2}} \cdot \left( \frac{1}{d_{\widetilde{\rho}}} \int_{\gp{J}} \norm[\det g]_{\F}^{\frac{d+1}{2}} \phi(\det(g) \cdot (\det x)^{-1}) \psi(\Tr(g)) \chi_{\widetilde{\rho}}(g) \ud g \right) \ud x.
\end{multline}
	Comparing \eqref{eq: ScVHFInt} with \eqref{eq: VorHViaHDFour} we deduce
\begin{multline*} 
	\VorH_{\pi}(\phi)(t) = \frac{1}{d_{\widetilde{\rho}}} \int_{\gp{J}} \norm[\det g]_{\F}^{\frac{d+1}{2}} \phi(\det(g) t^{-1}) \psi(\Tr(g)) \chi_{\widetilde{\rho}}(g) \ud g \\
	= \int_{\F^{\times}} \phi(y t^{-1}) \cdot \left( \norm[y]_{\F}^{\frac{d+1}{2}} \frac{1}{d_{\widetilde{\rho}}} \int_{\gp{J}(y)} \psi(\Tr(g)) \chi_{\widetilde{\rho}}(g) \ud g \right) \ud^{\times} y,
\end{multline*}
	which shows that $\VorH_{\pi}$ is of convolution type with kernel given by the stated formula.
	
\noindent (2) It is clear that $\VHF_{\pi}(t)$ is locally constant. We may assume $\gp{J} = \nG(\oA)$ for a principal order $\oA$ with parameter $(e,f)$ given by \eqref{eq: POrder}. Observing \eqref{eq: POrderNmSubGp} we see that $\gp{J}(t) \subset \Mat_d(\vO_{\F})$ for $\norm[t]_{\F} \leq 1$, on which $\psi(\Tr(g)) = 1$. The stated vanishing of $\VHF_{\pi}(t)$ follows from the vanishing of the following integral
	$$ \int_{\gp{J}(1)} \chi_{\widetilde{\rho}}(x g) \ud g = 0, $$
a consequence of Lemma \ref{lem: NoTrivialRp}. The non-smoothness of $\VHF_{\pi}(t)$ obviously follows from its formula near $\infty$.

\noindent (3) Obviously we have $\gp{J}(t) = \emptyset$ if $f \nmid v_{\F}(t)$. If $v_{\F}(t) = -kf$ we have $\gp{J}(t) \subset \Pi^{-k} \oA^{\times}$. Write $c = \cond(\widetilde{\rho})$ for the conductor exponent of $\widetilde{\rho}$, defined by \eqref{eq: CIDCondExp}. Suppose $k \in \Z_{\geq 1}$ satisfies 
\begin{equation} \label{eq: LargeAsympCond1}
	n := \lceil (k+1-e)/2 \rceil \geq m := \lfloor (k+1-e)/2 \rfloor \geq c (\geq 1).
\end{equation}
	For such $t$ we have (recall the subgroups $U_{\oA}^n(1)$ given by \eqref{eq: POrderNbhd1Var})
	$$ \frac{1}{d_{\widetilde{\rho}}} \int_{\gp{J}(t)} \psi(\Tr(g)) \chi_{\widetilde{\rho}}(g) \ud g = \frac{1}{d_{\widetilde{\rho}}} \int_{\gp{J}(t)} \left( \oint_{U_{\oA}^n(1)} \psi(\Tr(g \kappa)) \ud \kappa \right) \chi_{\widetilde{\rho}}(g) \ud g. $$
	By Lemma \ref{lem: AddCharAvg} (1) the inner integral is non-vanishing only if for some $l \in \Z_{\geq 1}$ we have $k=el$ and $g \in \varpi_{\F}^{-l} \widetilde{U}_{\oA}^m$. Under these conditions it is equal to $\psi(\Tr(g))$ and is constant on left cosets of $U_{\oA}^n(1)$. Note that $\chi_{\widetilde{\rho}}(g)$ is constant on left cosets of $U_{\oA}^m$, we get
\begin{multline*} 
	\frac{1}{d_{\widetilde{\rho}}} \int_{\gp{J}(t)} \psi(\Tr(g)) \chi_{\widetilde{\rho}}(g) \ud g = \frac{1}{d_{\widetilde{\rho}}} \int_{\varpi_{\F}^{-l} \widetilde{U}_{\oA}^m(\varpi_{\F}^{ld}t)} \left( \oint_{U_{\oA}^m(1)} \psi(\Tr(g \kappa)) \ud \kappa \right) \chi_{\widetilde{\rho}}(g) \ud g \\
	= \frac{\Vol(U_{\oA}^m(1))}{d_{\widetilde{\rho}}} \sum_{g \in \varpi_{\F}^{-l} \widetilde{U}_{\oA}^m(\varpi_{\F}^{ld}t)/U_{\oA}^m(1)} \left\{ \frac{1}{[U_{\oA}^m(1) : U_{\oA}^n(1)]} \sum_{\kappa \in U_{\oA}^m(1)/U_{\oA}^n(1)} \psi(\Tr(g \kappa)) \right\} \chi_{\widetilde{\rho}}(g).
\end{multline*}
	Write $t = \varpi_{\F}^{-ld}y$ for some $y \in \vO_{\F}^{\times}$. By Lemma \ref{lem: DetTrRel} the above sum over $g$ can be taken over $g = \varpi_{\F}^{-l} z \diag(z^{-d}y, 1, \dots, 1)$ for $z \in \vO_{\F}^{\times}/U_{\F}^{\lceil m/e \rceil} = \vO_{\F}^{\times}/U_{\F}^{\lfloor l/2 \rfloor}$ under the condition $z^d \in y U_{\F}^{\lfloor l/2 \rfloor}$. With such choice of representatives we have $\chi_{\widetilde{\rho}}(g) = d_{\widetilde{\rho}} \omega^{-1}(\varpi_{\F}^{-l}z)$. We apply Lemma \ref{lem: AddCharAvg} to the sum over $\kappa$ and get
\begin{multline*}
	\frac{1}{d_{\widetilde{\rho}}} \int_{\gp{J}(t)} \psi(\Tr(g)) \chi_{\widetilde{\rho}}(g) \ud g = I(e,f;l) \cdot \\
	\Vol(U_{\oA}^m(1)) \sum_{z \in \vO_{\F}^{\times}/U_{\F}^{\lfloor l/2 \rfloor}, z^d \in y U_{\F}^{\lfloor l/2 \rfloor}} \left\{ \sum_{\delta_j \in U_{\F}^{\lfloor l/2 \rfloor} / U_{\F}^{\lceil l/2 \rceil}} \psi \left( z \left( \frac{y}{z^d \delta_2 \cdots \delta_d} + \sideset{}{_{j=2}^d} \sum \delta_j \right) \right) \right\} \omega^{-1}\left( \frac{z}{\varpi_{\F}^l} \right).
\end{multline*}
	Note that $\lceil m / e  \rceil = \lfloor l/2 \rfloor \geq \cond(\omega)$ by Lemma \ref{lem: AddCharAvg} (0) and \eqref{eq: CChCondBd}. With the change of variables $z \leadsto z\delta_d^{-1}$ we can apply Lemma \ref{lem: GermParGI} (2) to get
	$$ \frac{1}{d_{\widetilde{\rho}}} \int_{\gp{J}(t)} \psi(\Tr(g)) \chi_{\widetilde{\rho}}(g) \ud g = I(e,f;l) \cdot \frac{\Vol(U_{\oA}^m(1))}{\Vol(U_{\F}^{\lfloor l/2 \rfloor})^{d-1}} \cdot G_l(\varpi_{\F}^{-ld}y; \omega, d). $$
	Let $\oM = \Mat_d(\vO_{\F})$. Since $\det U_{\oA}^m = U_{\F}^{\lfloor l/2 \rfloor}$ by Lemma \ref{lem: DetTrRel} and \ref{lem: AddCharAvg} (0), we deduce
\begin{multline*} 
	\frac{\Vol(U_{\oA}^m(1))}{\Vol(U_{\F}^{\lfloor l/2 \rfloor})^{d-1}} = \frac{\Vol(U_{\oA}^m)}{\Vol(U_{\F}^{\lfloor l/2 \rfloor})^{d}} = \frac{\Vol(\oA^{\times})}{\Vol(U_{\F})^d} \cdot \frac{q_{\F}^{d \lfloor \frac{l}{2} \rfloor} (1-q_{\F}^{-1})^d}{q_{\F}^{mef^2} \left( \prod_{i=1}^{f} (1-q_{\F}^{-i}) \right)^e} \\
	= \frac{\Vol(\oA)}{\Vol(\vO_{\F})^d} \cdot \frac{q_{\F}^{d \lfloor \frac{l}{2} \rfloor}}{q_{\F}^{mef^2}} = \frac{\Vol(\oM)}{\Vol(\vO_{\F})^d} \cdot q_{\F}^{d \lfloor \frac{l}{2} \rfloor - mef^2 - \frac{e(e-1)f^2}{2}}.
\end{multline*}
	Note the conductor exponent $\cond(\psi) = 0$, we have $\Vol(\oM) = \Vol(\vO_{\F}) = 1$. We leave the reader to check
	$$ q_{\F}^{d \lfloor \frac{l}{2} \rfloor - mef^2 - \frac{e(e-1)f^2}{2}} \cdot I(e,f; l) = q_{\F}^{-\frac{d(d-1)l}{2}} = \norm[t]_{\F}^{\frac{1-d}{2}} $$
 	and conclude the proof by rewriting \eqref{eq: LargeAsympCond1} in terms of $l$, taking into account the relation \eqref{eq: CondExpRel}.

\noindent (4) From (2) \& (3) we see that $\VHF_{\pi}$ is bounded by $O(\norm[t]_{\F}^M)$ for some $M>0$ as $\norm[t]_{\F} \to +\infty$, and vanishes as $\norm[t]_{\F} \ll 1$. The absolute convergence of its Mellin transform follows readily. The definition of $\VorH_{\pi}$ given by Definition \ref{def: VoronoiGLn} together with (1) implies for any $h \in \Cont_c^{\infty}(\F^{\times})$
\begin{equation} \label{eq: MellinID}
	\int_{\F^{\times}} \left( \int_{\F^{\times}} h(t) \VHF_{\pi}(ty) \ud^{\times} t \right) \chi^{-1}(y) \norm[y]_{\F}^{-s} \ud^{\times} y = \varepsilon(s, \pi \times \chi, \psi) \cdot \int_{\F^{\times}} h(y) \chi(y) \norm[y]_{\F}^s \ud^{\times} y
\end{equation}
since $\pi \times \chi$ is supercuspidal for any $\chi$. For $\Re s \ll -1$ the left hand side of \eqref{eq: MellinID} is absolutely convergent by Fubini. Hence we can change the order of integration and rewrite it as
	$$ \int_{\F^{\times}} \int_{\F^{\times}} h(t) \VHF_{\pi}(ty) \chi^{-1}(y) \norm[y]_{\F}^{-s} \ud^{\times} t \ud^{\times} y = \int_{\F^{\times}} \VHF_{\pi}(y) \chi^{-1}(y) \norm[y]_{\F}^{-s} \ud^{\times} y \cdot \int_{\F^{\times}} h(t) \chi(t) \norm[t]_{\F}^s \ud^{\times} t. $$
	It suffices to take any $h$ whose Mellin transform is non-vanishing at $\chi \norm_{\F}^s$ to conclude the first equation. Recall the following basic property of $\varepsilon$-factors
	$$ \varepsilon(s, \pi \times \chi, \psi) = \varepsilon(1/2, \pi \times \chi, \psi) q_{\F}^{\left( \frac{1}{2}-s \right) \cond(\pi \times \chi)}. $$
	It allows to rewrite the first equation as
\begin{equation} \label{eq: VHKernSCMellin}
	\int_{\varpi_{\F}^n \vO_{\F}^{\times}} \VHF_{\pi}(y) \chi^{-1}(y) \norm[y]_{\F}^{-s} \ud^{\times} y =
	\begin{cases}
		0 & \text{if } n \neq \cond(\pi \times \chi) \\
		\varepsilon(s, \pi \times \chi, \psi) & \text{if } n = \cond(\pi \times \chi)
	\end{cases}.
\end{equation}
	The above equation obviously extends to all $s \in \C$, and readily yields the second equation.
	
\noindent (5) From (3) and \eqref{eq: VHKernSCMellin} we deduce that the left hand side of \eqref{eq: VHKernSCMellin} is non-zero for some $\chi$ with $\cond(\chi) \geq l_0$ only if $n \leq -dl_0$. Therefore for $n > -dl_0$ the left hand side of \eqref{eq: VHKernSCMellin} is non-zero only if $\cond(\chi) \leq l_0-1$. Hence the function $\vO_{\F}^{\times} \to \C, \ \delta \mapsto \VHF_{\pi}(\varpi_{\F}^n \delta)$ is a linear combination of characters $\chi$ of $\vO_{\F}^{\times}$ with $\cond(\chi) \leq l_0-1$, and is consequently invariant by $U_{\F}^{l_0-1}$.
\end{proof}

\begin{corollary} \label{cor: StRangeSC}
	Let $\pi$ be a supercuspidal representation of $\GL_d(\F)$ with central character $\omega$. If $\chi$ is a character of $\F^{\times}$ with conductor exponent $\cond(\chi) \geq \lceil (2\cond(\pi)+f(\pi))/d \rceil - 1$, then
	$$ \varepsilon(s, \pi \times \chi, \psi) = \varepsilon(s, \chi, \psi)^{d-1} \varepsilon(s, \chi \omega, \psi), \quad \cond(\pi \times \chi) = d \cdot \cond(\chi) $$
for any additive character $\psi$ with $\cond(\psi) = 0$.
\end{corollary}
\begin{proof}
	Write $l = \cond(\chi)$. By Proposition \ref{prop: VHKernSc} (3) and Lemma \ref{lem: GermParGI} (1) we have
	$$ \int_{\varpi_{\F}^{-dl} \vO_{\F}^{\times}} \VHF_{\pi}(t) \chi^{-1}(t) \norm[t]_{\F}^{-s} \ud^{\times} t = \int_{\F^{\times}} G_l(t; \omega^{-1}, d) \chi^{-1}(t) \norm[t]_{\F}^{1-s} \ud^{\times} t = \varepsilon(s, \chi, \psi)^{d-1} \varepsilon(s, \chi \omega, \psi). $$
	We conclude by comparing the above equation with \eqref{eq: VHKernSCMellin}.
\end{proof}

	\subsection{Essentially Tame Case}
	
	We consider a special family of supercuspidal representations, the \emph{essentially tame} ones, introduced and studied by Bushnell--Henniart in a series of papers \cite{BuH05_T1, BuH05_T2, BuH10_T3}. We first recall the basic results in these works. Let $p$ be the characteristic of the residual field $k_{\F}$ of $\F$. Suppose $\pi$ is supercuspidal of $\GL_d(\F)$. Recall $f(\pi)$ is the number of unramified characters $\chi$ of $\F^{\times}$ such that $\pi \otimes \chi \simeq \pi$. Necessarily we have $\chi^d = \id$ by comparison of the central characters. Hence $f(\pi) \mid d$.

\begin{definition} \label{def: EssTame}
	A supercuspidal $\pi$ of $\GL_d(\F)$ is called \emph{essentially tame} if $p$ does not divide $e(\pi) = d/f(\pi)$. The set of essentially tame supercuspidal representations of $\GL_d(\F)$ is denoted by $\AutR_d^{\mathrm{et}}(\F)$.
\end{definition}

\begin{definition} \label{def: AdmP}
	Let $\E/\F$ be a finite, tamely ramified field extension and let $\xi$ be a quasicharacter of $\E^{\times}$. The pair $(\E/\F, \xi)$ is called \emph{admissible} if it satisfies the following two conditions. Let $\gp{K}$ range over intermediate fields, $\F \subset \gp{K} \subset \E$.
\begin{itemize}
	\item[(1)] If $\xi$ factors through the relative norm $\Nr_{\E/\gp{K}}$, then $\gp{K} = \E$.
	\item[(2)] If $\xi \mid U_{\E}^1$ factors through $\Nr_{\E/\gp{K}}$, then $\E/\gp{K}$ is unramified.
\end{itemize}
	Let $P_d(\F)$ denote the set of $\F$-isomorphism classes of admissible pairs $(\E/\F, \xi)$ in which $[\E : \F] = d$. Let $P_d^u(\F)$ be the subset of $(\E/\F, \xi)$ with unitary $\xi$.
\end{definition}

\noindent The set $\AutR_d^{\mathrm{et}}(\F)$ is parametrized by the set $P_d(\F)$. The corresponding map \cite[(4)]{BuH05_T1} is written as
\begin{equation} \label{eq: EssTamePar}
	P_d(\F) \to \AutR_d^{\mathrm{et}}(\F), \quad (\E/\F, \xi) \mapsto \pi_{\xi}.
\end{equation}

\noindent A consequence of the main result \cite[Theorem A]{BuH05_T1} can be presented in the following form.

\begin{theorem} \label{thm: EssTameLocCst}
	For each $(\E/\F,\xi) \in P_d(\F)$ there is a \emph{tamely ramified unitary character} $\mu_{\xi}$ of $\E^{\times}$ so that:
\begin{itemize}
	\item[(1)] the map $P_d(\F) \to P_d(\F), (\E/\F,\xi) \mapsto (\E/\F,\mu_{\xi} \xi)$ is a bijection;
	\item[(2)] the local epsilon factors for $\AutR_d^{\mathrm{et}}(\F)$ are expressed in terms of those of Tate as
\begin{equation} \label{eq: GammaIdLLC}
	\varepsilon(s, \pi_{\xi} \times \chi, \psi) = \lambda(\E/\F, \psi) \cdot \varepsilon \left( s, \mu_{\xi} \xi \cdot (\chi \circ \Nr_{\E/\F}), \psi \circ \Tr_{\E/\F} \right),
\end{equation}
	where $\lambda(\E/\F, \psi)$ is the \emph{Langlands constant}, a root of unity in $\Sph^1 \subset \C^{\times}$.
\end{itemize}
\end{theorem}

\noindent The character $\mu_{\xi}$ is uniquely determined by \eqref{eq: GammaIdLLC} only in some very special cases. The existence and uniqueness of $\mu_{\xi}$, compatible with the local Langlands correspondence, is the full content of \cite[Theorem A]{BuH05_T1}. An explicit description of $\mu_{\xi}$ is the subject of \cite[Theorem B]{BuH05_T1} and the subsequent papers \cite{BuH05_T2, BuH10_T3}.

	We then simplify the integral representation of $\VHF_{\pi}$ given by Proposition \ref{prop: VHKernSc} (2) in the essentially tame case using Bushnell--Henniart results, especially \eqref{eq: GammaIdLLC}, as follows. For simplicity of notation, we write $\psi_{\E} := \psi \circ \Tr_{\E/\F}$, and $\chi_{\E} := \chi \circ \Nr_{\E/\F}$. For any $t \in \F^{\times}$ we introduce the fibre $\E^{\times}(t)$ of $\Nr_{\E/\F}$ and the measure on it in a similar way as Definition \ref{def: FibreOfDet}. For any character $\xi$ of $\E^{\times}$ we introduce the \emph{partial Gauss integral/sum}
\begin{equation} \label{eq: ParGaussInt}
	H_{\xi}(t) := \int_{\E^{\times}(t)} \psi_{\E}(u) \xi(u) \ud u.
\end{equation}
	Arguing in a way similar to the proof of Proposition \ref{prop: VHKernSc} (2), we see that $H_{\xi}(t)$ is a locally constant function on $\F^{\times}$, vanishes as $\norm[t]_{\F} \ll 1$ since $\xi$ is not trivial on $\E^{\times}(1)$, and has polynomial growth as $\norm[t]_{\F} \to +\infty$. For any character $\chi$ of $\F^{\times}$ and any $s \in \C$, the character $\xi \cdot \chi_{\E} \norm_{\E}^s$ is not trivial on $\vO_{\E}^{\times}$, since it is not trivial on $\E^{\times}(1) < \vO_{\E}^{\times}$. We can apply \cite[\S 23.5 Exercise]{BuH06} to get for $\Re s \gg 1$ and any $\chi$
\begin{equation} \label{eq: ParGaussMellin}
	\int_{\F^{\times}} H_{\xi^{-1}}(t) \chi^{-1}(t) \norm[t]_{\F}^{1-s} \ud^{\times} t = \lim_{c \to - \infty} \int_{\substack{\E^{\times} \\ \norm[y]_{\E} \geq c}} \psi_{\E}(y) (\xi \cdot \chi_{\E})^{-1}(y) \norm[y]_{\E}^{1-s} \ud^{\times} y = \zeta_{\E}(1) \cdot \varepsilon(s, \xi \cdot \chi_{\E}, \psi_{\E}).
\end{equation}
	Comparing \eqref{eq: ParGaussMellin} with Proposition \ref{prop: VHKernSc} (2), we \emph{formally} deduce:

\begin{proposition} \label{prop: VHKernET}
	For essentially tame unitary supercuspidal representation $\pi = \pi_{\xi}$ associated with an admissible pair $(\E/\F, \xi)$ by \eqref{eq: EssTamePar}, we have
	$$ \VHF_{\pi}(t) = \zeta_{\E}(1)^{-1} \cdot \lambda(\E/\F, \psi) \cdot H_{(\mu_{\xi}\xi)^{-1}}(t) \cdot \norm[t]_{\F}. $$
\end{proposition}

\noindent Clearly the above formal deduction is justified by the following lemma.

\begin{lemma}
	Let $H(t)$ be a locally constant function on $\F^{\times}$, vanishes as $\norm[t]_{\F} \ll 1$ and has polynomial growth as $\norm[t]_{\F} \to +\infty$. If for $\Re s \gg 1$ and any unitary character $\chi$ of $\F^{\times}$ we have
	$$ \int_{\F^{\times}} H(t) \chi(t) \norm[t]_{\F}^{-s} \ud^{\times} t = 0. $$
	Then $H(t) = 0$ vanishes identically.
\end{lemma}
\begin{proof}
	It suffices to notice that $\int_{\F^{\times}} H(t) \chi(t) \norm[t]_{\F}^{-s} \ud^{\times} t \in \C[[q_{\F}^{-s}]]$ is a Laurent series in $q_{\F}^{-s}$. The vanishing condition is therefore equivalent with the vanishing of
	$$ \int_{\vO_{\F}^{\times}} H(\varpi_{\F}^n t) \chi(t) \ud^{\times} t = 0 $$
for any $n \in \Z$ and any character $\chi$ of $\vO_{\F}^{\times}$. We conclude by the Fourier inversion for locally constant functions on the compact group $\vO_{\F}^{\times}$, since locally constant is the same as smooth in this case.
\end{proof}

	\subsection{General Case}
	\label{sec: VHKernNA}
	
	We give a proof of Theorem \ref{thm: VHKernNA} and Corollary \ref{cor: StRangeGen} as follows.
	
	For simplicity of notation we write $\gp{G}_d = \GL_d(\F)$. Let $\gp{P}$ be a parabolic subgroup of $\gp{G}_d$, whose Levi-factor is isomorphic to $\gp{G}_{d_1} \times \cdots \times \gp{G}_{d_r}$ with $d_1 + \cdots + d_r = d$. Suppose $\sigma_j$ is an admissible representation of $\gp{G}_{d_j}$ for each $j$. We get the tensor product representation $\sigma_1 \boxtimes \cdots \boxtimes \sigma_r$ of $\gp{G}_{d_1} \times \cdots \times \gp{G}_{d_r}$, inflate it to $\gp{P}$ and induce it to $\gp{G}_d$. We denote the obtained representation as
	$$ \Ind(\gp{G}_d, \gp{P}; \sigma_1, \dots, \sigma_r). $$
	
	Let $\pi$ be a \emph{generic}, admissible and irreducible representation of $\GL_d(\F)$. By \cite[Theorem 9.7]{Ze80}, there exist a parabolic subgroup $\gp{Q}$ and \emph{generalized special representations} $\sigma_j$ as above such that
	$$ \pi \simeq \Ind(\gp{G}_d, \gp{Q}; \sigma_1, \dots, \sigma_r). $$
	By definition (see the paragraph between the equations \cite[(7) \& (8)]{JPS83}), for each $j$ there is a decomposition of $d_j = m_j a_j$, a parabolic subgroup $\gp{Q}_j$ of $\gp{G}_{d_j}$ whose Levi-factor is isomorphic to $\gp{G}_{a_j} \times \cdots \times \gp{G}_{a_j}$, and a supercuspidal representation $\xi_j$ of $\gp{G}_{a_j}$ such that $\sigma_j$ is the unique irreducible sub-representation of
	$$ \Ind(\gp{G}_{d_j}, \gp{Q}_j; \xi_j \otimes \norm_{\F}^{(m_j-1)/2}, \dots, \xi_j \otimes \norm_{\F}^{-(m_j-1)/2}). $$
As a result, $\pi$ is an irreducible \emph{component} of an induced representation of the form
	$$ \pi' \simeq \Ind(\gp{G}_d, \gp{P}; \pi_1, \dots, \pi_k), $$
where each $\pi_j$ is supercuspidal of some $\gp{G}_{k_j}$. The collection $\SuS(\pi) := \{ \pi_1, \dots, \pi_k \}$ is uniquely determined by $\pi$, and called the \emph{supercuspidal support} of $\pi$. Note that for any (quasi-)character $\chi$ of $\F^{\times}$, the representation $\pi \otimes \chi$ is an irreducible component of $\pi' \otimes \chi$ and we have $\SuS(\pi \otimes \chi) := \{ \pi_1 \otimes \chi, \dots, \pi_k \otimes \chi \}$. By the discussion on \cite[p.34]{GoJa72} and the multiplicativity of the gamma factors \cite[Theorem 3.4]{GoJa72} we have
\begin{equation} \label{eq: MultGammaF}
	\gamma(s, \pi \otimes \chi, \psi) = \gamma(s, \pi' \otimes \chi, \psi) = \prod_{j=1}^k \gamma(s, \pi_j \otimes \chi, \psi), 
\end{equation}
although $L(s, \pi \otimes \chi)/L(s, \pi' \otimes \chi)$ is in general a (non-constant) polynomial $P$ in $q_{\F}^{-s}$ with $P(0)=1$. By Proposition \ref{prop: VHKernSc} and explicit formula for $\gp{G}_1$, we have $\VHF_{\pi_j} \in \SSch(\F, \widetilde{\F^{\times}})$. For any $h \in \Cont_c^{\infty}(\F^{\times})$ it makes sense to define $h_1 := \left( \VHF_{\pi_1} *_{\mathrm{pv}} \cdots *_{\mathrm{pv}} \VHF_{\pi_k} \right) * \Inv(h)$ by Proposition \ref{prop: SSchAlgBis}. Both $h_1$ and $h_2 := \VorH_{\pi}(h)$ lie in $\SSch(\F)$, and by Proposition \ref{prop: SSchAlg}, Proposition \ref{prop: VHKernSc} (4) and Definition \ref{def: VoronoiGLn} satisfy
	$$ \int_{\F^{\times}} h_j(t) \chi^{-1}(t) \norm[t]_{\F}^{-s} \ud^{\times} t = \gamma(s, \pi \otimes \chi, \psi) \cdot \int_{\F^{\times}} h(t) \chi(t) \norm[t]_{\F}^s \ud^{\times} t. $$
	We get $h_1 = h_2$ by Proposition \ref{prop: PaleyWienerNA}, proving the first equation in Theorem \ref{thm: VHKernNA} (1). The second equation follows readily from the first one, Proposition \ref{prop: SSchAlg} (2), Proposition \ref{prop: VHKernSc} (4) as well as its analogue for $\GL_1$ which can be verified by hand.
	
	Write $l_i := l_0(\pi_i)$ for simplicity. Let $\omega_i$ be the central character of $\pi_i$. Assume $\pi_i$ is a representation of $\GL_{d_i}(\F)$. We claim that each $\VHF_{\pi_i}$ can be decomposed as
\begin{equation} \label{eq: VHFDecomp1}
	\VHF_{\pi_i}(t) = \norm[t]_{\F} \cdot G_{\geq l_i}(t; \omega_i^{-1}, d_i) + f_i(t)
\end{equation}
	where $f_i \in \SSch(\F)$ has support contained in $\left\{ t \ \middle| \ v_{\F}(t) > - d_i l_i  \right\}$, and is invariant by $U_{\F}^{l_i-1}$. In fact, if $d_i \geq 2$, the claim follows from Proposition \ref{prop: VHKernSc} (5); while if $d_i = 1$ then $\VHF_{\pi_i}(t) = \norm[t]_{\F} \cdot \pi_i^{-1}(t) \psi(t)$, as is easily verified by definition of $\VorH_{\pi_i} = \Mult_1(\pi_i^{-1}) \circ \invOFour \circ \Mult_0(\pi_i^{-1})$ in Definition \ref{def: VoronoiGLn} (2), and the claim is obvious. By \eqref{lem: GermParGI} the equation \eqref{eq: VHFDecomp1} can also be written, with $l_0 := l_0(\pi) = \max \left\{ l_i \ \middle| \ 1 \leq i \leq k \right\}$, as
\begin{equation} \label{eq: VHFDecomp2}
	\VHF_{\pi_i}(t) = \norm[t]_{\F} \cdot G_{\geq l_0}(t; \omega_i^{-1}, d_i) + g_i(t)
\end{equation}
	where $g_i \in \SSch(\F)$ has support contained in $\left\{ t \ \middle| \ v_{\F}(t) > - d_i l_0  \right\}$, and is invariant by $U_{\F}^{l_0-1}$. The equation \eqref{eq: 2ndRegConvAux} implies readily that $\norm[t]_{\F} \cdot G_{\geq l_0}(t; \omega_i^{-1}, d_i)$ is ``orthogonal'' to $g_j(t)$ with respect to $*_{\mathrm{pv}}$ for any $i,j$. Therefore we get by \eqref{eq: 1stRegConvAux}
\begin{equation} \label{eq: VHFDecomp3}
	\VHF_{\pi}(t) = \norm[t]_{\F} \cdot G_{\geq l_0}(t; \omega^{-1}, d) + \left( g_1 *_{\mathrm{pv}} \cdots *_{\mathrm{pv}} g_k \right)(t). 
\end{equation}
	Obviously the support of $g_1 *_{\mathrm{pv}} \cdots *_{\mathrm{pv}} g_k$ is contained in $\left\{ t \ \middle| \ v_{\F}(t) > - d l_0  \right\}$. Theorem \ref{thm: VHKernNA} (2) follows readily from \eqref{eq: VHFDecomp3}.
	
	Turning to Corollary \ref{cor: StRangeGen}, the regularized Mellin transform of $g_1 *_{\mathrm{pv}} \cdots *_{\mathrm{pv}} g_k$ related to $\chi$ with $\cond(\chi) \geq l_0(\pi)$ vanishes, since each $g_j$ and hence $g_1 *_{\mathrm{pv}} \cdots *_{\mathrm{pv}} g_k$ is invariant by $U_{\F}^{l_0-1}$. We conclude the desired formula by comparing Theorem \ref{thm: VHKernNA} (1) with Lemma \ref{lem: GermParGI} (1). Alternatively one may combine \eqref{eq: MultGammaF} with Corollary \ref{cor: StRangeSC} and its analogue for $\GL_1$, i.e., \cite[23.8]{BuH06} to conclude.

\section*{Acknowledgement}

	Han Wu would like to thank Yueke Hu for several advisory discussions with generous sharing of knowledge. Part of the work is done during Wu's visit at the IASM of Zhejiang University. Wu would like to thank Professor Binyong Sun and IASM for hospitality, and the support from the New Cornerstone Science Foundation.

\appendix

\section{Equivalence via Paley--Wiener}
\label{sec: EqViaPW}

	\subsection{Paley--Wiener Type Isomorphisms}

	Jiang--Luo and Jiang--Li's treatment of Theorem \ref{thm: EquivDefs} (1) is based on two Paley--Wiener type theorems. To formulate them, we need to define the relevant space of meromorphic functions $\MF_{\pi}$. To this end, we first note a direct product decomposition of $\F^{\times}$ as
\begin{equation} \label{eq: LocFMGDecomp}
	\F^{\times} = R_{\F} \times \Omega_{\F}, \quad R_{\F} =
	\begin{cases}
		\R_{>0} & \text{if } \F \in \{ \R, \C \} \\
		\varpi_{\F}^{\Z} & \text{if } \F \text{ is non-archimedean}
	\end{cases}.
\end{equation}
	It allows to parametrize the group $\QC(\F^{\times}) = \QC(R_{\F}) \times \widehat{\Omega_{\F}}$ of quasi-characters of $\F^{\times}$ with
\begin{multline} \label{eq: LocFMGQuasiChar}
	\widehat{R_{\F}} = \norm_{\F}^{i\R} \subset \QC(R_{\F}) = \norm_{\F}^{\C} \simeq 
	\begin{cases}
		\C & \text{if } \F \in \{ \R, \C \} \\
		\C/2\pi i (\log q_{\F})^{-1} \Z & \text{if } \F \notin \{ \R, \C \}
	\end{cases}, \\
	\widehat{\Omega_{\F}} = \ac^{\Z} \simeq
	\begin{cases}
		\Z/2\Z & \text{if } \F = \R \\
		\Z & \text{if } \F = \C
	\end{cases}.
\end{multline}
	In the non-archimedean case we use the variable $X := q_{\F}^{-s}$ to define functions on $\QC(R_{\F})$ as in Proposition \ref{prop: PaleyWienerNA}; while in the archimedean case we have used $\ac(t) := t/\norm[t]$. We shall adapt the Mellin--Fourier transform such as given in Proposition \ref{prop: PaleyWienerNA} with the above parametrization of $\QC(\F^{\times})$. Let $\pi$ be an irreducible admissible representation of $\GL_n(\F)$, which is assumed to be Casselman--Wallach for archimedean $\F$. 

\begin{definition} \label{def: MFNA}
	For non-archimedean $\F$, let $\MF_{\pi}$ be the space of functions $\mathfrak{z}$ on $\QC(\F^{\times})$ with finite support in $\widehat{\Omega_{\F}}$, and for any $\xi \in \widehat{\Omega_{\F}}$, extended to $\F^{\times}$ by the trivial character of $R_{\F}$, $\mathfrak{z}(X, \xi) \in L(s, \pi \times \xi) \cdot \C[X, X^{-1}]$. 
\end{definition}

\begin{definition} \label{def: MFA}
	For archimedean $\F$, let $\MF_{\pi}$ be the space of functions $\mathfrak{z}$ on $\QC(\F^{\times})$ satisfying:
\begin{itemize}
	\item[(1)] For any $p \in \widehat{\Omega_{\F}}$, the function $\mathfrak{z}(s,p)/L(s, \pi \times \ac^p)$ is entire in $s \in \C$;
	\item[(2)] For any $a < b$ and any polynomial $P \in \C[X, Y]$ such that $s \mapsto P(s,p) L(s, \pi \times \ac^p)$ is holomorphic in $a \leq \Re(s) \leq b$ for all $p$, we have $\sup_{p \in \widehat{\Omega_{\F}}} \sup_{a \leq \Re(s) \leq b} \extnorm{P(s,p) \mathfrak{z}(s, p)} < + \infty$.
\end{itemize}
\end{definition}

\begin{remark}
	For any $\delta > 0$, let $C(\delta; \pi, p)$ be the union of balls with radius $\delta$ centered at the poles of $L(s, \pi \times \ac^p)$. The condition (2) in Definition \ref{def: MFA} has several equivalent forms:
\begin{itemize}
	\item[(2.1)] For any $a < b$ and any polynomial $P(s, p) \in \C[X, Y]$ such that $s \mapsto P(s,p) L(s, \pi \times \ac^p)$ is \emph{bounded} in $a \leq \Re(s) \leq b$ for all $p$, we have $\sup_{p \in \widehat{\Omega_{\F}}} \sup_{a \leq \Re(s) \leq b} \extnorm{P(s,p) \mathfrak{z}(s, p)} < + \infty$;
	\item[(2.2)] For any $a < b$, $\delta > 0$, and any polynomial $P(s, p) \in \C[X, Y]$ we have 
	$$ \sideset{}{_{p \in \widehat{\Omega_{\F}}}} \sup \sideset{}{_{a \leq \Re(s) \leq b, s \notin C(\delta; \pi,p)}} \sup \extnorm{P(s,p) \mathfrak{z}(s, p)} < +\infty. $$
\end{itemize}
	These forms appeared in literature such as \cite{Qi20, JiL23}. We will use any of them at convenience.
\end{remark}

\begin{theorem} \label{thm: PW}
	(1) The Mellin--Fourier transform of any $f \in S_{\pi}(\F^{\times}) := \norm_{\F}^{-1/2} \cdot \Sch_{\pi}(\F^{\times})$ is absolutely convergent for $\Re(s) > \sigma_{\pi}$ with $s \in \QC(R_{\F})$ and some constant $\sigma_{\pi}$ depending only on $\pi$, has meromorphic continuation to all $s \in \C$, and maps $S_{\pi}(\F^{\times})$ bijectively to $\MF_{\pi}$.
	
\noindent (2) If $\pi$ is generic, then the assertions in (1) hold true with $S_{\pi}(\F^{\times})$ replaced by $\VorH(\pi)$.
\end{theorem}

	Before attacking this main result of the appendix, we note that the assertions on the absolute convergence and meromorphic continuations are included in the standard Godement--Jacquet theory for (1) and Rankin--Selberg theory for (2). We also note that it suffices to prove the \emph{surjectivity} in the last assertion. In fact, granted the surjectivity we can deduce the injectivity via a \emph{distributional} argument as follows. Observe that $S_{\pi}(\F^{\times}), \VorH(\pi) \subset \Cont^{\infty}(\F^{\times})$ by definition. Suppose $f \in S_{\pi}(\F^{\times})$ (or $\VorH(\pi)$) has $\Mellin{f} = 0$. Take any $\varphi \in \Cont_c^{\infty}(\F^{\times})$. Fixing some $A > \sigma_{\pi}$ we have the Plancherel formula for $\F^{\times}$
\begin{equation} \label{eq: DisPl}
	\int_{\F^{\times}} f(t) \varphi(t) \ud^{\times} t = \sideset{}{_{\xi \in \widehat{\Omega_{\F}}}} \sum \int_{\widehat{R_{\F}}} \Mellin{f}(A+i\tau,\xi) \cdot \Mellin{\varphi}(-A-i\tau,\xi^{-1}) \ud \mu_{\F}(i\tau) 
\end{equation}
	where $\mu_{\F}$ is the dual/Plancherel (Haar) measure. Note that the absolute convergence on the right hand side of \eqref{eq: DisPl} is ensured by the ``finite support in $\widehat{\Omega_{\F}}$'' property in Definition \ref{def: MFNA}, resp. rapid decay property in $(s,p)$ in Definition \ref{def: MFA} for $\MF_{\pi}$, which obviously also hold true for $\Mellin{\varphi}$. Now that $\Mellin{f}(s,\xi) = 0$, we get $\int_{\F^{\times}} f(t) \varphi(t) \ud^{\times} t = 0$ for any $\varphi \in \Cont_c^{\infty}(\F^{\times})$. Therefore we deduce $f = 0$ since $\Cont^{\infty}(\F^{\times})$ embeds injectively into $\mathfrak{D}'(\F^{\times})$ (the dual of $\Cont_c^{\infty}(\F^{\times})$), proving the injectivity. Moreover, we note that a simple argument based on \eqref{eq: DisPl} yields Theorem \ref{thm: EquivDefs} (1) from Theorem \ref{thm: PW}, which is precisely what Jiang--Luo and Jiang--Li intended to achieve. Finally, we note that the inclusion in $\MF_{\pi}$ of the image of $S_{\pi}(\F^{\times})$ (or $\VorH(\pi)$) under $\Mellin{\cdot}$ is also more or less included in the standard theories mentioned above. In fact, only the rapid decay in $p \in \Z$ for the case $\F = \C$ may not have been explicitly written in the literature. We leave this simple fact as an exercise to the reader.
	
	The proof of surjectivity in the non-archimedean case is the easiest one. We sketch it as follows. By the standard theories, we already have some \emph{local} surjectivity. Namely for any given $\xi \in \widehat{\Omega_{\F}}$ the space of $X \mapsto \Mellin{f}(X,\xi)$ as $f$ traverses $S_{\pi}(\F^{\times})$ (or $\VorH(\pi)$) is precisely $L(s, \pi \times \xi) \cdot \C[X,X^{-1}]$ by the definition of $L(s, \pi \times \xi) = L(s, \pi \otimes \xi)$ given in \cite[Theorem 3.3]{GoJa72} (or \cite[Theorem (2.7)]{JPS83}). Hence for any $\xi \in \widehat{\Omega_{\F}}$ we can find $e_{\xi} \in S_{\pi}(\F^{\times})$ (or $\VorH(\pi)$) with 
\begin{equation} \label{eq: LocTestVec}
	\Mellin{e_{\xi}}(X,\xi) = L(s, \pi \times \xi).
\end{equation}
	With little effort we can also assume $\Mellin{e_{\xi}}(\cdot, \xi') = 0$ for $\xi' \neq \xi$; or equivalently $e_{\xi}(t u) = \xi^{-1}(u) e_{\xi}(t)$ for any $t \in \F^{\times}$ and $u \in \Omega_{\F}$. Given $\mathfrak{z} \in \MF_{\pi}$, we have a finite subset $S \subset \widehat{\Omega_{\F}}$ with $\xi \notin S \Rightarrow \mathfrak{z}(\cdot, \xi) = 0$, and for any $\xi \in S$ a polynomial $h_{\xi} \in \C[X,X^{-1}]$ such that $\mathfrak{z}(X, \xi) = h_{\xi}(X) L(s, \pi \times \xi)$. Then we observe
\begin{equation} \label{eq: SchIsoTypD}
	\Sch(\F^{\times}) = \sideset{}{_{\xi \in \widehat{\Omega_{\F}}}} \bigoplus \Sch(\F^{\times}; \xi), \quad \Sch(\F^{\times}; \xi) := \left\{ \varphi \in \Sch(\F^{\times}) \ \middle| \ \varphi(tu) = \xi(u) \varphi(t), \ \forall t \in \F^{\times}, u \in \Omega_{\F} \right\}.
\end{equation}
	Since the space of $\Mellin{\varphi}(\cdot, \xi')$ as $\varphi$ traverses $\Sch(\F^{\times}; \xi^{-1})$ is either $\C[X,X^{-1}]$ if $\xi' = \xi$, or $0$ if $\xi' \neq \xi$, we find for each $\xi \in S$ a function $\varphi_{\xi} \in \Sch(\F^{\times}; \xi^{-1})$ with $\Mellin{\varphi_{\xi}}(X, \xi) = h_{\xi}(X)$. Therefore the function $f := \sideset{}{_{\xi \in S}} \sum \varphi_{\xi} * e_{\xi} = \left( \sideset{}{_{\xi \in S}} \sum \varphi_{\xi} \right) * \left( \sideset{}{_{\xi \in S}} \sum e_{\xi} \right)$ satisfies $\Mellin{f} = \mathfrak{z}$.
	
	In the above argument, finding $e_{\xi}$ (as simple as possible) satisfying \eqref{eq: LocTestVec} is called the \emph{test function problem} in the literature (e.g. \cite{Is19, Hum19, Hum25}). Passing to $e_S := \sideset{}{_{\xi \in S}} \sum e_{\xi}$ can be thought as the \emph{global test function problem}. In the archimedean (especially complex) case, this problem has new features:
\begin{itemize}
	\item[(1)] The support of a general $\mathfrak{z} \in \MF_{\pi}$ in $\widehat{\Omega_{\C}}$ need not be finite;
	\item[(2)] The function $\mathfrak{z}_0(s,p) := L(s, \pi \times \ac^p)$ does not belong to $\mathcal{L}_{\pi}$ due to lack of rapid decay in $p$.
\end{itemize}
	We remedy this by introducing some rapidly decreasing scalars $e_p > 0$, and prove for some $A \in S_{\pi}$
\begin{equation} \label{eq: GlobTestVecC}
	\Mellin{A}(s,p) = e_p L(s, \pi \times \ac^p).
\end{equation}
	We then modify the subsequent argument by showing 
\begin{equation} \label{eq: DualCP}
	\MF_{\pi} \subset \Mellin{\Sch(\F^{\times})} \cdot \Mellin{A} + \Mellin{\Sch(\F^{\times})}.
\end{equation}
	The space $\Mellin{\Sch(\F^{\times})}$ is easily characterized as taking $L(s, \pi \times \xi) = 1$ in the definition of $\MF_{\pi}$. So the proof of \eqref{eq: DualCP} amounts to cancelling the poles of $s \mapsto \mathfrak{z}(\cdot, p)$ for all $p \in \Z$ and a given $\mathfrak{z} \in \MF_{\pi}$. We shall call it the \emph{cancellation of poles} (CP for brevity), and transform the problem to a version of Borel's lemma. To conclude, we need to show that the spaces $S_{\pi}(\F^{\times}), \VorH(\pi)$ are $\Sch(\F^{\times})$-modules under convolution, and they contain $\Sch(\F^{\times})$. These fundamental looking properties will be established in \S \ref{sec: FPs}. We will establish CP and give the detailed proof for $\F = \C$ in \S \ref{sec: SurjC}.
	
	We end the introductory part of the appendix with a concrete description of $S_{\pi}(\F^{\times}) = \VorH(\pi)$. Recall that a \emph{finite function} $\lambda: \F^{\times} \to \C$ is one, whose translates $(\tau_y \lambda)(x) := \lambda(xy)$ span a finite dimensional representation of $\F^{\times}$ (see \cite[\S 4.3.1]{MV10}). Concretely they are of the form $\lambda(x) = \chi(x) (\log \norm[x]_{\F})^n$ for some quasi-character $\chi$ of $\F^{\times}$, and some $n \in \Z_{\geq 0}$. If $\F \in \{ \R, \C \}$, we say $\lambda_1$ is ``smaller'' than $\lambda_2$ for $\lambda_j(x) = \xi_j(x) \norm[x]_{\F}^{\mu_j} (\log \norm[x]_{\F})^{n_j}$, if $n_1 = n_2$ and $\xi_2\xi_1^{-1}(x) \norm[x]_{\F}^{\mu_2-\mu_1}$ is a polynomial/monomial in $x$ and $\overline{x}$.
\begin{definition} \label{def: ExpSet}
	We define the \emph{exponent set} $\Ex(\pi)$ of the representation $\pi$ to be the set of ``smallest'' finite functions $\lambda(x) = \xi(x) \norm[x]_{\F}^{\mu} (\log \norm[x]_{\F})^n$, where $\xi \in \widehat{\Omega_{\F}}$ extended to $\F^{\times}$ by the trivial character on $R_{\F}$, $n \in \Z_{\geq 0}$ are those for which the function $L(s, \pi \times \xi^{-1})$ has a pole at $s=-\mu$ of order $> n$.
\end{definition}

\begin{corollary} \label{cor: ConSp}
	The space $S_{\pi}(\F^{\times}) = \VorH(\pi)$ is precisely $\sideset{}{_{\lambda \in \Ex(\pi)}} \sum \lambda \cdot \Sch(\F)$.
\end{corollary}
\begin{proof}
	The fact that $\VorH(\pi) = \VorH(\pi, 0) \subset V := \sideset{}{_{\lambda \in \Ex(\pi)}} \sum \lambda \cdot \Sch(\F)$ is a direct consequence of the description of Whittaker functions on the diagonal (e.g. \cite[Proposition 6.1 \& 8.1]{Cog03}), together with $\Mellin{\VorH(\pi)} \subset \MF_{\pi}$. Since we have $\Mellin{V} \subset \MF_{\pi}$ by Tate's local theory, Theorem \ref{thm: PW} forces the equality $\Mellin{\VorH(\pi)} = \Mellin{V} = \MF_{\pi}$, hence $\VorH(\pi) = V$ to hold.
\end{proof}

	\subsection{Some Fundamental Properties}
	\label{sec: FPs}
	
\begin{proposition} \label{prop: StructureWtSch}
	Let $\F \in \{ \R, \C \}$. Let $\pi$ be irreducible Casselman--Wallach of $\GL_n(\F)$.

\noindent (1) The space $S_{\pi}(\F^{\times})$ is stable by convolution with $\Sch(\F^{\times})$. We have $\Sch(\F^{\times}) \subset S_{\pi}(\F^{\times})$.

\noindent (2) For generic $\pi$, the space $\VorH(\pi)$ is stable by convolution with $\Sch(\F^{\times})$. We have $\Sch(\F^{\times}) \subset \VorH(\pi)$.
\end{proposition}
\begin{proof}
	(1, Stability) Take $\varphi \in \Sch(\F^{\times})$ and a generator of $S_{\pi}(\F^{\times})$ as
	$$ f(x) = \norm[x]_{\F}^{\frac{n-1}{2}} \int_{\SL_n(\F)} \left( \Phi \cdot \beta \right) \left( \begin{pmatrix} x & \\ & \id_{n-1} \end{pmatrix} h \right) \ud h $$
	for some $\Phi \in \Sch(\Mat_n(\F))$ and $\beta \in C(\pi^{\infty})$. If $\F = \C$ or $2 \nmid n$, then the group homomorphism $\F^{\times} \to \F^{\times}, t \mapsto t^n$ is surjective with finite kernel $\mu_n(\F)$. We have
\begin{multline*}
	(\varphi*f)(x) = \int_{\F^{\times}} \varphi(t) f(t^{-1}x) \ud^{\times} t = \frac{\norm[n]_{\F}}{\norm[\mu_n(\F)]} \int_{\F^{\times}} \varphi(t^n) f(t^{-n}x) \ud^{\times} t \\
	= \frac{\norm[n]_{\F}}{\norm[\mu_n(\F)]} \norm[x]_{\F}^{\frac{n-1}{2}} \int_{\SL_n(\F)} \left\{ \int_{\F^{\times}} \norm[t]_{\F}^{-\frac{n(n-1)}{2}} \varphi(t^n) \omega(t)^{-1} \Phi \left( t^{-1} \begin{pmatrix} x & \\ & \id_{n-1} \end{pmatrix} h \right) \ud^{\times} t \right\} \cdot \beta \left( \begin{pmatrix} x & \\ & \id_{n-1} \end{pmatrix} h \right) \ud h,
\end{multline*}
	where $\omega$ is the central character of $\pi$. The function $t \mapsto \norm[t]_{\F}^{-\frac{n(n-1)}{2}} \varphi(t^n) \omega(t)^{-1}$ lying in $\Sch(\F^{\times})$, the inner integral is equal to $\Psi \left( \begin{pmatrix} x & \\ & \id_{n-1} \end{pmatrix} h \right)$ for some $\Psi \in \Sch(\Mat_n(\F))$, since $\Sch(\Mat_n(\F))$ is a smooth Fr\'echet representation of moderate growth of $\GL_n(\F)$. If $\F = \R$ and $2 \mid n$, we need to consider
	$$ (\varphi*f)(x) = n \int_0^{\infty} \varphi(t^n) f(t^{-n}x) \ud^{\times} t + n \int_0^{\infty} \varphi(-t^n) f(-t^{-n}x) \ud^{\times} t. $$
	Similar argument shows that $\varphi*f$ is the sum of two generators in $S_{\pi}(\F^{\times})$, proving $\varphi * f \in S_{\pi}(\F^{\times})$.
	
\noindent (1, Inclusion) Take $\varphi \in \Sch(\F^{\times})$. We note the following smooth coordinates
\begin{equation} \label{eq: SmCorGLn}
	\R_{>0} \times \Omega_{\F} \times \SL_n(\F) \to \GL_n(\F), \quad (t, \omega, h) \mapsto t a(\omega) h
\end{equation}
	where we have written $a(\omega) := \diag(\omega, 1, \dots, 1)$. Note that any $\beta \in C(\pi^{\infty})$ satisfies $\beta(t a(\omega) h) = \omega(t) \beta(a(\omega) h)$ if $\omega$ denotes the central character of $\pi$. For any $\omega_0 \in \Omega_{\F}$, we can find a $\beta_{\omega_0} \in C(\pi^{\infty})$ with $\beta_{\omega_0}(a(\omega_0)) \neq 0$. We choose $\rho_{\omega_0} \in \Cont_c^{\infty}(\SL_n(\F))$ with support sufficiently close to $\id$ so that
	$$ c_{\omega_0}(\omega) := \int_{\SL_n(\F)} \rho_{\omega_0}(h) \beta_{\omega_0}(a(\omega) h) \ud h \neq 0, \quad \forall \ \omega \in U_{\omega_0} $$
for an open neighborhood $U_{\omega_0}$ of $\omega_0$. By compactness of $\Omega_{\F}$, we can find finitely many such quadruples $(U_{\omega_0}, \rho_{\omega_0}, \beta_{\omega_0}, c_{\omega_0})$, renumbered as $(U_j, \rho_j, \beta_j, c_j)$, with $\Omega_{\F} = \sideset{}{_j} \bigcup U_j$. Take a smooth partition of unity
	$$ 1 = \sideset{}{_j} \sum \eta_j(\omega), \quad \mathrm{supp}(\eta_j) \subset U_j. $$
	Write $\varphi_0(x) := \norm[x]_{\F}^{-\frac{n-1}{2}} \varphi(x) \in \Sch(\F^{\times})$. The following function defined via the coordinates \eqref{eq: SmCorGLn}
	$$ \Phi_j(t a(\omega) h) := \varphi_0(t^n \omega) \omega(t)^{-1} \frac{\eta_j(\omega)}{c_j(\omega)} \rho_j(h) $$
	lies in $\Sch(\R_{>0}) \widehat{\otimes} \Cont^{\infty}(\Omega_{\F}) \widehat{\otimes} \Cont_c^{\infty}(\SL_n(\F)) \subset \Sch(\GL_n(\F)) \subset \Sch(\Mat_n(\F))$. For $x \in \F^{\times}$ writing $x = t^n \omega$ we have
\begin{multline*}
	S_{\pi}(\F^{\times}) \ni f_j(x) := \norm[x]_{\F}^{\frac{n-1}{2}} \int_{\SL_n(\F)} \Phi_j(t a(\omega) h) \beta_j(t a(\omega) h) \ud h \\
	= \norm[x]_{\F}^{\frac{n-1}{2}} \int_{\SL_n(\F)} \varphi_0(x) \omega(t)^{-1} \frac{\eta_j(\omega)}{c_j(\omega)} \rho_j(h) \cdot \omega(t) \beta_j(a(\omega) h) \ud h = \norm[x]_{\F}^{\frac{n-1}{2}} \varphi_0(x) \eta_j(\omega) = \varphi(x) \eta_j(\omega).
\end{multline*}
	Summing over $j$ shows $\varphi = \sideset{}{_j} \sum f_j \in S_{\pi}(\F^{\times})$.
	
\noindent (2, Stability) Since $\pi$ is Casselman--Wallach for $\GL_n(\F)$, it is Casselman--Wallach for $\F^{\times}$, embedded into $\GL_n(\F)$ via the $a(\cdot)$ map above. Regarding $\VorH(\pi)$ as $\VorH(\pi, 0)$ we see the desired stability by $\Sch(\F^{\times})$.

\noindent (2, Inclusion) The case $n=1$ follows by definition. Suppose $n \geq 2$. Take $\varphi \in \Sch(\F^{\times})$. Choose $0 \neq h \in \VorH(\pi, 0)$ associated with $W \in \Whi(\pi^{\infty}, \psi)$ by $h(x) = W(\mathrm{diag}(x, 1, \dots, 1))$. Take $f \in \Cont_c^{\infty}(\F^{\times})$ with
	$$ \int_{\F^{\times}} f(x) h(x) \ud^{\times} x = 1. $$
	The function $\F^{\times} \times \F, (x,y) \mapsto \varphi(y) f(xy)$ lies in $\Sch(\F^{\times} \times \F)$. Let $\Phi \in \Sch(\F^{\times} \times \F)$ be defined by
	$$ \int_{\F} \Phi(x,t) \psi(yt) \ud t = \varphi(y) f(xy). $$
	Now that $\pi$ is Casselman--Wallach for $\GL_n(\F)$, it is so for any subgroup. In particular, we have
	$$ W_1 := \int_{\F^{\times} \times \F} \Phi(u,v) \cdot \pi \begin{pmatrix} u & v & \\ & 1 & \\ & & \id_{n-2} \end{pmatrix}.W \ud^{\times} u \ud v \in \Whi(\pi^{\infty}, \psi). $$
	The associated function $h_1 \in \VorH(\pi, 0)$ is given by
	$$ h_1(x) = \int_{\F^{\times} \times \F} \Phi(u,v) \cdot \psi(xv) h(xu) \ud^{\times} u \ud v = \phi(x) \int_{\F^{\times}} f(xu) h(xu) \ud^{\times} u = \phi(x) \in \VorH(\pi). $$
\end{proof}

\begin{remark} \label{rmk: StOmegaF}
	Replacing $\varphi \in \Sch(\F^{\times})$ by $\varphi \in \Cont^{\infty}(\Omega_{\F})$, the above argument for ``stability'' still works. This is obvious for $\VorH(\pi)$. For $S_{\pi}(\F^{\times})$ it is easy for $\F=\R$. For $\F=\C$, using the surjectivity of the $n$-th power map on $\Omega_{\C}$ we have
\begin{multline*}
	\int_{\Omega_{\C}} \int_{\SL_n(\F)} \left( \Phi \cdot \beta \right) \left( \begin{pmatrix} t x & \\ & \id_{n-1} \end{pmatrix} h \right) \varphi(t) \ud h \ud t \\
	= \int_{\SL_n(\F)} \left\{ \int_{\Omega_{\C}} \Phi \left( t \begin{pmatrix} x & \\ & \id_{n-1} \end{pmatrix} h \right) \omega(t) \varphi(t^n) \ud t \right\} \beta \left( \begin{pmatrix} x & \\ & \id_{n-1} \end{pmatrix} h \right) \ud h.
\end{multline*}
	We are reduced to proving $\Phi_1 \in \Sch(\Mat_n(\C))$ for $\Phi_1(X) := \int_{\Omega_{\C}} \Phi(tX) \omega(t) \varphi(t^n) \ud t$, which is obvious.
\end{remark}

\begin{proposition} \label{prop: PWforSch}
	Let $\F \in \{ \R, \C \}$. The Mellin transform of any function in $\Sch(\F^{\times})$ is absolutely convergent at any point of $\QC(\F^{\times})$. It establishes an isomorphism between $\Sch(\F^{\times})$ and the space $\mathcal{M}(\QC(\F^{\times}))$ consisting of functions $\mathfrak{z}$ on $\QC(\F^{\times})$ satisfying:
\begin{itemize}
	\item[(1)] For any $p \in \widehat{\Omega_{\F}}$, the function $s \mapsto \mathfrak{z}(s,p)$ is entire in $\C$;
	\item[(2)] For any $a < b$ and any polynomial $P \in \C[X, Y]$, we have 
	$$ \sideset{}{_{p \in \widehat{\Omega_{\F}}}} \sup \sideset{}{_{a \leq \Re(s) \leq b}} \sup \extnorm{P(s,p) \mathfrak{z}(s, p)} < + \infty. $$
\end{itemize}
\end{proposition}
\begin{proof}
	It suffices to note that $\Sch(\F^{\times}) = \Sch(\R_{>0}) \widehat{\otimes} \Cont^{\infty}(\Omega_{\F})$. The usual Paley--Wiener theorem for $\Sch(\R_{>0})$ (e.g. \cite[Proposition 5.13]{Wu19_TF}) together with Fourier analysis on $\Omega_{\F}$ readily give the desired isomorphism.
\end{proof}

	\subsection{Surjectivity for $\F = \C$}
	\label{sec: SurjC}
	
	We first focus on the case of $S_{\pi}(\C^{\times})$, and start the discussion on the global test function problem \eqref{eq: GlobTestVecC}. By the solution of the usual test function problem (e.g. \cite[Theorem 8.7 (3)]{GoJa72}), for any $p \in \Z$ there is $A_p \in S_{\pi}(\F^{\times})$ with $\Mellin{A_p}(s,p) = L(s, \pi \times \ac^p)$. Replacing $A_p$ by $x \mapsto \int_{\Omega_{\C}} A_p(tx) t^p \ud t$ if necessary, and taking into account Remark \ref{rmk: StOmegaF}, we may assume $\Mellin{A_p}(s,q) = 0$ for any $\Z \ni q \neq p$. Observing the general form in terms of the Langlands parameters $(\lambda_i, l_i) \in \C \times \Z$ of
\begin{equation} \label{eq: LViaLan}
	L(s, \pi \times \ac^p) = \sideset{}{_{i=1}^n} \prod \Gamma_{\C} \left( s + \lambda_i + \frac{\norm[p+l_i]}{2} \right) \quad \text{with} \quad \Gamma_{\C}(s) := 2(2\pi)^{-s} \Gamma(s), 
\end{equation}
	we find the following relation with $p_0 = p_0(\pi) := \max_i \norm[l_i]$
\begin{equation} \label{eq: LShiftRel}
	L(s, \pi \times \ac^p) = L \left( s + \frac{\norm[p]-p_0}{2}, \pi \times \ac^{\sgn(p)p_0} \right) =
	\begin{cases}
		\Mellin{\overline{\bold{a}}^{p-p_0}A_{p_0}}(s, p) & \text{if } p \geq p_0 \\
		\Mellin{\bold{a}^{-p-p_0}A_{-p_0}}(s, p) & \text{if } p \leq -p_0
	\end{cases},
\end{equation}
	where $\bold{a}(x) := x$. By definition there are finitely many $\Phi_i^{\pm} \in \Sch(\Mat_n(\C))$ and $\beta_i^{\pm} \in C(\pi^{\infty})$ such that
\begin{equation} \label{eq: DataAtp0}
	A_{\pm p_0}(x) = \norm[x]_{\F}^{\frac{n-1}{2}} \cdot \sideset{}{_i} \sum \int_{\SL_n(\F)} \left( \Phi_i^{\pm} \cdot \beta_i^{\pm} \right) \left( \begin{pmatrix} x & \\ & \id_{n-1} \end{pmatrix} h \right) \ud h.
\end{equation}
	Our solution of the global test function problem relies on the following result.
\begin{lemma} \label{lem: SchExist}
	Given $\Phi_1, \dots, \Phi_r \in \Sch(\Mat_n(\C))$, we can find $e_p > 0$ so that for each $i$
	$$ \Psi_i(X) := \sideset{}{_{p \in \Z_{\geq 0}}} \sum e_p \cdot \det(X)^p \Phi_i(X) $$
defines a function $\Psi_i \in \Sch(\Mat_n(\C))$. The same holds true if we replace $\det(X)$ by $\overline{\det(X)}$.
\end{lemma}
\begin{proof}
	We parametrize a system of seminorms on $\Sch(\Mat_n(\C))$ by $N \in \Z_{\geq 0}$ and $\alpha \in \Z_{\geq 0}^{n^2}$ as
	$$ p_{N, \alpha}(f) := \sup_{X \in \Mat_n(\C)} \left( 1+\Norm[X] \right)^N \extnorm{\partial^{\alpha}f(X)}, \quad f \in \Sch(\Mat_n(\C)). $$
	Write $\norm[\alpha]$ for the sum of the components of $\alpha$. Define
	$$ e_p := 2^{-p} \left( 1 + \sideset{}{_{i=1}^r} \sum \sideset{}{_{\substack{0 \leq N \leq p \\ \norm[\alpha] \leq p}}} \sum p_{N,\alpha} \left( \sideset{}{^p} \det \cdot \Phi_i \right) \right)^{-1}. $$
	Given any $(N, \alpha)$ and for $p \geq \max(N, \norm[\alpha])$ we have
	$$ p_{N, \alpha}(e_p \cdot \sideset{}{^p} \det \cdot \Phi_i) = e_p \cdot p_{N, \alpha}(\sideset{}{^p} \det \cdot \Phi_i) \leq 2^{-p}. $$
	Therefore the series defining $\Psi_i$ absolutely converges for $p_{N, \alpha}$, proving $\Psi_i \in \Sch(\Mat_n(\C))$.
\end{proof}

\noindent Applying Lemma \ref{lem: SchExist} to $\Phi_i^{\pm}$ given by \eqref{eq: DataAtp0}, we find $e_p$ for $\norm[p] \geq p_0$ and $G_i^{\pm} \in \Sch(\Mat_n(\C))$ by
	$$ G_i^+ := \sideset{}{_{p=p_0}^{\infty}} \sum e_p \cdot \overline{\det}^{p-p_0} \cdot \Phi_i^+, \quad G_i^- := \sideset{}{_{p=-\infty}^{-p_0}} \sum e_p \cdot \sideset{}{^{-p-p_0}} \det \cdot \Phi_i^- $$
	and construct $A_{\pm} \in S_{\pi}(\C^{\times})$ by
	$$ A_{\varepsilon}(x) := \norm[x]_{\F}^{\frac{n-1}{2}} \cdot \sideset{}{_i} \sum \int_{\SL_n(\F)} \left( G_i^{\varepsilon} \cdot \beta_i \right) \left( \begin{pmatrix} x & \\ & \id_{n-1} \end{pmatrix} h \right) \ud h = 
	\begin{cases}
		\sideset{}{_{p=p_0}^{\infty}} \sum e_p \cdot \overline{x}^{p-p_0} \cdot A_{p_0}(x) & \varepsilon = + \\
		\sideset{}{_{p=-\infty}^{-p_0}} \sum e_p \cdot x^{-p-p_0} \cdot A_{-p_0}(x) & \varepsilon = -
	\end{cases}. $$
	We take $e_p=1$ for $\norm[p] < p_0$, and solve the global test function problem \eqref{eq: GlobTestVecC} by
\begin{equation}
	A = A_+ + A_- + \sideset{}{_{\norm[p] < p_0}} \sum A_p.
\end{equation}

	We turn to the CP by \eqref{eq: DualCP}. For any $\mathfrak{z} \in \MF_{\pi}$, we shall construct a ``nice'' $f \in \Cont^{\infty}(\C^{\times})$ satisfying
\begin{itemize}
	\item[(\rmnum{1})] $\Mellin{f}(s, p)$ is absolutely convergent in $\Re(s) > \sigma_{\pi}$, and has meromorphic continuation to $s \in \C$;
	\item[(\rmnum{2})] $\Mellin{f}(s, p) - e_p^{-1} e^{-s^2} \mathfrak{z}(s,p)$ is entire for every $p \in \Z$;
	\item[(\rmnum{3})] For any $a < b$ and any polynomial $P \in \C[X, Y]$ such that $s \mapsto P(s,p) L(s, \pi \times \ac^p)$ is holomorphic in $a \leq \Re(s) \leq b$ for all $p$, we have $\sup_{p \in \Z} \sup_{a \leq \Re(s) \leq b} \extnorm{P(s,p) \Mellin{f}(s, p)} < + \infty$.
\end{itemize}
	Write $L(s, \pi \times \ac^p)$ in the form \eqref{eq: LViaLan}. Reducing the number of $\lambda_i$ we may assume $\lambda_i - \lambda_j \notin 2^{-1}\Z$ if $1 \leq i \neq j \leq n_{\pi}$ for some integer $n_{\pi} \leq n$. The poles of $s \mapsto e_p^{-1} e^{-s^2} \mathfrak{z}(s,p)$ are of the form $-\lambda_i-\norm[p+l_i]/2-k$ for $k \in \Z_{\geq 0}$ and $1 \leq i \leq n_{\pi}$, with order $\leq n$. Write its principal part as
\begin{equation} \label{eq: CPTargetPP}
	e_p^{-1} e^{-s^2} \mathfrak{z}(s,p) = \sum_{m=1}^n \frac{b_{i,k,m,p}}{\left( s + \lambda_i + \norm[p+l_i]/2 +k \right)^m} + O(1).
\end{equation}
	We look for $f$ of the following form for some $f_{i,m} \in \Cont_c^{\infty}(\Omega_{\F} \times \R)$
\begin{equation} \label{eq: CPFuncF}
	f(x) = \sum_{i=1}^{n_{\pi}} \sum_{m=1}^n f_{i,m}(x/\norm[x], \norm[x]) \cdot \left( x/\norm[x] \right)^{l_i} \norm[x]_{\C}^{\lambda_i} \left( 2 \log \norm[x] \right)^{m-1}.
\end{equation}
	Its Mellin transform $\Mellin{f}(s,p)$ satisfies (\rmnum{1}) \& (\rmnum{3}) by standard theory/computation. Its poles are of the form $-\lambda_j-d/2$ for $d \in \Z_{\geq 0}$ with principal part
\begin{multline*}
	\Mellin{f}(s,p) = \sum_{j=1}^{n_{\pi}} \sum_{m=1}^n \int_0^{\infty} \left\{ \int_0^{2\pi} f_{j,m}(e^{i\theta}, t) e^{i(p+l_j)\theta} \ud \theta \right\} \cdot t^{2(s+\lambda_j)} (2 \log t)^{m-1} \frac{\ud t}{t} \\
	= \sum_{m=1}^n \left\{ \int_0^{2\pi} \left. \frac{\partial^d}{\partial t^d} \right| f_{j,m}(e^{i\theta}, 0) e^{i(p+l_j)\theta} \ud \theta \right\} \cdot \frac{(-1)^{m-1}(m-1)!}{2 \cdot d!} \cdot \frac{1}{\left( s+\lambda_j+d/2 \right)^m} + O(1).
\end{multline*}
	In order for $\Mellin{f}$ to satisfy (\rmnum{2}), it suffices that
\begin{equation} \label{eq: CPDerCond}
	\left. \frac{\partial^d}{\partial t^d} \right| f_{j,m}(u, 0) = \frac{(-1)^{m-1} d!}{(m-1)! \cdot \pi} u^{-l_j} \cdot \sideset{}{_{\substack{p \in \Z, k \in \Z_{\geq 0} \\ \norm[p+l_j] + 2k = d}}} \sum b_{j,k,m,p} u^{-p}.
\end{equation}
	The existence of such $f_{j,m}$ is guaranteed by the following variant of Borel's lemma \cite[\S 4 Lemma 2]{J04}:
\begin{lemma}[Borel's Lemma]
	Suppose $\{f_k\}_{k\geqslant0}$ is a sequence of bounded holomorphic
functions on a connected open set $\Omega\subset\mathbb{C}^d$.
For every $\varepsilon>0$, there exists a smooth function $f(u,t)$ on $\Omega\times(-\varepsilon,\varepsilon)$ 
that is holomorphic in $u$ and satisfies
\[
\frac{\partial^j f}{\partial t^{j}}(u,0)=f_j(u)
\]
for all $j\geqslant 0$. Moreover, the support of $f(u,t)$ has a compact projection onto the second factor.
\end{lemma}

\noindent With this ``nice'' $f$ we define
	$$ m(s,p) := e^{s^2} \frac{\Mellin{f}(s,p)}{L(s, \pi \times \ac^p)}. $$
	It is entire for any $p$ and the principal part of
	$$ m(s,p) \cdot e_p L(s, \pi \times \ac^p) = e_p e^{s^2} \Mellin{f}(s, p) $$
cancels the principal part of $\mathfrak{z}(s,p)$ at any pole. We justify $m \in \Mellin{\Sch(\F^{\times})}$ via Proposition \ref{prop: PWforSch}, and conclude the proof of \eqref{eq: DualCP} by the following simple estimate, whose proof is left to the reader.

\begin{lemma} \label{lem: LFuncEst}
	For any $\pi$ irreducible Casselman--Wallach of $\GL_n(\C)$ and any $a < b$ we have
	$$ \sideset{}{_{p \in \Z}} \sup \sideset{}{_{a \leq \Re(s) \leq b}} \sup \extnorm{e^{s^2} L(s, \pi \times \ac^p)^{-1}} < \infty.  $$
\end{lemma}

	Finally we briefly indicate the necessary modifications for $\VorH(\pi)$. The CP by \eqref{eq: DualCP} being solved in the same way, we only need to solve the global test function problem by \eqref{eq: GlobTestVecC}. The relevant usual test function problem is solved by \cite[Theorem 2.6 (\rmnum{1})]{J09}. To pass to the global version, we use (completeness of) the Fr\'echet space $\Whi(\pi^{\infty}, \psi)$ instead of $\Sch(\Mat_n(\C))$. For $W \in \Whi(\pi^{\infty}, \psi)$ multiplication by $x$ or $\overline{x}$ to $W(a(x))$ (with $a(x) := \mathrm{diag}(x, 1, \dots, 1)$) is realized by $(X.W)(a(x)) = x W(a(x))$ or $(Y.W)(a(x)) = \overline{x} W(a(x))$ for some $X,Y$ in the Lie algebra of $\GL_n(\C)$. If $A_{\pm p_0}(x) = \norm[x]_{\C}^{\frac{1-n}{2}} W_{\pm}(a(x))$, we then form $A_+ = \norm_{\C}^{\frac{1-n}{2}} \cdot \sideset{}{_{p \geq p_0}} \sum e_p \cdot (Y^{p-p_0}.W_+)$ and $A_- = \norm_{\C}^{\frac{1-n}{2}} \cdot \sideset{}{_{p \leq -p_0}} \sum e_p \cdot (X^{-p-p_0}.W_-)$ for suitable $e_p > 0$.

%

\begin{thebibliography}{10}

\bibitem{BS11}
{\sc Baruch, E.~M., and Snitz, K.}
\newblock A note on {Bessel} functions for supercuspidal representations of
  {$\mathrm{GL}(2)$} over a {$p$-}adic field.
\newblock {\em Algebra Colloquium 18}, Spec 1 (2011), 733--738.

\bibitem{Br08}
{\sc Brenner, E.}
\newblock Stability of the local gamma factor in the unitary case.
\newblock {\em J. Number Theory 128\/} (2008), 1358--1375.

\bibitem{BrM03}
{\sc Bruggeman, R.~W., and Motohashi, Y.}
\newblock Sum formula for {K}loosterman sums and fourth moment of the
  {D}edekind zeta-function over the {G}aussian number field.
\newblock {\em Funct. Approx. Comment. Math. \Rmnum{31}\/} (2003), 23--92.

\bibitem{Bu87}
{\sc Bushnell, C.~J.}
\newblock Hereditary orders, {Gauss} sums and supercuspidal representations of
  $\mathrm{GL}_n$.
\newblock {\em J. Reine Angew. Math. 375-376\/} (1987), 184--210.

\bibitem{Bu17}
{\sc Bushnell, C.~J.}
\newblock {\em Arithmetic of cuspidal representations}.
\newblock King's College London, December 2017.

\bibitem{BuF85}
{\sc Bushnell, C.~J., and Fr\"ohlich, A.}
\newblock Non-abelian congruence {Gauss} sums and {$p$-adic} simple algebras.
\newblock {\em Proc. London Math. Soc. 3}, 50 (1985), 207--264.

\bibitem{BuH05_T1}
{\sc Bushnell, C.~J., and Henniart, G.}
\newblock The essentially tame local {Langlands} correspondence, {I}.
\newblock {\em J. Amer. Math. Soc. 18}, 3 (2005), 685--710.

\bibitem{BuH05_T2}
{\sc Bushnell, C.~J., and Henniart, G.}
\newblock The essentially tame local {Langlands} correspondence, {II}: totally
  ramified representations.
\newblock {\em Compos. Math. 141\/} (2005), 979--1011.

\bibitem{BuH06}
{\sc Bushnell, C.~J., and Henniart, G.}
\newblock {\em The Local Langlands Conjecture for GL(2)}.
\newblock No.~335 in Grundlehren der mathematischen Wissenschaften.
  Springer-Verlag, 2006.

\bibitem{BuH10_T3}
{\sc Bushnell, C.~J., and Henniart, G.}
\newblock The essentially tame local {Langlands} correspondence, {III}: the
  general case.
\newblock {\em Proc. London Math. Soc. 101\/} (2010), 497--553.

\bibitem{BK93}
{\sc Bushnell, C.~J., and Kutzko, P.}
\newblock {\em The admissible dual of {$\mathrm{GL}(N)$} via compact open
  subgroups}, vol.~129 of {\em Annals of Mathematics Studies}.
\newblock Princeton University Press, 1993.

\bibitem{Cog03}
{\sc Cogdell, J.~W.}
\newblock Lectures on {$L$-}functions, converse theorems, and functoriality for
  {$\mathrm{GL}_n$}, 2003.

\bibitem{CPS98}
{\sc Cogdell, J.~W., and Piatetski-Shapiro, I.~I.}
\newblock Stability of gamma factors for {$\mathrm{SO}(2n+1)$}.
\newblock {\em Manuscripta Math. 95\/} (1998), 437--461.

\bibitem{CHH88}
{\sc Cowling, M., Haargerup, U., and Howe, R.}
\newblock Almost {$L^2$} matrix coefficients.
\newblock {\em J. Reine Angew. Math.}, 387 (1988), 97--110.

\bibitem{GoJa72}
{\sc Godement, R., and Jacquet, H.}
\newblock {\em Zeta Functions of Simple Algebras}.
\newblock No.~260 in Lecture Notes in Mathematics. Springer-Verlag, 1972.

\bibitem{GoJ11}
{\sc Goldfeld, D., and Hundley, J.}
\newblock {\em Automorphic Representations and $L$-Functions for the General
  Linear Group}, vol.~129, 130 of {\em Cambridge studies in advanced
  mathematics}.
\newblock Cambridge University Press, 2011.

\bibitem{He09}
{\sc He, H.}
\newblock Bounds on smooth matrix coefficients on {$\mathrm{L}^2$-spaces}.
\newblock {\em Sel. Math. New Ser. 15}, 3 (2009), 419--433.

\bibitem{HT92}
{\sc Howe, R., and Tan, E.~C.}
\newblock {\em Non-abelian Harmonic Analysis: Applications of
  {$\mathrm{SL}(2,\mathbb{R})$}}.
\newblock Universitext. Springer-Verlag, 1992.

\bibitem{HNS18}
{\sc Hu, Y., Nelson, P., and Saha, A.}
\newblock Some analytic aspects of automorphic forms on {$\mathrm{GL}(2)$} of
  minimal type.
\newblock {\em Comment. Math. Helv. 94}, 4 (2019), 767--801.

\bibitem{Hum19}
{\sc Humphries, P.}
\newblock Test vectors for non-{Archimedean Godement-Jacquet} zeta integrals.
\newblock {\em Bull. London Math. Soc. 53\/} (2021), 92--99.

\bibitem{Hum25}
{\sc Humphries, P.}
\newblock Archimedean newform theory for {$\mathrm{GL}_n$}.
\newblock {\em J. Inst. Math. Jussieu 24}, 1 (2025), 41--116.

\bibitem{Is19}
{\sc Ishii, T.}
\newblock {Godement-Jacquet} integrals on $\mathrm{GL}(n, \mathbb{C})$.
\newblock {\em Ramanujan J. 49\/} (2019), 129--139.

\bibitem{J04}
{\sc Jacquet, H.}
\newblock Integral representation of {W}hittaker functions.
\newblock In {\em Contributions to Automorphic Forms, Geometry \& Number
  Theory}, H.~Hida, D.~Ramakrishnan, and F.~Shahidi, Eds. The Johns Hopkins
  University Press, 2004, ch.~15, pp.~373--419.

\bibitem{J09}
{\sc Jacquet, H.}
\newblock Archimedean {R}ankin-{S}elberg integrals.
\newblock In {\em Automorphic forms and $L$-functions \Rmnum{2}. Local
  aspects\/} (2009), vol.~489 of {\em Contemp. Math.}, Israel Math. Conf.
  Proc., Amer. Math. Soc., Providence, RI, pp.~57--172.

\bibitem{JPS79}
{\sc Jacquet, H., Piatetski-Shapiro, I.~I., and Shalika, J.}
\newblock Automorphic forms on {$\mathrm{GL}(3)$ \Rmnum{1}, \Rmnum{2}}.
\newblock {\em Ann. of Math. 109\/} (1979), 169--258.

\bibitem{JPS83}
{\sc Jacquet, H., Piatetskii-Shapiro, I.~I., and Shalika, J.}
\newblock {Rankin-Selberg} convolutions.
\newblock {\em Amer. J. Math. 105}, 2 (April 1983), 367--464.

\bibitem{JS83}
{\sc Jacquet, H., and Shalika, J.}
\newblock The {Whittaker} models of induced representations.
\newblock {\em Pacific J. Math. 109}, 1 (1983), 107--120.

\bibitem{JS85}
{\sc Jacquet, H., and Shalika, J.}
\newblock A lemma on highly ramified {$\varepsilon$-}factors.
\newblock {\em Math. Ann. 271\/} (1985), 319--332.

\bibitem{JS90}
{\sc Jacquet, H., and Shalika, J.}
\newblock {R}ankin-{S}elberg convolutions: {A}rchimedean theory.
\newblock In {\em Festschrift in honor of I. I. Piatetski-Shapiro on the
  occasion of his sixtieth birthday\/} (Jerusalem, 1990), vol.~Part I, Weizmann
  Science Press, pp.~125--207.

\bibitem{JLi25}
{\sc Jiang, D., and Li, Z.}
\newblock The {Voronoi} summation formula for {$\mathrm{GL}_n$} and the
  {Godement--Jacquet} kernels.
\newblock {\em Math. Z. 309}, 62 (2025), 1--45.

\bibitem{JiL22}
{\sc Jiang, D., and Luo, Z.}
\newblock Certain {Fourier} operators on {$\mathrm{GL}_1$} and local
  {Langlands} gamma functions.
\newblock {\em Pacific J. Math. 318}, 2 (2022), 339--374.

\bibitem{JiL23}
{\sc Jiang, D., and Luo, Z.}
\newblock Certain {Fourier} operators and their associated {Poisson} summation
  formulae on {$\mathrm{GL}_1$}.
\newblock {\em Pacific J. Math. 326}, 2 (2023), 301--372.

\bibitem{Ki00}
{\sc Kim, J.}
\newblock Gamma factors of certain supercuspidal representations.
\newblock {\em Math. Ann. 317\/} (2000), 751--781.

\bibitem{MV10}
{\sc Michel, P., and Venkatesh, A.}
\newblock The subconvexity problem for {$GL_2$}.
\newblock {\em Publ. Math. Inst. Hautes \'Etudes Sci. 111}, 1 (2010), 171--271.

\bibitem{MS04}
{\sc Miller, S.~D., and Schmid, W.}
\newblock Distributions and analytic continuation of {Dirichlet} series.
\newblock {\em J. Functional Analysis 214\/} (2004), 155--220.

\bibitem{MS06}
{\sc Miller, S.~D., and Schmid, W.}
\newblock Automorphic distributions, {$L$-}functions, and {Voronoi} summation
  for {$\mathrm{GL}(3)$}.
\newblock {\em Ann. of Math. 164\/} (2006), 423--488.

\bibitem{Oh02}
{\sc Oh, H.}
\newblock Uniform pointwise bounds for matrix coefficients of unitary
  representations and applications to {Kazhdan} constants.
\newblock {\em Duke Mathematical Journal 113\/} (2002), 133--192.

\bibitem{Qi20}
{\sc Qi, Z.}
\newblock Theory of fundamental {Bessel} functions of high rank.
\newblock {\em Mem. Amer. Math. Soc. 267}, 1303 (2020).

\bibitem{RS05}
{\sc Rallis, S., and Soudry, D.}
\newblock Stability of the local gamma factor arising from the doubling method.
\newblock {\em Math. Ann. 333\/} (2005), 291--313.

\bibitem{Ru91}
{\sc Rudin, W.}
\newblock {\em Functional Analysis}, 2nd~ed.
\newblock McGraw-Hill Book Company, 1991.

\bibitem{Shal00}
{\sc Shalom, Y.}
\newblock Rigidity, unitary representations of semisimple groups, and
  fundamental groups of manifolds with rank one transformation group.
\newblock {\em Ann. of Math. 152\/} (2000), 113--182.

\bibitem{Sun09}
{\sc Sun, B.}
\newblock Bounding matrix coefficients for smooth vectors of tempered
  representations.
\newblock {\em Proc. Amer. Math. Soc. 137}, 1 (2009), 353--357.

\bibitem{Ve10}
{\sc Venkatesh, A.}
\newblock Sparse equidistribution problems, period bounds and subconvexity.
\newblock {\em Ann. of Math. 172}, 2 (2010), 989--1094.

\bibitem{Wal88}
{\sc Wallach, N.~R.}
\newblock {\em Real Reductive Groups {\Rmnum{1}}}, vol.~132 of {\em Pure and
  applied mathematics}.
\newblock Academic Press, INC. Harcourt Brace Jovanovich, 1988.

\bibitem{Wal92}
{\sc Wallach, N.~R.}
\newblock {\em Real Reductive Groups {\Rmnum{2}}}, vol.~132-{\Rmnum{2}} of {\em
  Pure and applied mathematics}.
\newblock Academic Press, INC. Harcourt Brace Jovanovich, 1992.

\bibitem{We64}
{\sc Weil, A.}
\newblock Sur certains groupes d'op\'erateurs unitaires.
\newblock {\em Acta Mathematica 111}, 143-211 (1964).

\bibitem{Wu24+}
{\sc Wu, H.}
\newblock On a generalization of {Motohashi's} formula.
\newblock arXiv:2310.08236.

\bibitem{Wu19_TF}
{\sc Wu, H.}
\newblock Deducing {Selberg} trace formula via {Rankin-Selberg} method for
  {$\mathrm{GL}_2$}.
\newblock {\em Trans. Amer. Math. Soc. 372}, 12 (2019), 8507--8551.

\bibitem{Ze80}
{\sc Zelevinsky, A.}
\newblock Induced representations of reductive {$p$-adic} groups \rmnum{2}:
  {On} irreducible representations of {$\mathrm{GL}(n)$}.
\newblock {\em Ann. scient. \'Ec. Norm. Sup. 4}, 13 (1980), 165--210.

\end{thebibliography}

\end{document}